\documentclass{amsart}
\usepackage{fullpage,amssymb}
\usepackage{url}
\usepackage{graphicx}
\newcommand{\dbar}{{\overline{\partial}}}
\def\Xint#1{\mathchoice
  {\XXint\displaystyle\textstyle{#1}}%
  {\XXint\textstyle\scriptstyle{#1}}%
  {\XXint\scriptstyle\scriptscriptstyle{#1}}%
  {\XXint\scriptscriptstyle\scriptscriptstyle{#1}}%
  \!\int}
\def\XXint#1#2#3{{%
    \setbox0=\hbox{$#1{#2#3}{\int}$}
    \vcenter{\hbox{$#2#3$}}\kern-.5\wd0}}

\def\dashint{\Xint-}
\urldef{\mcl}\url{mcl@math.arizona.edu}
\urldef{\millerpd}\url{millerpd@umich.edu}

\newtheorem{lemma}{Lemma}
\newtheorem{theorem}{Theorem}
\newtheorem{rhp}{Riemann-Hilbert Problem}
\newtheorem{rhdbp}{Riemann-Hilbert-$\dbar$ Problem}
\newtheorem{dbp}{$\dbar$ Problem}
\newtheorem{condition}{Condition}
\newtheorem{ar}{Asymptotic Result}

\newenvironment{remark}{$\triangleleft$\,\,{\bf Remark:}}{$\triangleright$}

\title[The $\dbar$ steepest descent method] {The $\dbar$ steepest
  descent method for orthogonal polynomials on the real line with varying weights}
\author{K. T.-R. McLaughlin}
\address{K. T.-R. McLaughlin:  Department of Mathematics\\University of Arizona\\ Tucson, AZ 85721\\
  Email address: \mcl } \author{P. D. Miller} \address{P. D. Miller:
  Department of Mathematics\\ University of Michigan\\ East Hall\\ 530
  Church St.\\ Ann Arbor, MI 48109-1109\\ Email address:
  \millerpd} \date{\today}

\begin{document}

\begin{abstract}
  We obtain Plancherel-Rotach type asymptotics valid in all regions of
  the complex plane for orthogonal polynomials with varying weights of
  the form $e^{-NV(x)}$ on the real line, assuming that $V$ has only
  two Lipschitz continuous derivatives and that the corresponding
  equilibrium measure has typical support properties.  As an
  application we extend the universality class for bulk and edge
  asymptotics of eigenvalue statistics in unitary invariant Hermitian
  random matrix theory.  Our methodology involves developing a new
  technique of asymptotic analysis for matrix Riemann-Hilbert problems
  with nonanalytic jump matrices suitable for analyzing such problems
  even near transition points where the solution changes from
  oscillatory to exponential behavior.
\end{abstract}

\maketitle

\section{Introduction}
\label{sec:Intro}
Let $V(x)$ be a real-valued function (an \emph{external field} or
\emph{potential}) growing faster as $|x|\to\infty$ than $[\log(1 +
x^{2})]^{1+\epsilon}$ for some $\epsilon>0$.  In particular, this
implies that all moments of the measure on $\mathbb{R}$ given by
\begin{equation}
\label{eq:OPMeas}
d\nu_N(x):=e^{-NV(x)}\,dx
\end{equation}
are finite.  A measure of this form is said to be a \emph{varying
  weight} due to the presence of the parameter $N$.  This paper
concerns the asymptotic behavior of polynomials orthogonal with
respect to varying weights of the form \eqref{eq:OPMeas}.  They are
defined as follows.  For $n = 0, 1, 2,\dots$, let $p_{n} = p_{n}(z) =
p_{n}(z;N) = \kappa_{n,n} z^{n} + \cdots + \kappa_{n,0}$,
$\kappa_{n,n} > 0$ be the (unique) polynomial of degree $n$ satisfying
\begin{equation}
\label{eq:OPDEF1}
\int_{\mathbb{R}} p_{n}(x) x^{k} d\nu_N(x) = 0\,,\quad\quad
\text{for $0 \le k < n$,} 
\end{equation}
and
\begin{equation}
\label{eq:OPDEF2}
\int_{\mathbb{R}} p_{n}(x)^{2} d\nu_N(x) = 1.
\end{equation}
The interest is in the behavior of the polynomials \emph{of degree $N$
  and $N-1$}, where the integer $N$ is the same which appears in the
measure of orthogonality, in the limit $N \to \infty$. We obtain a
precise description of the polynomials $p_{N}(z;N)$ and $p_{N-1}(z;N)$
which is uniformly valid for all $z \in \mathbb{C}$, for all $N$
sufficiently large.  This type of asymptotic description is often
referred to as \emph{Plancherel-Rotach asymptotics}, after the
analysis of the Hermite polynomials in \cite{PlaRot}.

In the late 1990s new Riemann-Hilbert techniques originally developed
for the asymptotic analysis of problems in integrable nonlinear
partial differential equations were applied to the asymptotic analysis
of Riemann-Hilbert problems encoding systems of orthogonal polynomials
with respect to varying weights on $\mathbb{R}$, first for the case of
external fields of the form $V(x)=x^{4} - \gamma x^{2}$ \cite{BI}, and
then for the case of general real analytic $V$ in \cite{op1,op2}.  (See
\cite{op1} for more information about Plancherel-Rotach type
asymptotics for orthogonal polynomials prior to the use of
Riemann-Hilbert techniques.)  The Riemann-Hilbert method has been
extended, and applied to various types of asymptotic questions in
approximation theory (see, for example, \cite{KrMcL}, \cite{DbarOP1},
\cite{DOP}, and \cite{KMVV}).  With the exception of \cite{KrMcL} and
\cite{DbarOP1}, all of these applications and extensions deal with
weights that are \emph{real analytic}.

The main purpose of this manuscript is to establish Plancherel-Rotach
type asymptotics for orthogonal polynomials, when the external field
$V$ possesses \emph{only two Lipschitz continuous derivatives}, i.e.\@
in the absence of analyticity.  (The precise assumptions on the
external field $V$ are most naturally described in terms of the
\emph{equilibrium measure} to be defined in subsection
\ref{subsec:EQMEAS} below.)  To obtain such a uniform asymptotic
description we present a new hybrid Riemann-Hilbert-$\dbar$ method of
asymptotic analysis, that is a significant extension of the
$\dbar$-method introduced in \cite{DbarOP1} to analyze orthogonal
polynomials on the unit circle.  By contrast with that method, a
fundamental new feature of orthogonal polynomials with varying weights
\emph{on the real line} is the presence of ``transition points'' (also
known as endpoints of the equilibrium measure) in the neighborhood of
which the asymptotic behavior exhibits a complicated transition from
oscillatory to exponential behavior.

\subsection{Application to random matrices}
\label{sec:RMTOP}
Among many applications of the asymptotic theory of orthogonal
polynomials is the calculation of certain statistics of eigenvalues in
random matrix theory. Unitary invariant ensembles of random matrices
are described by probability measures of the form
\begin{equation}
\label{eq:001}
d\mathbb{P}_N(\mathbf{M}) = \frac{1}{Z_{N}}e^{-N \mathrm{Tr}(V(\mathbf{M}))}\,d\mathbf{M},
\end{equation}
defined on $N \times N$ Hermitian matrices $\mathbf{M}$, where $V(x)$
is an external field of the type described earlier.  Here
$d\mathbf{M}$ denotes Lebesgue measure on the algebraically
independent entries:
\begin{equation}
d\mathbf{M} = \prod_{j=1}^{N} dM_{jj} \prod_{1 \le j < k \le
  N}d\mathrm{Re}(M_{jk})\, d \mathrm{Im}(M_{jk}),
\end{equation}
and $Z_N$ is a normalization constant (partition function).  One of
the origins of the theory of random matrices in physics was the study
of nuclear resonance levels in the 1950s.  See \cite{Mehta} and the
references contained therein for more information.  

\subsubsection{Connection with orthogonal polynomials}
A remarkable connection to orthogonal polynomials was discovered in
the late 1960s by Gaudin and Mehta \cite{GaudinMehta}.  The connection
is the following formula for the density of the probability measure on
eigenvalues induced by \eqref{eq:001}:
\begin{equation}
\label{eq:002}
\mathbb{P}_{N}(\lambda_{1},\dots,\lambda_{N}) =  
\frac{1}{N!}\det\left(  K_{N}(\lambda_{i}, \lambda_{j}) \right)_{1 \le i,j \le N},
\end{equation}
where the function $K_{N}(x,y)$ is the so-called \emph{reproducing kernel} of
orthogonal polynomials:
\begin{equation}
K_{N}(x,y) = e^{-N\left(V(x) + V(y) \right)/2} \sum_{n=0}^{N-1} p_{n}(x) p_{n}(y),
\end{equation}
the polynomials $p_{n}(x)$ being defined in
\eqref{eq:OPDEF1}--\eqref{eq:OPDEF2}.  It is a basic result of the
theory that \eqref{eq:002} indeed defines a probability measure on
$\mathbb{R}^N$.

From formula \eqref{eq:002} one may effectively compute many
statistical quantities involving the eigenvalues (see \cite{Mehta},
and also \cite{Deift}).  Two examples are as follows:
\begin{itemize}
\item \emph{Mean density of eigenvalues} $\rho_{1}^{(N)}(\lambda)$
  defined as
\begin{equation}
  \rho_{1}^{(N)}(\lambda) := \frac{d}{d \lambda} \  \mathbb{E}_{N} \left(
    \frac{1}{N} \# \left\{ \text{eigenvalues $\lambda_{j}$ such that $\lambda_{j} < \lambda$} \right\}
  \right),
\end{equation}
where $\mathbb{E}_{N}(\cdot)$ denotes the expectation of $\cdot$ with
respect to the probability measure \eqref{eq:001} or equivalently
\eqref{eq:002}.  This may be equivalently represented in terms of the
orthogonal polynomials:
\begin{equation}
\rho_{1}^{(N)}(\lambda) = \frac{1}{N} K_{N}(\lambda, \lambda).
\end{equation}
\item \emph{Gap probabilities} $F_{(a,b)}$ defined as
\begin{equation}
\label{eq:gap01}
F_{(a,b)}:=\text{\bf Prob}\left(\text{no eigenvalues in $(a,b)$}\right),
\end{equation}
which may be equivalently represented in terms of a Fredholm
determinant built out of the orthogonal polynomials:
\begin{equation}
\label{eq:gap02}
F_{(a,b)}
= \det\left(1 - \left.\mathcal{K}_{N}\right|_{L^{2}(a,b)}\right).
\end{equation}
Here the integral operator $\mathcal{K}_{N}: L^{2}(a,b) \rightarrow
L^{2}(a,b)$ possesses the integral kernel $K_{N}(x,y)$:
\begin{equation}
\mathcal{K}_{N} h (x)= \int_{a}^{b} K_{N}(x,y) h(y)\,dy.
\end{equation}
\end{itemize}
One important example of the gap probability described in
\eqref{eq:gap01} and \eqref{eq:gap02} is the case that $b=\infty$, for
then the gap probability coincides with the distribution function of
the largest eigenvalue:
\begin{equation}
F_{(a,+\infty)}=\text{\bf Prob}\left(\lambda_{\text{max}} < a \right) = 
\det\left(1 - \left.\mathcal{K}_{N}\right|_{L^2(a, +\infty)}\right).
\end{equation}

\subsubsection{Asymptotic behavior as $N \to \infty$}
A basic and important result concerning the $N\to \infty$ asymptotic
behavior of random matrices is that the mean density of eigenvalues
$\rho_{1}^{(N)}$ has a limit: for all $\lambda \in \mathbb{R}$,
\begin{equation}
\lim_{N \to \infty} \rho_{1}^{(N)}(\lambda) = \psi(\lambda).
\label{eq:densitylimit}
\end{equation}
Note: the \emph{Gaussian Unitary Ensemble} (GUE) first studied by
Wigner corresponds to $V(x) = x^{2}$, and in this case $\psi(\lambda)
= \pi^{-1} \sqrt{2 - \lambda^{2} } $, which is the famous \emph{Wigner
  semicircle law}.  It is well-known that the limit
\eqref{eq:densitylimit} exists for quite general $V(x)$.  It is also
known that if $V$ is real analytic, the convergence in
\eqref{eq:densitylimit} is uniform.  For those nonanalytic $V$ for
which existence of the limiting density $\psi(\lambda)$ can be
established, the convergence implied by the statement
\eqref{eq:densitylimit} has only been proven in a weaker sense.  One
consequence of the present work is that the convergence in
\eqref{eq:densitylimit} is in fact uniform assuming only that the
function $V$ possesses 2 Lipschitz continuous derivatives.

The function $\psi$ is also a well-known quantity in approximation
theory, where it is referred to as the density of the
\emph{equilibrium measure}.  The equilibrium measure is defined in
subsection \ref{subsec:EQMEAS} (for the purposes of the current
discussion one may take the parameter $c$ appearing in the definition
of the equilibrium measure to be unity).

In many circumstances, the largest eigenvalue distribution has been
shown to possess a limit as $N\to\infty$ known as the
\emph{Tracy-Widom distribution}, a distribution function expressible
in closed form in terms of the Hastings-McLeod solution of the
Painlev\'e II equation.  The form of the asymptotic result is:
\begin{equation}
\label{eq:TW}
\lim_{N \to \infty}  \text{\bf Prob}\left( \lambda_{\mathrm{max}} < \beta + 
(\lambda N)^{-2/3} s \right) = F_{\mathrm{TW}}(s)
\end{equation}
where the constant $\lambda$ depends on the external field $V$,
$\beta=\sup(\mathrm{supp}(\psi))$, and $F_{\mathrm{TW}}(s)$ is the
famous Tracy-Widom distribution.

Another fundamental object concerning the eigenvalues of random
matrices is the limiting spacing distribution.  Properly speaking,
this is defined in terms of the \emph{spacing} between ordered
eigenvalues; however a ``poor-man's'' version of this distribution is
the following (easier to define) quantity:
\begin{equation}
\label{eq:SineKer}
Q(s) := \lim_{N \to \infty} \text{\bf Prob}\left(\text{no eigenvalues in $\displaystyle\left(a, a + \frac{s}{N \rho_{1}^{(N)}(a)}\right)$} \right).
\end{equation}
This limit is known to exist provided the external field is real
analytic and provided that $a$ is such that $\psi(a)>0$, and it turns
out that the function $Q(s)$ which emerges in the limit is
\emph{universal} in that it does not depend on properties of the
function $V$.  Indeed, under the assumption that $V$ is real analytic,
one has
\begin{equation}
Q(s) = \det\left(1 - \left.\mathcal{S}\right|_{L^{2}(0,s)}\right),
\end{equation}
where $\mathcal{S}$ is an integral operator on the interval $(0,s)$:
\begin{equation}
\mathcal{S}h(x) := \int_{0}^{s} \frac{\sin(\pi(x - y))}{\pi ( x - y) } h(y)\,dy.
\end{equation}

\vskip 0.2in
Via the connection to orthogonal polynomials explained earlier, the
following asymptotic result concerning the reproducing kernel
$K_{N}(x,y)$ built from the orthogonal polynomials implies
\eqref{eq:TW}:

\begin{ar}
There is a constant $\lambda$ so
  that for every $u,v \in \mathbb{R}$, we have
\begin{equation}
  \lim_{N \to \infty} \frac{1}{(\lambda N)^{2/3}} 
  K_{N}\left(\beta + \frac{u}{(\lambda N)^{2/3}}, \beta + \frac{v}{(\lambda N)^{2/3}} \right) 
  =  \frac{\mathrm{Ai}(u)\mathrm{Ai}'(v)-\mathrm{Ai}(v)\mathrm{Ai}'(u)}{u-v}.
\end{equation}
Here $\mathrm{Ai}$ denotes the unique solution to Airy's equation $y''
= xy$ that is real, and that behaves as follows as $x \to +\infty$:
$\mathrm{Ai}(x) \sim e^{-2 x^{3/2}/3} / ( 2\sqrt{\pi} x^{1/4})$.
\label{ar:Airykernel}
\end{ar}
Similarly, the limit appearing in \eqref{eq:SineKer} is implied by the
following result:
\begin{ar}
For every $a$ with
  $\psi(a)>0$, and every $u,v \in \mathbb{R}$, we have
\begin{equation}
\label{eq:sinkernel}
\lim_{N \to \infty} \frac{1}{N \psi(a) }K_{N}\left(a + \frac{u}{N \psi(a)}, a + \frac{v}{N \psi(a)} \right) = 
\frac{\sin(\pi(u-v ))}{\pi ( u - v) }. 
\end{equation}
\label{ar:sinekernel}
\end{ar}

Asymptotic Result \ref{ar:Airykernel} was first established in the
special case of the Gaussian Unitary Ensemble (i.e. $V(x) = x^{2}$)
\cite{TracyWidom}, using the classical Plancherel-Rotach asymptotics
of Hermite polynomials \cite{PlaRot}.  This was extended to quartic
potentials of the form $V(x)=x^4-\gamma x^2$ in \cite{BI}, where
furthermore Asymptotic Result \ref{ar:sinekernel} was also
established.  Because the polynomials associated with quartic $V$ are
not known to possess elementary contour integral representations, the
analysis of \cite{BI} required a new method, namely the use of the
Riemann-Hilbert formulation of orthogonal polynomials found in
\cite{FIK2}.  Asymptotic Result \ref{ar:sinekernel} was established for
general real analytic potentials $V$ in \cite{op2}, and the asymptotic
formulae for orthogonal polynomials given in \cite{op2} were used to
establish Asymptotic Result \ref{ar:Airykernel} in \cite{DG06}.  The
new strategy introduced in \cite{op2} was a general method linking the
equilibrium measure associated with $V$ to a so-called $g$-function
enabling the use of the non-commutative steepest descent technique for
Riemann-Hilbert problems originally invented by Deift and Zhou
\cite{steepestintro} and extended in \cite{steepestKdV}.  Pastur and
Shcherbina \cite{PS97} have also studied the problem of establishing
Asymptotic Result \ref{ar:sinekernel} under the assumption that $V$
has three continuous derivatives.

As is clear from the above discussion, the historical trend is toward
establishing Asymptotic Results \ref{ar:Airykernel} and
\ref{ar:sinekernel} for more and more general external fields $V$.
The program of \emph{universality} in random matrix theory is
concerned with determining the most general external fields $V$ under
which such results hold true.  In particular, it is of some interest
to admit external fields that are not real analytic.  As pointed out
by Deift in \cite{DeiftOpen}, the steepest descent method that works
so well for analytic $V$ cannot be easily applied to the nonanalytic
case.  The authors recently introduced a ``$\dbar$ steepest descent
method'' applicable to some Riemann-Hilbert problems involving
nonanalytic data, but as formulated in \cite{DbarOP1} this method does
not apply to the orthogonal polynomials described by conditions
\eqref{eq:OPDEF1}--\eqref{eq:OPDEF2} because the equilibrium measure
is compactly supported and the endpoints of support obstruct the type
of nonanalytic deformations involved in the method.  

Among the applications of the results in this manuscript are proofs of
Asymptotic Results \ref{ar:Airykernel} and \ref{ar:sinekernel} under
weakened hypotheses on the external field $V$ (we require two
Lipschitz continuous derivatives) via rigorous Plancherel-Rotach type
asymptotics for orthogonal polynomials.  Our method involves a hybrid
``Riemann-Hilbert-$\dbar$ steepest descent method'' generalizing the
simpler method of \cite{DbarOP1} to handle support endpoints. Our
results hold under weaker conditions on $V$ than those under which
Asymptotic Result \ref{ar:sinekernel} is considered in the paper
\cite{PS97}, and to our knowledge we have the first proof that
Asymptotic Result \ref{ar:Airykernel} holds in the absence of
analyticity of $V$.

We remark that it is not necessary to first obtain asymptotics for the
orthogonal polynomials themselves in order to deduce enough
information about the reproducing kernels $K_N(x,y)$ to establish
Asymptotic Results \ref{ar:Airykernel} and \ref{ar:sinekernel} for
certain general external fields $V$.  For example, a recently
introduced method (based on a comparison principle for Christoffel
functions) of Levin and Lubinsky has been quite successful in
establishing Asymptotic Result \ref{ar:sinekernel} \cite{LLBulk} under
extremely weak global conditions on the external field $V$ and
stronger local conditions (but still far from analyticity) near the
point $a$ of expansion in the spectrum.  Also, Asymptotic Result
\ref{ar:Airykernel} has been studied for real analytic $V$ without
the use of orthogonal polynomials by Pastur and Shcherbina
\cite{PS03}.  

\subsection{Essence of the $\dbar$ method}
In the asymptotic analysis of Riemann-Hilbert problems there is an
analog of contour deformations which plays a crucial role in
identifying subsets of the plane which produce the dominant
contribution to the Riemann-Hilbert problem's solution.  It is common
to begin with a Riemann-Hilbert problem whose solution, $\mathbf{A}$,
is analytic off a given contour $\Sigma_{\mathbf{A}}$, and to use
explicit piecewise analytic quantities to define a new matrix
$\mathbf{B}$ solving a new equivalent Riemann-Hilbert problem in which
the relevant contour $\Sigma_{\mathbf{B}}$ is a deformation of the
original contour $\Sigma_{\mathbf{A}}$.

A fundamental obstacle to this procedure occurs when one requires the
analytic extension from a given contour of a rapidly oscillating
function whose phase possesses no analyticity properties. For the
asymptotic analysis of orthogonal polynomials with varying weights on
the real line, in which the external field $V$ possesses only finitely
many derivatives, this is a central issue.

In this paper, the new approach which circumvents this problem is to
depart from Riemann-Hilbert problems entirely, by introducing
transformations that explicitly violate analyticity.  Instead of
Riemann-Hilbert problems, we then characterize our newly-defined
matrix-valued function as the unique solution of a $\dbar$ problem.

Given a smooth matrix-valued function $\mathbf{W}(x,y)$ of compact
support in $\mathbb{R}^{2}$, a $\dbar$ problem is a
first-order system of linear partial differential equations on
$\mathbb{R}^2$ involving $\mathbf{W}(x,y)$ as coefficients and the
Cauchy-Riemann operator
\begin{equation}
\dbar:=\frac{1}{2}\left(\frac{\partial}{\partial x} + 
i\frac{\partial}{\partial y}\right)
\label{eq:CR}
\end{equation}
acting on the unknown. Here is a prototypical example.
\setcounter{dbp}{-1}
\begin{dbp}[Prototype]
Determine a $2\times 2$ matrix $\mathbf{A}(x,y)$ for $(x,y)\in\mathbb{R}^2$
having the following properties:
\begin{itemize}
\item[]\textbf{Continuity.}  $\mathbf{A}(x,y)$ is a continuous function
  of $x$ and $y$ for $x+iy\in\mathbb{C}$.
\item[]\textbf{Deviation From Analyticity.}  For $x+iy\in\mathbb{C}$, 
\begin{equation}
\dbar\mathbf{A}(x,y) = \mathbf{A}(x,y)\mathbf{W}(x,y), 
\label{eq:dbarEXample}
\end{equation}
(note that $\mathbf{A}(x,y)$ is analytic off the support of
$\mathbf{W}$, because there one has $\dbar\mathbf{A} (x,y) = 0$).
\item[]\textbf{Normalization.}  The matrix $\mathbf{A}(x,y)$ is
  normalized as follows:
\begin{equation}
\lim_{x,y\rightarrow\infty}\mathbf{A}(x,y)=\mathbb{I}.
\label{eq:dbarEXANorm}
\end{equation}
\end{itemize}
\label{dbp:prototype}
\end{dbp}

Once one admits the possibility to use non-analytic extensions of
functions originally defined on contours, one is faced with an
overabundance of choices, and the issue becomes one of selecting,
constructing, or otherwise establishing the existence of, an extension
suitable for subsequent asymptotic analysis.

This idea actually yields an interesting approach to a classical
result of asymptotic analysis.  Given a real-valued function $\theta:
[-1,1] \to \mathbb{R}$ satisfying $\theta(0) = \theta'(0) = 0$,
$\theta''(0) > 0$, the problem is to provide a large $n$ asymptotic
description for the integral
\begin{equation}
\label{eq:steepintex01}
I(n) := \int_{-1}^{1}e^{i n \theta(x) }\, dx\,.
\end{equation}
For convenience, let us assume that $\theta'''(x)$ is bounded and
$\theta''(x)\ge w>0$ for all $x \in [-1,1]$.

The usual approach to this problem (see, for example, \cite{AAA})
involves many steps, including integration by parts, implicit variable
changes, and Taylor expansions, the result of which is
\begin{equation}
\label{eq:steepintXXX}
I(n) = \sqrt{ \frac{ 2 \pi }{ n \theta''(0)}} e^{ i \pi / 4} \left( 1 + \mathcal{O} \left( n^{-1/2} \right) \right).
\end{equation}
We may instead establish this in the following way, which elucidates
certain aspects of the methods we use in the sequel.  Let $\Theta(x,y)$
represent an arbitrary extension of $\theta(x)$, which satisfies
$\Theta(x,0) = \theta(x)$.  Then with the aid of Stokes' theorem, we may
write
\begin{equation}
\label{eq:steepintSto}
I(n) = - \int_{\Gamma} e^{ i n \Theta(x,y)}\, d(x+iy) + 
2i \iint_{A} \dbar e^{i n \Theta(x,y)} \,dA,
\end{equation}
where $\Gamma$ represents a contour in $\mathbb{C}$ from $1$ to $-1$
(different than the interval $[-1,1]$), and $A$ represents the
(oriented) area enclosed by the oriented contour formed by the union
of $[-1,1]$ with $\Gamma$.  See Figure~\ref{fig:StatPhase}.
\begin{figure}
\includegraphics{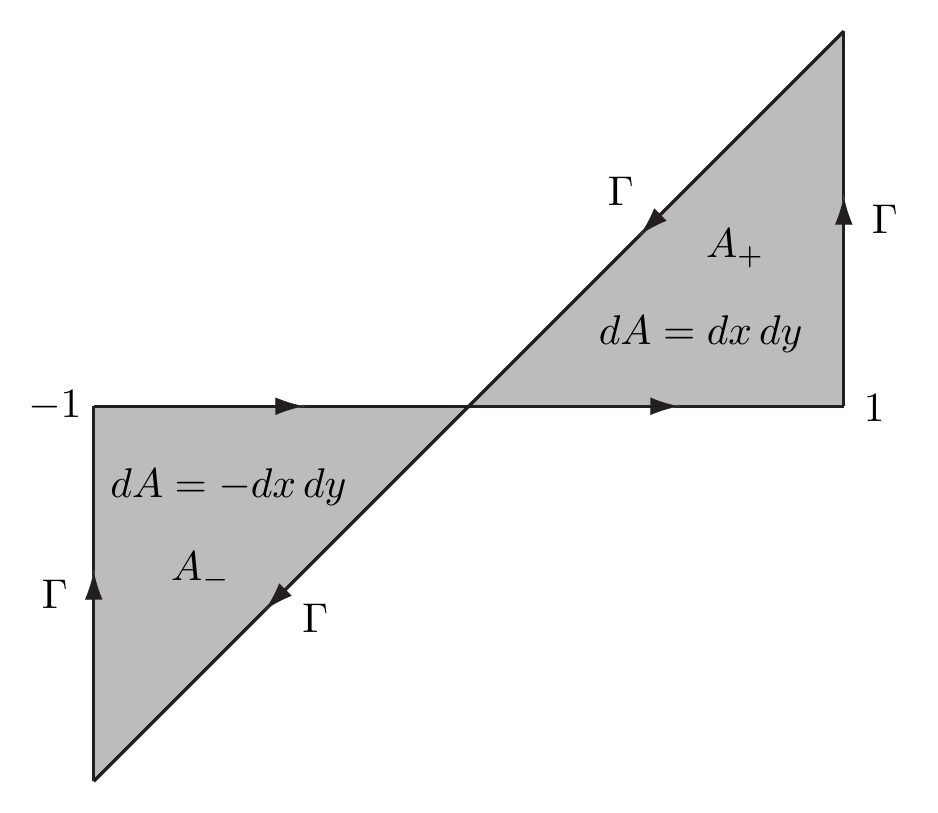}
\caption{The contour $\Gamma$ and the oriented area $A=A_+\cup A_-$
  for the analysis of $I(n)$.}
\label{fig:StatPhase}
\end{figure}

Were the function $\Theta(x,y)$ analytic, the double integral appearing on
the right-hand side of \eqref{eq:steepintSto} would not be present,
and we could choose the contour $\Gamma$ to be the contour of steepest
descent.  Although $\Theta(x,y)$ cannot be analytic if only three derivatives
of $\theta(x)$ are assumed to exist, we nonetheless observe that the
right-hand side of \eqref{eq:steepintSto} is still independent of the
choice of both the contour $\Gamma$ and the particular extension
$\Theta(x,y)$.  This begs the question: Can we pick the extension
$\Theta(x,y)$ and the contour $\Gamma$ so that the right-hand side of
\eqref{eq:steepintSto} may be easily estimated?  The answer is yes.

We take $\Gamma$ to be the contour comprised of a vertical segment
from $(1,0)$ to $(1,1)$, followed by the line segment connecting
$(1,1)$ to $(-1,-1)$, and ending with the vertical line segment from
$(-1,-1)$ to $(-1,0)$, and let $A_{+}$ and $A_{-}$ denote the interior
of the two triangles formed by this contour and the real interval
$[-1,1]$, $A_{+}$ in the first quadrant, and $A_{-}$ in the third
quadrant, exactly as illustrated in Figure~\ref{fig:StatPhase}.  We
will explicitly construct an extension $\Theta(x,y)$ of $\theta(x)$
defined for $(x,y)\in \overline{A_+\cup A_-}$ to satisfy the following
conditions for some constants $K>0$ and $k>0$:
\begin{itemize}
\item[(C1)] $\Theta(x,0) = \theta(x)$, for $-1\le x\le 1$.
\item[(C2)] $\Theta(x,x) = \tfrac{1}{2}\theta''(0)(x + i x)^{2}$, 
for $-1 \le x \le 1$.
\item[(C3)] $\left| \dbar \Theta(x,y) \right| \le K y^2$ for all $(x,y)
  \in A_{+} \cup A_{-}$.
\item[(C4)] $\mathrm{Im}(\Theta(x,y)) \ge k x y $ for all $(x,y) \in 
\overline{A_{+} \cup A_{-}}$.
\end{itemize}
Using such an extension and properties (C1) and (C2), the
representation \eqref{eq:steepintSto} may be rewritten as
\begin{multline}
I(n) - e^{i\pi/4}\int_{-\sqrt{2}}^{\sqrt{2}} e^{- n \theta''(0) s^{2}/2}\, ds
\\
\begin{aligned}
&=i\int_{0}^{-1} e^{i n \Theta(-1,y)}\,dy + i\int_{1}^{0}e^{i n \Theta(1,y)}\,dy \\
&\quad\quad{}- 2 n \iint_{A_{+}}e^{i n \Theta(x,y)} \dbar \Theta(x,y)\, dx\, dy + 
2 n \iint_{A_{-}} e^{ i n \Theta(x,y) }\dbar \Theta(x,y)\, dx\, dy\,.
\end{aligned}
\label{eq:steepintrewrite}
\end{multline}
The four integrals on the right-hand side may be estimated directly
with the help of properties (C3) and (C4):
\begin{equation}
\left|i\int_0^{-1}e^{in\Theta(-1,y)}\,dy\right|\le
\int_0^1 e^{-n\mathrm{Im}(\Theta(-1,-s))}\,ds \le 
\int_0^1 e^{-k n s}\,ds \le \int_0^{+\infty}e^{-k ns}\,ds = 
\frac{1}{k n}\,.
\end{equation}
\begin{equation}
\left|i\int_1^{0}e^{in\Theta(1,y)}\,dy\right|\le
\int_0^1 e^{-n\mathrm{Im}(\Theta(1,s))}\,ds \le 
\int_0^1 e^{-k ns}\,ds \le \int_0^{+\infty}e^{- kns}\,ds = 
\frac{1}{k n}\,.
\end{equation}
\begin{equation}
\begin{split}
\left|\mp 2n\iint_{A_\pm}e^{in\Theta(x,y)}\dbar\Theta(x,y)\,dx\,dy\right|&\le
2n\iint_{A_\pm}e^{-n\mathrm{Im}(\Theta(x,y))}\left|\dbar\Theta(x,y)\right|\,dx\,dy\\
&\le 2Kn\iint_{A_\pm}e^{-k nxy}y^2\,dx\,dy\\
&= 2Kn\iint_{A_+}e^{-k nxy}y^2\,dx\,dy\\
&\le 2Kn\int_0^\infty r\,dr\int_0^{\pi/4} d\theta\,
e^{-knr^2\cos(\theta)\sin(\theta)}r^2\sin^2(\theta)\\
&=\frac{2K}{k^2n}\int_0^\infty e^{-u^2}u^3\,du\int_0^{\pi/4}\frac{d\theta}{\cos^2(\theta)}\,.
\end{split}
\end{equation}
Therefore all terms on the right-hand side of
\eqref{eq:steepintrewrite} are $\mathcal{O}(n^{-1})$ as $n\to\infty$.
Now, since
\begin{equation}
e^{i\pi/4}\int_{-\sqrt{2}}^{\sqrt{2}} e^{- n \theta''(0) s^{2}/2}\, ds = 
\sqrt{ \frac{2 \pi}{n \theta''(0)}} e^{i \pi / 4 } + \text{exponentially
small terms as $n\to\infty$},
\end{equation}
we have established \eqref{eq:steepintXXX} if we can find an extension
$\Theta(x,y)$ of $\theta(x)$ satisfying conditions (C1)--(C4).

The extension $\Theta(x,y)$ may be defined as follows.  First let $B(t)$
represent a ``cut-off'' or ``bump'' function which is infinitely
differentiable and satisfies $B(t) \equiv 0$ for $t$ near $0$, and
$B(t) \equiv 1$ for $t$ near $1$.  More precisely, we assume that
$B:\mathbb{R}\to [0,1]$ is of class $C^{(\infty)}(\mathbb{R})$ and
satisfies $B(x)\equiv 0$ for $x\le 0$ and $B(x)\equiv 1$ for $x\ge 1$.
An example of such a function is 
\begin{equation}
B(t):=\begin{cases}
0\,,& x\le 0\\
\displaystyle
\frac{1}{2}\tanh\left(\frac{t}{1-t^2}\right)+\frac{1}{2}\,,& 0<x<1\\
1\,,&x\ge 1\,,
\end{cases}
\end{equation}
but our analysis will never require the detail of this formula.  Next
define
\begin{equation}
\begin{split}
\Theta_{0}(x,y) &:= \theta(x) + i y \theta'(x) + \frac{1}{2} ( i y )^{2} \theta''(x), \\
\Theta_{0}^{\mathrm{hol}}(x,y) &:= \frac{1}{2}\theta''(0) ( x + i y )^{2}.
\end{split}
\end{equation}
Our extension $\Theta(x,y)$ is then defined via
\begin{equation}
\Theta(x,y) := \left[ 1 - B\left( \frac{y}{x} \right) \right] \Theta_{0}(x,y) 
+  B\left( \frac{y}{x} \right) \Theta_{0}^{\mathrm{hol}}(x,y).
\label{eq:steepintPhidef}
\end{equation}
Note that the function $\Theta_{0}(x,y)$ is an extension of $\theta(x)$
that satisfies
\begin{equation}
\label{eq:steepdbarphi1}
\dbar \Theta_{0}(x,y) = \frac{1}{4} (i y )^{2} \theta'''(x).
\end{equation}
The function $\Theta_0(x,y)$ is a rectilinear version of the type of
extension discussed in \cite{DbarOP1}.  It does not match the desired
quadratic on the diagonal part of the contour $\Gamma$; we use the
function $B$ to smoothly deform this extension to the quadratic
$\Theta_{0}^{\mathrm{hol}}(x,y)$.  Straightforward calculations show
that $\Theta(x,y)$ defined in \eqref{eq:steepintPhidef} satisifes the
four conditions (C1)-(C4) described above.  Indeed,
$\Theta(x,0)=\Theta_0(x,0)=\theta(x)$, so condition (C1) holds, and
$\Theta(x,x)=\Theta_0^{\mathrm{hol}}(x,x)=\tfrac{1}{2}\theta''(0)(x+ix)^2$,
so condition (C2) holds.  To confirm condition (C3), note first that
$\dbar\Theta_0^{\mathrm{hol}}(x,y)\equiv 0$, so using
\eqref{eq:steepdbarphi1} we have
\begin{equation}
\dbar\Theta(x,y)=-\left[1-B\left(\frac{y}{x}\right)\right]
\frac{1}{4}\theta'''(x)y^2 + \frac{1}{2}B'\left(\frac{y}{x}\right)\left[\frac{y}{x^2}-i\frac{1}{x}\right]
\left[\Theta_0(x,y)-\Theta_0^{\mathrm{hol}}(x,y)\right],
\end{equation}
and then by Taylor expansion
\begin{equation}
\Theta_0(x,y)-\Theta_0^{\mathrm{hol}}(x,y)=\frac{1}{3}\theta'''(\xi_1)x^3+\frac{i}{2}\theta'''(\xi_2)x^2y -\frac{1}{2}\theta'''(\xi_3)xy^2,
\end{equation}
for some numbers $\xi_1$, $\xi_2$, and $\xi_3$ in $[-1,1]$, so since
$\theta'''$ is uniformly bounded and $0\le y/x \le 1$ for $(x,y)\in
\overline{A_+\cup A_-}$, we have
\begin{equation}
\Theta_0(x,y)-\Theta_0^{\mathrm{hol}}(x,y) = \mathcal{O}(x^3),
\end{equation}
and so
\begin{equation}
\dbar\Theta(x,y) = \mathcal{O}(y^2) + B'\left(\frac{y}{x}\right)
\mathcal{O}(x^2).
\end{equation}
Finally, since $B'(t)=\mathcal{O}(t^2)$ holds for all
$t\in\mathbb{R}$, we have confirmed condition (C3).  To check condition (C4),
we calculate directly
\begin{equation}
\begin{split}
\mathrm{Im}(\Theta(x,y)) &= 
\left[1-B\left(\frac{y}{x}\right)\right]\mathrm{Im}(\Theta_0(x,y)) + 
B\left(\frac{y}{x}\right)\mathrm{Im}(\Theta_0^{\mathrm{hol}}(x,y))\\
&=\left[1-B\left(\frac{y}{x}\right)\right]y\theta'(x) + B\left(\frac{y}{x}\right)\theta''(0)xy\\
&=\left[1-B\left(\frac{y}{x}\right)\right]xy\theta''(\xi) + B\left(\frac{y}{x}\right)\theta''(0)xy\,,
\end{split}
\end{equation}
for some number $\xi\in [-1,1]$.  Then since by assumption $\theta''(x)\ge w>0$
holds for $x\in [-1,1]$ and since $B:\mathbb{R}\to[0,1]$, we have
\begin{equation}
\mathrm{Im}(\Theta(x,y))\ge\left[1-B\left(\frac{y}{x}\right)\right]wxy +
B\left(\frac{y}{x}\right)wxy  = wxy,
\end{equation}
so condition (C4) is verified as well.

Now we will be starting with a $2 \times 2$ matrix $\mathbf{B}$, which
is the solution of a Riemann-Hilbert problem in which the jump matrix
contains entries of the form $e^{i n \theta(x)}$, with $\theta$ real,
and possessing only $2$ Lipschitz continuous derivatives.  We will
define an extension of $\theta$ in exactly the spirit of the above
calculations, and use it to define a new matrix-valued function
$\mathbf{D}$, that is no longer analytic.  The matrix-valued function
$\mathbf{D}$ will be characterized by a hybrid Riemann-Hilbert-$\dbar$
problem.  The main point is this: our extension of $\theta$ will be
chosen so that this hybrid Riemann-Hilbert-$\dbar$ problem succumbs
to a large-$n$ asymptotic analysis.

\subsection{The equilibrium measure and associated quantities.}
\label{subsec:EQMEAS}
The so-called equilibrium measure associated with the function $V(x)$
and the ratio $c:=N/n$ is well-known to be a key ingredient in
large-degree asymptotics of the polynomial $p_n(z)$ of degree $n$ in
the orthonormal system associated with the measure $d\nu_N(x)$ defined
in terms of $N$ and $V$ by \eqref{eq:OPMeas}.  Here $c>0$ is held
fixed as $n$ (and hence also $N$) tends to infinity.  Generally, given
a real-valued field $V(x)$ defined for $x\in\mathbb{R}$ and a
parameter $c>0$, we may consider the following associated weighted
energy of a positive charge (Borel measure) $\mu$ on the real line
$\mathbb{R}\subset\mathbb{C}$:
\begin{equation}
E[\mu]:=\int_{\mathrm{supp}(\mu)}\int_{\mathrm{supp}(\mu)}
\log\left(\frac{1}{|x-y|}\right)\,d\mu(x)\,d\mu(y) + c
\int_{\mathrm{supp}(\mu)}
V(x)\,d\mu(x).
\end{equation}
The equilibrium measure $\mu_{*}$ is defined to be the unique positive
measure $\mu_{*}$ minimizing $E[\mu]$ subject to the constraint
\begin{equation}
\int_{\mathrm{supp}(\mu)}d\mu(x) = 1.
\label{eq:constraint}
\end{equation}
The equilibrium measure is equivalently characterized by the
corresponding Euler-Lagrange variational conditions.  There is a real
constant $\ell$ (the Lagrange multiplier originating from the
constraint \eqref{eq:constraint}) such that
\begin{equation}
\frac{\delta E}{\delta \mu}\Bigg|_{\mu=\mu_*}\hspace{-0.25 in}(x):=
2\int_{\mathrm{supp}(\mu_*)}
\log\left(\frac{1}{|x-y|}\right)\,d\mu_*(y) + cV(x) \equiv - \ell,
\quad
x\in\mathrm{supp}(\mu_*),
\label{eq:ELinside}
\end{equation}
and
\begin{equation}
\frac{\delta E}{\delta \mu}\Bigg|_{\mu=\mu_*}\hspace{-0.25 in}(x):=
2\int_{\mathrm{supp}(\mu_*)}
\log\left(\frac{1}{|x-y|}\right)\,d\mu_*(y) + cV(x) \ge - \ell,
\quad
x\not\in\mathrm{supp}(\mu_*).
\label{eq:ELoutside}
\end{equation}

\subsection{Assumptions on external field $V$}
\label{subsec:Ass}
We now impose several conditions on the external field, some of which
are best described in terms of the equilibrium measure and its complex
valued ``log-transform'' $g(z)$ defined below in \eqref{eq:gdef}.

\setcounter{condition}{-1}
\begin{condition}[Smoothness of $V$] 
The external field $V$ possesses
  two Lipschitz continuous derivatives.  
\label{cond:smooth}
\end{condition}
\noindent A consequence of this is that the equilibrium measure is absolutely
continuous with respect to Lebesgue measure, with continuous density
$\psi(x)$.

\begin{condition}[Support properties of $\mu_*$] 
We suppose
  that the external field $V$ is such that the equilibrium measure is
  supported on a finite union of intevals,
  $\cup_{j=1}^{G+1}[\alpha_{j},\beta_{j}]$, with $ \alpha_{1} <
  \beta_{1} < \alpha_{2} <\beta_{2} < \cdots < \beta_{G+1}$.
\label{cond:support}
\end{condition}
\noindent By convention for future convenience, we set $\beta_0:=-\infty$ and
$\alpha_{G+2}:=+\infty$.

To describe further conditions on $V$ imposed via its equilibrium
measure $\mu_{*}$, we will require an auxiliary function $g(z)$
analytic for $z \in \mathbb{C} \setminus (-\infty, \beta_{G+1})$
defined in terms of $\mu_{*}$ by
\begin{equation}
g(z):=\int\log(z-s)\,d\mu_*(s) = \int_{\alpha_{1}}^{\beta_{G+1}} \log(z-s)\psi(s)\,ds,
\label{eq:gdef}
\end{equation}
where $\psi(x)$ is the Radon-Nikodym derivative of $\mu_*$, that is,
$d\mu_{*}(x)=\psi(x)\,dx$.  Here we are choosing the branch cut of the
integrand so that for each $s\in\mathbb{R}$, $\log{(z-s)}$ is an
analytic function of $z$ for $z \in \mathbb{C} \setminus (-\infty,s]$
that is real-valued for $z>s$, which ensures the claimed analyticity
properties of $g(z)$.  In terms of $g(z)$ the variational condition
\eqref{eq:ELinside} becomes
\begin{equation}
cV(x)-\left(g_+(x)+g_-(x)\right)=- \ell,\quad
\alpha_{j}<x<\beta_{j},\quad j = 1, \ldots, G+1,
\label{eq:ELinsideSpecial}
\end{equation}
where $g_+(x)$ and $g_-(x)$ denote the boundary values taken by $g(z)$
as $z\rightarrow x$ with $z\in \mathbb{C}_+$ and $z\in\mathbb{C}_-$
respectively.  Also, Condition 1 and the reality of the equilibrium
measure together imply that there are real constants
$\Omega_0,\dots,\Omega_G$ such that
\begin{equation}
g_+(x)-g_-(x)=\begin{cases} i\Omega_0,&\quad x<\alpha_1,\\
i\Omega_j,&\quad \beta_j<x<\alpha_{j+1}\,,\quad j=1,\dots,G,\\
0,&\quad x>\beta_{G+1}\,,
\end{cases}
\label{eq:DifferenceOutside}
\end{equation}
and from the normalization \eqref{eq:constraint} it follows further
that $\Omega_0=2\pi$.  Furthermore, since $\mu_*$ is a positive
measure, 
\begin{equation}
\theta(x):=-i( g_{+} (x)- g_{-}(x) )
\label{eq:thetadefine}
\end{equation}
is real and nonincreasing for all $x\in\mathbb{R}$, so in particular
$2\pi=\Omega_0>\Omega_1>\cdots >\Omega_G>0$.  Assuming that
differentiation commutes with taking boundary values (this may be
easily justified later) \eqref{eq:ELinsideSpecial} and
\eqref{eq:DifferenceOutside} imply that
\begin{equation}
\begin{split}
g'_+(x)+g'_-(x)&=cV'(x)\,,\quad
\alpha_{j}<x<\beta_{j},\quad j = 1, \ldots, G+1,\\
g'_+(x)-g'_-(x)&=0, \quad x \in \mathbb{R} \setminus \mathrm{supp}(\psi).
\end{split}
\end{equation}
In particular, $g'(z)$ is an analytic function for
$z\in\mathbb{C}\setminus\mathrm{supp}(\psi)$.  

Next, for $x\in\mathbb{R}$, define the real-valued function
\begin{equation}
\phi(x):=cV(x)+\ell-g_+(x)-g_-(x).
\label{eq:phidefine}
\end{equation}
According to \eqref{eq:ELinside} and \eqref{eq:ELoutside}, we have
$\phi(x)\equiv 0$ for $x\in\mathrm{supp}(\psi)$ and $\phi(x)\ge 0$ for
$x\in\mathbb{R}\setminus\mathrm{supp}(\psi)$.  

Finally, for $j=1,\dots,G+1$ we define functions
$h_{\alpha_j}:(\beta_{j-1},\beta_j)\to\mathbb{R}$ and
$h_{\beta_j}:(\alpha_j,\alpha_{j+1})\to\mathbb{R}$ by the formulae
\begin{equation}
h_{\alpha_j}(x):=\begin{cases}
\displaystyle 
-\frac{\phi(x)}{\sqrt{\alpha_j-x}}, &\quad \beta_{j-1}<x<\alpha_j,\\
\\
\displaystyle\frac{\theta(\alpha_j)-\theta(x)}{\sqrt{x-\alpha_j}}, &\quad
\alpha_j<x<\beta_j,
\end{cases}
\label{eq:halphajdef}
\end{equation}
and
\begin{equation}
h_{\beta_j}(x):=\begin{cases}
\displaystyle\frac{\theta(x)-\theta(\beta_j)}{\sqrt{\beta_j-x}}, &\quad
\alpha_j<x<\beta_j,\\\\
\displaystyle -\frac{\phi(x)}{\sqrt{x-\beta_j}},&\quad \beta_j<x<\alpha_{j+1}.
\end{cases}
\label{eq:hbetajdef}
\end{equation}
Under the assumption of Condition~\ref{cond:smooth}, the definition
\eqref{eq:halphajdef} extends by continuity to $x=\alpha_j$ and the
definition \eqref{eq:hbetajdef} extends by continuity to $x=\beta_j$;
moreover, these functions will all have one Lipschitz continuous
derivative.  Moreover, if $x$ is bounded away from the support
interval endpoints, $h_{\alpha_j}(x)$ and $h_{\alpha_j}(x)$ will have
a second derivative that is also Lipschitz.  This is shown in
the Appendix in the case of $G=0$ but the same reasoning
also works for $G>0$.  Note that the nonnegativity of the equilibrium
measure implies that $h_{\alpha_j}(x)\ge 0$ and $h_{\beta_j}(x)\ge 0$
for $\alpha_j<x<\beta_j$, and the variational inequality
\eqref{eq:ELoutside} implies that $h_{\alpha_j}(x)\le 0$ for
$\beta_{j-1}<x<\alpha_j$ and that $h_{\beta_j}(x)\le 0$ for
$\beta_j<x<\alpha_{j+1}$.

Now we may state the rest of the conditions that we impose on the
external field $V$.

\begin{condition}[Strict inequalities and behavior at
  endpoints]
For $j=1,\dots,G+1$, we suppose that $\psi(x)>0$ for
  $\alpha_j<x<\beta_j$ and that the functions
  $h_{\alpha_j}:(\beta_{j-1},\beta_j)\to\mathbb{R}$ and
  $h_{\beta_j}:(\alpha_j,\alpha_{j+1})\to\mathbb{R}$ defined by
  \eqref{eq:halphajdef} and \eqref{eq:hbetajdef} 
  satisfy the strict inequalities
\begin{equation}
\text{$h_{\alpha_j}(x)<0$ for $\beta_{j-1}<x<\alpha_j$,}\quad\quad
\text{$h_{\alpha_j}(x)>0$ for $\alpha_j<x<\beta_j$,}\quad\quad
\text{$h_{\alpha_j}'(\alpha_j)>0$,}
\end{equation}
and
\begin{equation}
\text{$h_{\beta_j}(x)>0$ for $\alpha_j<x<\beta_j$,}\quad\quad
\text{$h_{\beta_j}(x)<0$ for $\beta_j<x<\alpha_{j+1}$,}\quad\quad
\text{$h_{\beta_j}'(\beta_j)<0$.}
\end{equation}
\label{cond:strict}
\end{condition}

\begin{condition}[Single interval of support
    w.l.o.g.]  
We assume that $G=0$.
\label{cond:G0}
\end{condition}  
\noindent The analysis for the case of $G>0$ (i.e. more than one interval
comprising the support of $\mu_{*}$) may be deduced in a
straightforward manner from the case of $G=0$ (i.e. one interval
comprising the support of $\mu_{*}$).  So, in the course of our
presentation of the details of the asymptotic analysis of the
orthogonal polynomials, we will assume, \textbf{without loss of
  generality}, that $G=0$ and hence the equilibrium measure is
supported on the single interval $[\alpha,\beta]=[\alpha_1, \beta_1]$.

\subsection{Statement of results}
\label{sec:Results}
Because of the complex-conjugation symmetry $p_{n}(z^*) = p_{n}(z)^*$,
we only need to present asymptotic formulae for the orthogonal
polynomials in the upper half-plane.  While our methods yield
asymptotic formulae valid throughout the whole complex plane, we will
restrict our attention to the regions $\Omega_{+}$ and
$\mathbb{C}_{+}\cap S_{\beta}$ as indicated in
Figure~\ref{fig:Regions}.  We focus on these regions for simplicity
and also because these are most important for applications to random
matrix theory.
\begin{figure}[h]
\begin{center}
\includegraphics{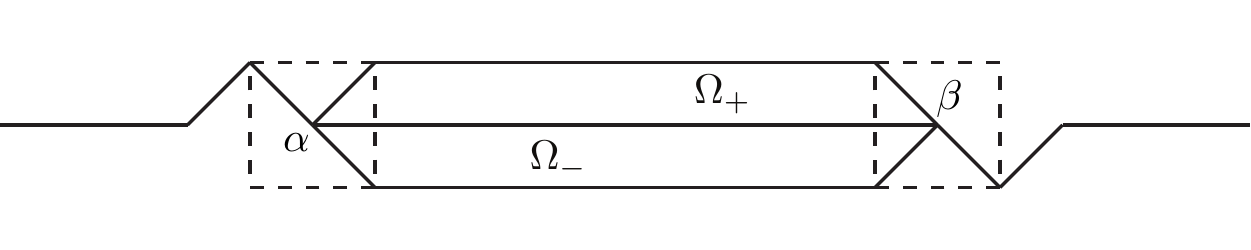}
\end{center}
\caption{The regions $\Omega_\pm$ of the complex plane surround the
  interval $(\alpha,\beta)$.  The square regions $S_\alpha$ centered
  at $\alpha$ and $S_\beta$ centered at $\beta$ are shown with dashed
  boundaries.  These squares have sides of length $2\delta$ for some
  $\delta>0$, and the regions $\Omega_\pm$ each have vertical
  thickness $\delta$.}
\label{fig:Regions}
\end{figure}
\begin{theorem}
  \label{thm:01}
  Suppose the external field $V$ satisfies Condition~\ref{cond:smooth}
  through Condition~\ref{cond:G0} as described earlier.  Let $n, N \to
  \infty$ so that $N/n \to c$ with $0 < c < \infty$.  Then, with
  $(\alpha,\beta)$ representing the support of the equilibrium measure
  $\mu_*$, the following asymptotic descriptions are valid.
\begin{itemize}
\item[1.] The asymptotic behavior of the leading coefficients
  $\kappa_{n-1,n-1}$ and $\kappa_{n,n}$ is given by
  \eqref{eq:kappaasymp}.
\item[2.] For $z$ within $\Omega_{+}$ but outside the squares
  $S_\alpha$ and $S_\beta$, the orthogonal polynomials $p_{n}(z)$ and
  $p_{n-1}(z)$ possess the asymptotic description \eqref{eq:polyas1}
  and \eqref{eq:polyas2} respectively (see also
  \eqref{eq:polyas1a}--\eqref{eq:polyas2a} and
  \eqref{eq:polyas1b}--\eqref{eq:A21axis}).
\item[3.] For $z$ within $S_\beta\cap \mathbb{C}_+$, the orthogonal
  polynomials $p_{n}(z)$ and $p_{n-1}(z)$ possess the asymptotic
  descriptions \eqref{eq:AiryAsTri1}--\eqref{eq:AiryAsTri2}
  respectively (see also \eqref{eq:A11asymp}--\eqref{eq:A11asymp2},
  \eqref{eq:A11betaclose}--\eqref{eq:A21betaclose}, and
  \eqref{eq:A11betacloseAGAIN}--\eqref{eq:A21betacloseAGAIN}).
\item[4.] Asymptotic formulae for the derivatives of the orthogonal
  polynomials may be obtained by differentiating the leading-order
  asymptotics for the polynomials themselves, as described in
  Section~\ref{sec:derivs}.
\end{itemize}
\end{theorem}
These formulae describe the orthogonal polynomials in terms of the
first column of the matrix $\mathbf{A}(z)$ given in
\eqref{eq:Acontents}, and the error terms therein are expressed in
terms of the quantity $\Delta_n$ defined by \eqref{eq:deltadef}.  To
mediate between the orthogonal polynomials contained in the first
column of $\mathbf{A}(z)$ and the \emph{orthonormal} polynomials
$p_{n-1}(z)$ and $p_n(z)$, one must normalize by the leading
coefficients $\kappa_{n-1,n-1}$ and $\kappa_{n,n}$ respectively, whose
asymptotic behavior for large $n$ is given by \eqref{eq:kappaasymp}.
As mentioned at the end of the previous subsection, the assumption
that the support is a single interval is for convenience of
presentation only.  Theorem \ref{thm:01} may be easily extended to
more general settings.  As an example, it straightforward to carry out
all the details if one assumes only Conditions~\ref{cond:smooth},
\ref{cond:support}, and \ref{cond:strict} of the previous subsection.
The following Theorems, describing the application of our results to
random matrix theory, emphasize this point.
\begin{theorem}
  Suppose that the external field $V$ satisfies Conditions
  \ref{cond:smooth}--\ref{cond:strict} of subsection \ref{subsec:Ass}.
  Then Asymptotic Result~\ref{ar:sinekernel} from
  Section~\ref{sec:RMTOP} holds true.
\end{theorem}
\begin{proof}
  The calculations to deduce \eqref{eq:sinkernel} from the asymptotic
  results concerning the orthogonal polynomials are by now standard,
  and we refer the reader to \cite{Mehta} for the details.  (See also
  \cite{op1}.)
\end{proof}

\begin{theorem}
  Suppose that the external field $V$ satisfies
  Conditions~\ref{cond:smooth}--\ref{cond:strict} of subsection
  \ref{subsec:Ass}.  Then Asymptotic Result~\ref{ar:Airykernel} from Section
  \ref{sec:RMTOP} holds true.
\end{theorem}
\begin{proof}
  We again refer the reader to \cite{Mehta} for the details of this
  calculation.  (See also \cite{DG06}.)
\end{proof}

\subsection{Notation}
We will indicate complex conjugation with an asterisk: $z^*$.  All
matrices are written boldface with the notable exception of the identity
matrix $\mathbb{I}$ and the Pauli matrices:
\begin{equation}
\sigma_1:=\begin{pmatrix}0 & 1\\1 & 0\end{pmatrix},\quad
\sigma_2:=\begin{pmatrix}0 & -i \\ i & 0\end{pmatrix},\quad
\sigma_3:=\begin{pmatrix}1 & 0 \\ 0 & -1\end{pmatrix}.
\end{equation}

\subsection{Acknowledgements}  
We are thankful for the hospitality of the faculty and staff of the
Departamento de Matem\'{a}tica of the Pontif\'{i}cia Universidade
Cat\'{o}lica do Rio de Janeiro where work on this project began in
July 2003.  We are also grateful for the partial support of the
National Science Foundation under grant numbers DMS-0200749,
DMS-0451495, and DMS-0800979 (McLaughlin), and DMS-0103909 and
DMS-0354373 (Miller).

\section{Orthogonal Polynomials, Riemann-Hilbert Problems, and
  Equilibrium Measures}
\label{sec:OPRHPEQ}
\subsection{Characterization of orthogonal polynomials via a Riemann-Hilbert 
problem.}
Let $N>0$ be a parameter, and let $V(x)$ be a real-valued function
satisfying merely the conditions set down in the beginning of Section
\ref{sec:Intro}.

The following Riemann-Hilbert problem \cite{FIK2} is known to
characterize the polynomials $\{p_{n}\}_{n=0}^\infty$ orthonormal with
respect to the measure $\nu_N$ given in \eqref{eq:OPMeas}, and defined
by the conditions \eqref{eq:OPDEF1}--\eqref{eq:OPDEF2}.
\begin{rhp}
Find a $2\times 2$ matrix $\mathbf{A}(z)=\mathbf{A}(z;n,N)$
with the properties:
\begin{itemize}
\item[]\textbf{Analyticity.} $\mathbf{A}(z)$ is analytic for $z \in
\mathbb{C}\setminus \mathbb{R}$, and takes continuous boundary
values $\mathbf{A}_+(x)$, $\mathbf{A}_-(x)$ as $z$ tends to $x$ with $x
\in \mathbb{R}$ and $z \in \mathbb{C}_{+}$, $z\in\mathbb{C}_{-}$.
\item[]\textbf{Jump Condition.} The boundary values are connected by
the relation
\begin{equation}
\mathbf{A}_+(x)=\mathbf{A}_-(x)\begin{pmatrix}1 & e^{-NV(x)}\\
0 & 1
\end{pmatrix}.
\end{equation}
\item[]\textbf{Normalization.} The matrix $\mathbf{A}(z)$ is
normalized at $z=\infty$ as follows:
\begin{equation}
\lim_{z\rightarrow\infty}\mathbf{A}(z)\begin{pmatrix}
z^{-n} & 0 \\
 0 & z^n\end{pmatrix}=\mathbb{I}.
 \label{eq:Mnorm}
\end{equation}
\end{itemize}
\label{rhp:A}
\end{rhp}
It was discovered in \cite{FIK2} that Riemann-Hilbert
Problem~\ref{rhp:A} characterizes polynomials orthogonal with respect
to $\nu_N$.  The connection between these orthogonal polynomials and the
solution of Riemann-Hilbert Problem~\ref{rhp:A} is the following:
\begin{equation}
\mathbf{A}(z) = \begin{pmatrix}
\displaystyle\frac{1}{\kappa_{n,n}} p_{n}(z) &
\displaystyle\frac{1}{2 \pi i \kappa_{n,n}} \int_{\mathbb{R}}
\frac{p_{n}(s) e^{-NV(s)}}{s-z}\, ds \\ - 2 \pi i
\kappa_{n-1,n-1}p_{n-1}(z) & 
\displaystyle - \kappa_{n-1,n-1}
\int_{\mathbb{R}} \frac{p_{n-1}(s)e^{-NV(s)}}{s-z}\, ds\end{pmatrix}.
\label{eq:Acontents}
\end{equation}
Note that \eqref{eq:Acontents} implies in particular that
\begin{equation}
\kappa_{n-1,n-1}^2=-\frac{1}{2\pi i}\lim_{z\to\infty}z^{-(n-1)}A_{21}(z)\quad
\text{and}\quad
\kappa_{n,n}^2 = -\frac{1}{2\pi i}\lim_{z\to\infty}z^{-(n+1)}A_{12}(z)^{-1}.
\label{eq:kappas}
\end{equation}
These relationships provide a useful avenue for asymptotic analysis of
the orthogonal polynomials in the limit $n\rightarrow\infty$; it is
sufficient to carry out a rigorous asymptotic analysis of
Riemann-Hilbert Problem~\ref{rhp:A}.

\subsection{Use of the equilibrium measure.}
Given the function $g(z)$ defined by \eqref{eq:gdef}, we introduce an
explicit change of dependent variable into Riemann-Hilbert
Problem~\ref{rhp:A}.  Set
\begin{equation}
\mathbf{B}(z) := e^{-n \ell \sigma_{3}/2} \mathbf{A}(z) 
e^{-ng(z)\sigma_3}e^{n \ell \sigma_{3}/2}.
\label{eq:BfromA}
\end{equation}
It follows from the properties of $\mathbf{A}(z)$ set out in
Riemann-Hilbert Problem~\ref{rhp:A} that the matrix $\mathbf{B}(z)$
is analytic for $z\in\mathbb{C}\setminus\mathbb{R}$ and satisfies the
normalization condition
\begin{equation}
\lim_{z\rightarrow\infty}\mathbf{B}(z)=\mathbb{I}\,.
\end{equation}
The boundary values $\mathbf{B}_+(x)$ and $\mathbf{B}_-(x)$, taken on
the real axis as $z\rightarrow x$ from the upper and lower half-planes
respectively, are continuous functions of $x\in\mathbb{R}$ related by the
following jump condition:
\begin{equation}
\begin{split}
\mathbf{B}_+(x)&=\mathbf{B}_-(x)\begin{pmatrix}
e^{-n(g_{+}(x)-g_{-}(x))} & e^{n(g_+(x)+g_-(x)-cV(x)-\ell)}\\0 & e^{n(g_{+}(x)-g_{-}(x))}
\end{pmatrix}\\ &=
\mathbf{B}_-(x)\begin{pmatrix}e^{-in\theta(x)} & e^{-n\phi(x)}\\0 & e^{in\theta(x)}
\end{pmatrix},
\end{split}
\end{equation} 
where $\theta(x)$ is defined by \eqref{eq:thetadefine} and $\phi(x)$
is defined by \eqref{eq:phidefine}.  From \eqref{eq:DifferenceOutside}
and \eqref{eq:thetadefine}, we see that for $x<\alpha$ and $x>\beta$
(recall that without loss of generality we are assuming that
$\mathrm{supp}(\psi)=[\alpha,\beta]$) this jump condition can be
equivalently written in the form
\begin{equation}
\mathbf{B}_+(x)=\mathbf{B}_-(x)\begin{pmatrix}
1 & e^{-n\phi(x)}\\
0 & 1\end{pmatrix},\quad\text{$x<\alpha$ or $x>\beta$.}
\end{equation}
Similarly, from \eqref{eq:ELinside}, we see that for $\alpha<x<\beta$ the
jump condition takes the form
\begin{equation}
\mathbf{B}_+(x)=\mathbf{B}_-(x)\begin{pmatrix}
e^{-in\theta(x)} & 1 \\0 & e^{in\theta(x)}
\end{pmatrix},\quad\alpha<x<\beta.
\label{eq:Bjumpalphabeta}
\end{equation}

\section{Extensions of $\theta(x)$ and $\phi(x)$}
\label{sec:ThetaMod}
In this section we will define extensions from certain intervals of
the real axis of the functions $\theta(x)$ and $\phi(x)$.  We shall
assume throughout the conditions on the external field $V$ 
described in the Introduction.

\subsection{Existence of extensions with required properties}
\begin{lemma}[Extension of $\theta(x)$]
\label{lem:thetaextendgeneral}
Suppose that $c>0$ is held fixed as $n\rightarrow\infty$ so that the
real-valued function $\theta(x)$ is independent of $n$.  There exists
a function $\Theta(x,y)$ that satisfies the following:

\noindent\textbf{Property 1: Domain, smoothness, and boundary behavior.}
The function $\Theta(x,y)$ is defined for $\alpha<x<\beta$ and
$|y|<\delta$ for some $\delta>0$.  In its domain of definition,
$\Theta(x,y)$ and the partial derivatives $\Theta_x(x,y)$ and
$\Theta_y(x,y)$ are all continuous and uniformly bounded.  Moreover,
if $x+iy$ is bounded away from both $\alpha$ and $\beta$, the second
partial derivatives $\Theta_{xx}(x,y)$, $\Theta_{xy}(x,y)$, and
$\Theta_{yy}(x,y)$ are also continuous and bounded.  The function
$\Theta(x,y)$ is an extension of the real-valued function $\theta(x)$
in the sense that
\begin{equation}
\Theta(x,0)\equiv \theta(x),\quad\alpha<x<\beta.
\end{equation}

\noindent\textbf{Property 2:  Behavior near the real axis.}
There exist finite constants
  $K>0$ and $k>0$, such that the following three estimates hold true:
\begin{equation}
\label{eq:thetadbar}
\left|\dbar \Theta(x,y)\right|
\le K|y| \left| x + i y - \alpha \right|^{1/2} \left| x + i y - \beta \right|^{1/2},\quad\alpha<x<\beta,\quad |y|<\delta.
\end{equation}
\begin{equation}
\label{eq:thetaPbeh}
\mathrm{Im}(\Theta(x,y))\le -k y^{3/2}, 
\quad\alpha<x<\beta,\quad 0\le y<\delta.
\end{equation}
\begin{equation}
\label{eq:thetaMbeh}
\mathrm{Im}(\Theta(x,y))\ge k |y|^{3/2},
\quad\alpha<x<\beta,\quad -\delta<y\le 0.
\end{equation}

\noindent\textbf{Property 3:  Behavior near $\alpha$ and $\beta$.}  
The function
\begin{equation}
\label{eq:thetaadef}
G_\alpha(x,y):=\frac{2\pi-\Theta(x,y)}{(x+iy-\alpha)^{3/2}}
\end{equation}
extends continuously to $x+iy=\alpha$ with the limiting value
$G_\alpha(\alpha,0)=h_\alpha'(\alpha)>0$.
Moreover,
\begin{equation}
\label{eq:thetaaD1}
G_\alpha(x,y)=h_\alpha'(\alpha) + 
\mathcal{O}(|x+iy-\alpha|), \quad \text{as $x+iy\to\alpha$},
\end{equation}
and
\begin{equation}
\label{eq:thetaaD2}
G_\alpha(x,\pm(x-\alpha))\equiv h_\alpha'(\alpha),
\quad \alpha\le x\le \alpha+\delta.
\end{equation}
Similarly, the function
\begin{equation}
\label{eq:thetabdef}
G_\beta(x,y):=\frac{\Theta(x,y)}{(\beta-(x+iy))^{3/2}}
\end{equation}
extends continuously to $x+iy=\beta$, and the limiting value
$G_\beta(\beta,0)=-h_\beta'(\beta)>0$.
Moreover,
\begin{equation}
\label{eq:thetabC1}
G_\beta(x,y)=-h_\beta'(\beta) + 
\mathcal{O}(|x+iy-\beta|), \quad \text{as $x+iy\to\beta$},
\end{equation}
and
\begin{equation}
\label{eq:thetabC2}
G_\beta(x,\pm(\beta-x))\equiv -h_\beta'(\beta),\quad
\beta-\delta\le x\le \beta.
\end{equation}
\end{lemma}

\begin{lemma}[Extension of $\phi(x)$]
\label{lem:phiextendgeneral}
Suppose that $c>0$ is held fixed as $n\rightarrow\infty$ so that the
real-valued function $\phi(x)$ is independent of $n$.  Then there
exists a function $\Phi(x,y)$ that satisfies the following:

\noindent\textbf{Property 1: Domain, smoothness, and boundary
  behavior.}  The function $\Phi(x,y)$ is 
defined in two rectangles: $R_\alpha$ given by
$\alpha-2\delta<x<\alpha$ with $0\le y <\delta$ and
$R_\beta$ given by $\beta<x<\beta+2\delta$ with
$-\delta<y\le 0$ for some $\delta>0$.  In its domain of definition,
$\Phi(x,y)$ and the partial derivatives $\Phi_x(x,y)$ and
$\Phi_y(x,y)$ are all continuous and uniformly bounded.  Moreover, if
$x+iy$ is bounded away from both $\alpha$ and $\beta$, the second
partial derivatives $\Phi_{xx}(x,y)$, $\Phi_{xy}(x,y)$, and
$\Phi_{yy}(x,y)$ are also continuous and bounded.  The function
$\Phi(x,y)$ is an extension of the real-valued function $\phi(x,y)$ in
the sense that
\begin{equation}
  \Phi(x,0)\equiv \phi(x),\quad 
\text{$\beta<x<\beta+2\delta$ and $\alpha - 2 \delta<x<\alpha$}.
\end{equation}

\noindent\textbf{Property 2:  Behavior near the real axis.}
There exist finite constants $K>0$ and $k>0$, such that the
following estimates hold true:
\begin{equation}
\label{eq:phidbar}
\left|\dbar \Phi(x,y)\right|
\le K|y| \left| x + i y - \alpha \right|^{1/2} 
\left| x + i y - \beta \right|^{1/2},\quad
(x,y)\in R_\alpha\cup R_\beta,
\end{equation}
\begin{equation}
\text{$\mathrm{Re}(\Phi(x,y))\ge k|x+iy-\alpha|^{3/2}$ for $(x,y)\in R_\alpha$}\quad\text{and}\quad
\text{$\mathrm{Re}(\Phi(x,y))\ge k|x+iy-\beta|^{3/2}$ for $(x,y)\in R_\beta$}.
\label{eq:RePhiraw}
\end{equation}
(Note that from \eqref{eq:RePhiraw}, the weaker inequality
\begin{equation}
\label{eq:phiPbeh}
\mathrm{Re}(\Phi(x,y))\ge k |y|^{3/2},\quad (x,y)\in R_\alpha\cup 
R_\beta
\end{equation}
follows immediately.)

\noindent\textbf{Property 3:  Behavior near $\alpha$ and $\beta$.}  
The function
\begin{equation}
\label{eq:phiadef}
H_\alpha(x,y):=\frac{\Phi(x,y)}{(\alpha-(x+iy))^{3/2}}
\end{equation}
extends continuously to $x+iy=\alpha$ with the limiting value
$H_\alpha(\alpha,0)=h_\alpha'(\alpha)>0$.  Moreover,
\begin{equation}
\label{eq:phiaD1}
H_\alpha(x,y)=h_\alpha'(\alpha) + \mathcal{O}(|x+iy-\alpha|),  
\quad\text{as $x+iy\to\alpha$}.
\end{equation}
and
\begin{equation}
\label{eq:phiaD2}
H_\alpha(x,\pm(\alpha-x))\equiv h_\alpha'(\alpha),
\quad \alpha-\delta\le x\le\alpha.
\end{equation}
Similarly, the function
\begin{equation}
\label{eq:phibdef}
H_\beta(x,y):=\frac{\Phi(x,y)}{((x+iy)- \beta)^{3/2}}
\end{equation}
extends continuously to $x+iy=\beta$, and the limiting value satisfies
$H_\beta(\beta,0)= -h_\beta'(\beta)>0$.  Moreover,
\begin{equation}
\label{eq:phibC1}
H_\beta(x,y)=-h_\beta'(\beta) + \mathcal{O}(|x+iy-\beta|^{k-2}), 
\text{as $x+iy\to\beta$},
\end{equation}
and
\begin{equation}
\label{eq:phibC2}
H_\beta(x,\pm (x-\beta))\equiv -h_\beta'(\beta),\quad
\beta\le x\le \beta+\delta.
\end{equation}
\end{lemma}

\subsection{Proofs of Lemmas \ref{lem:thetaextendgeneral} and
  \ref{lem:phiextendgeneral}.}  In this subsection we construct
suitable extensions $\Theta(x,y)$ and $\Phi(x,y)$ by further
developing a strategy used in \cite{DbarOP1}.  We will use the
following notation generalizing the ``bump'' function $B(\cdot)$
introduced in Section~\ref{sec:Intro}: for an interval $[a,b]$,
\begin{equation}
B_{[a,b]}(x):=B\left(\frac{x-a}{b-a}\right)\,.
\end{equation}
This $C^{(\infty)}(\mathbb{R})$ function maps $\mathbb{R}$ onto the
interval $[0,1]$ with $B_{[a,b]}(x)\equiv 0$ for $x\le a$ and
$B_{[a,b]}(x)\equiv 1$ for $x\ge b$.  

\subsubsection{Proof of Lemma \ref{lem:thetaextendgeneral}: extension
  of $\theta(x)$}  
\paragraph{\underline{Definition of the extension $\Theta(x,y)$}}
First, we define functions $\Theta_{\alpha,0}(x,y)$ and $\Theta_{\beta,0}(x,y)$
by
\begin{equation}
  \Theta_{\alpha,0}(x,y):=2\pi - ((x+iy)-\alpha)^{1/2}\left[h_\alpha(x)+i(h_\alpha(x+y)-h_\alpha(x))\right],\quad\quad \text{$x<\beta$ and $x+y<\beta$},
\label{eq:thetaextend1alpha}
\end{equation}
and
\begin{equation}
  \Theta_{\beta,0}(x,y):=(\beta-(x+iy))^{1/2}\left[h_{\beta}(x)+i(h_\beta(x+y)-h_\beta(x))\right],\quad\quad\text{$x>\alpha$ and $x+y>\alpha$}.
\label{eq:thetaextend1beta}
\end{equation}
Here, the function $h_\alpha(x)=h_{\alpha_1}(x)$ is defined by
\eqref{eq:halphajdef} for $-\infty<x<\beta$ and the function
$h_\beta(x)=h_{\beta_1}(x)$ is defined by \eqref{eq:hbetajdef} for
$\alpha<x<+\infty$.  

\begin{remark}
  The extensions of the functions $h_\alpha$ and $h_{\beta}$ within
  the square brackets in \eqref{eq:thetaextend1alpha} and
  \eqref{eq:thetaextend1beta} respectively are Cartesian versions of
  the polar-coordinate extensions discussed in \cite{DbarOP1}, further
  generalized with the use of difference quotients in place of 
  derivatives.
\end{remark}

Next, we define analytic approximations of $\Theta_{\alpha,0}(x,y)$ and
$\Theta_{\beta,0}(x,y)$ valid for $x+iy\approx\alpha$ and $x+iy\approx\beta$
respectively:
\begin{equation}
\Theta_{\alpha,0}^{\mathrm{hol}}(x,y):=2\pi-h_\alpha'(\alpha)((x+iy)-\alpha)^{3/2},\quad
\quad\text{$x>\alpha$ or $y\neq 0$},
\end{equation}
and
\begin{equation}
\Theta_{\beta,0}^{\mathrm{hol}}(x,y):=-h_\beta'(\beta)(\beta-(x+iy))^{3/2},\quad\quad
\text{$x<\beta$ or $y\neq 0$}.
\end{equation}
In precisely the spirit of the simple example described in
Section~\ref{sec:Intro}, we may combine the two types of extensions
with the help of an appropriate angular bump function:
\begin{equation}
\Theta_\alpha(x,y):=
B\left(\left|\frac{y}{x-\alpha}\right|\right)\Theta_{\alpha,0}^{\mathrm{hol}}(x,y) +
\left[1-B\left(\left|\frac{y}{x-\alpha}\right|\right)\right]
\Theta_{\alpha,0}(x,y),
\end{equation}
and
\begin{equation}
\Theta_\beta(x,y):=
B\left(\left|\frac{y}{x-\beta}\right|\right)\Theta_{\beta,0}^{\mathrm{hol}}(x,y) +
\left[1-B\left(\left|\frac{y}{x-\beta}\right|\right)\right]\Theta_{\beta,0}(x,y).
\end{equation}
For short we will occasionally write
\begin{equation}
B_{\mathrm{ang},\alpha}:=B\left(\left|\frac{y}{x-\alpha}\right|\right)\quad\text{and}\quad
B_{\mathrm{ang},\beta}:=B\left(\left|\frac{y}{x-\beta}\right|\right).
\end{equation}

Finally, letting
\begin{equation}
a:=\alpha+\frac{1}{3}(\beta-\alpha)\quad\quad\text{and}\quad\quad
b:=\beta-\frac{1}{3}(\beta-\alpha)
\end{equation}
so that $[a,b]\subset (\alpha,\beta)$, we may smoothly glue these two
extensions together through the vertical strip $a<x<b$ in the
$(x,y)$-plane:
\begin{equation}
\Theta(x,y):=B_{[a,b]}(x)\Theta_\beta(x,y) +
\left[1-B_{[a,b]}(x)\right]
\Theta_\alpha(x,y).
\label{eq:Thetaalphabeta}
\end{equation}
This will be our extension of the function $\theta(x)$ from the
interior $(\alpha,\beta)$ of the support interval.  Taking into
account the supports of $B_{[a,b]}(x)$ and $1-B_{[a,b]}(x)$ and
comparing with the regions of definition of $\Theta_{\alpha,0}(x,y)$,
$\Theta_{\beta,0}(x,y)$, $\Theta_{\alpha,0}^{\mathrm{hol}}(x,y)$, and
$\Theta_{\beta,0}^{\mathrm{hol}}(x,y)$, we see that whenever
$\delta<(\beta-\alpha)/3$, $\Theta(x,y)$ is well-defined on the
rectangle $R$ given by the inequalities $\alpha<x<\beta$ and
$|y|<\delta$.

\begin{remark}
  It turns out that if we replace both angular bump functions
  $B_{\mathrm{ang},\alpha}$ and $B_{\mathrm{ang},\beta}$ by the
  constant function $B\equiv 1$, then the extension obtained only
  involves the functions $\Theta_{\alpha,0}(x,y)$ and
  $\Theta_{\beta,0}(x,y)$ glued together through the vertical strip
  $a<x<b$, and this simpler function satisfies all of the desired
  properties with the exception of \eqref{eq:thetabC2} and
  \eqref{eq:thetaaD2} from Property 3.  The purpose of the angular
  bump functions is to smoothly deform the simpler extension into one
  that satisfies these additional conditions (without ruining any of
  the other conditions, of course).
\end{remark}

\paragraph{\underline{Confirmation of Property 1}}
To confirm Property 1, we note that under the assumptions in force,
both functions $h_\alpha$ and $h_\beta$ have one Lipschitz continuous
derivative throughout their respective domains of definition, which
implies that $\partial_x\Theta_{\alpha,0}(x,y)$,
$\partial_y\Theta_{\alpha,0}(x,y)$, $\partial_x\Theta_{\beta,0}(x,y)$,
and $\partial_y\Theta_{\beta,0}(x,y)$ are all continuous and uniformly
bounded throughout the rectangle $R$ of definition of $\Theta(x,y)$.
Since $\Theta_{\alpha,0}^{\mathrm{hol}}(x,y)$ and
$\Theta_{\beta,0}^{\mathrm{hol}}(x,y)$ are analytic functions for
$(x,y)\in R$, their first partial derivatives are certainly
continuous.  Then, since $\Theta(x,y)$ is constructed from these more
elementary functions with the help of $C^{(\infty)}$ bump functions,
it is clear that $\Theta(x,y)$, $\Theta_x(x,y)$, and
$\Theta_y(x,y)$ are all continuous an uniformly bounded
throughout $R$.  As $h_\alpha(x)$ and $h_\beta(x)$ have a second
Lipschitz derivative for $x$ bounded away from $\alpha$ and $\beta$,
similar arguments show that $\Theta_{xx}(x,y)$, $\Theta_{xy}(x,y)$,
and $\Theta_{yy}(x,y)$ are continuous and bounded for $x+iy$ bounded
away from $\alpha$ and $\beta$.  Furthermore, for $\alpha<x<\beta$,
\begin{equation}
\begin{split}
\Theta(x,0)&=B_{[a,b]}(x)\Theta_\beta(x,0) +\left[1-B_{[a,b]}(x)\right]\Theta_\alpha(x,0)\\
&=B_{[a,b]}(x)\Theta_{\beta,0}(x,0)+\left[1-B_{[a,b]}(x)\right]\Theta_{\alpha,0}(x,0)\\
&=B_{[a,b]}(x)\theta(x)+\left[1-B_{[a,b]}(x)\right]\theta(x)\\
&=\theta(x),
\end{split}
\end{equation}
so $\Theta(x,y)$ is indeed an extension of $\theta(x)$ from the
interval $(\alpha,\beta)$ to the rectangle $R$.

\paragraph{\underline{Confirmation of Property 2}}
To confirm Property 2, first note that from \eqref{eq:Thetaalphabeta} we have
\begin{equation}
\dbar\Theta=\dbar B_{[a,b]}\cdot\left(\Theta_\beta-\Theta_\alpha\right)
+ B_{[a,b]}\dbar\Theta_\beta + \left[1-B_{[a,b]}\right]\dbar\Theta_\alpha.
\end{equation}
Now, since $\Theta_\alpha(x,y)$ and $\Theta_\beta(x,y)$ are both extensions
from $(\alpha,\beta)$ of the same function $\theta(x)$, and since they
are both uniformly Lipschitz for $x+iy$ bounded away from $\alpha$ and $\beta$
(this is where $\dbar B_{[a,b]}$ is nonzero), the first term on the right-hand
side is supported in $a<x<b$ and is $\mathcal{O}(|y|)$.  Therefore we certainly
have
\begin{equation}
\left|\dbar B_{[a,b]}\cdot\left(\Theta_\beta-\Theta_\alpha\right)\right|\le K|y|
|x+iy-\alpha|^{1/2}|x+iy-\beta|^{1/2}
\label{eq:alphabetaextenddiff}
\end{equation}
for some constant $K>0$.  It therefore remains to estimate
$\dbar\Theta_\alpha$ for $x+iy$ bounded away from $\beta$ and
$\dbar\Theta_\beta$ for $x+iy$ bounded away from $\alpha$.  
Since $\dbar\Theta_{\alpha,0}^{\mathrm{hol}}(x,y)\equiv 0$ and
$\dbar\Theta_{\beta,0}^{\mathrm{hol}}(x,y)\equiv 0$, we see that
\begin{equation}
\begin{split}
\dbar\Theta_\alpha&=\dbar B_{\mathrm{ang},\alpha}\cdot\left(\Theta_{\alpha,0}^{\mathrm{hol}}-\Theta_{\alpha,0}\right) + \left[1-B_{\mathrm{ang},\alpha}\right]\dbar\Theta_{\alpha,0}\\
\dbar\Theta_\beta&=\dbar B_{\mathrm{ang},\beta}\cdot\left(\Theta_{\beta,0}^{\mathrm{hol}}-\Theta_{\beta,0}\right)+\left[1-B_{\mathrm{ang},\beta}\right]\dbar\Theta_{\beta,0}.
\end{split}
\end{equation}
Now, by direct calculation
\begin{equation}
\left|\dbar B_{\mathrm{ang},\beta}\right| = 
\left|B'\left(\left|\frac{y}{x-\beta}\right|\right)\right|\cdot
\left|\dbar\left|\frac{y}{x-\beta}\right|\right| = 
\left|B'\left(\left|\frac{y}{x-\beta}\right|\right)\right|\cdot
\frac{1}{2}\frac{|x+iy-\beta|}{(x-\beta)^2},
\end{equation}
and since the inequality $|x-\beta|>|y|$ holds whereever the derivative
of the bump function in this formula is nonzero, 
\begin{equation}
\left|\dbar B_{\mathrm{ang},\beta}\right|\le
\left|B'\left(\left|\frac{y}{x-\beta}\right|\right)\right|\cdot
\frac{|x+iy-\beta|}{(x-\beta)^2 + y^2} = 
\left|B'\left(\left|\frac{y}{x-\beta}\right|\right)\right|\cdot
\frac{1}{|x+iy-\beta|}.
\label{eq:dbarbangestimate}
\end{equation}
Also, since for any $w$,
\begin{equation}
h_\beta(w)=\int_\beta^wh_\beta'(s)\,ds = h_\beta'(\beta)(w-\beta)+\int_\beta^w
\left[h_\beta'(s)-h_\beta'(\beta)\right]\,ds,
\end{equation}
we have
\begin{multline}
\Theta_{\beta,0}^{\mathrm{hol}}(x,y)-\Theta_{\beta,0}(x,y) \\
{}= -(\beta-(x+iy))^{1/2}
\left[(1-i)
\int_\beta^x\left[h_\beta'(s)-h_\beta'(\beta)\right]\,ds
+i\int_\beta^{x+y}\left[h_\beta'(s)-h_\beta'(\beta)\right]\,ds\right],
\end{multline}
so, since $h_\beta(x)$ has one Lipschitz continuous derivative, there are
constants $K_1>0$ and $K_2>0$ such that
\begin{equation}
\begin{split}
\left|\Theta_{\beta_0}^{\mathrm{hol}}(x,y)-\Theta_{\beta,0}(x,y)\right| &\le
|x-\beta+iy|^{1/2}\left[K_1(x-\beta)^2 + K_2(x-\beta+y)^2\right]\\
&=|x-\beta+iy|^{1/2}\left[2\overline{K}(x-\beta)^2 + 2K_2y(x-\beta) + K_2y^2\right]\\
&\le |x-\beta+iy|^{1/2}\left[2\overline{K}(x-\beta)^2 + 2K_2|y||x-\beta| + K_2y^2\right],
\end{split}
\end{equation}
where $\overline{K}:=(K_1+K_2)/2$.
Using again the inequality $|x-\beta|>|y|$ (since we are going to
multiply by $\dbar B_{\mathrm{ang},\beta}$), we therefore have
\begin{equation}
\left|\Theta_{\beta_0}^{\mathrm{hol}}(x,y)-\Theta_{\beta,0}(x,y)\right| \le
|x-\beta+iy|^{1/2}\left[2\overline{K}(x-\beta)^2 + 3K_2|y||x-\beta|\right].
\label{eq:thetaholminustheta}
\end{equation}
Since we have both $|B'(t)|\le C$ and $|B'(t)|\le C|t|$ for some
constant $C>0$, it follows from \eqref{eq:dbarbangestimate} and
\eqref{eq:thetaholminustheta} that
\begin{equation}
\begin{split}
\dbar B_{\mathrm{ang},\beta}\cdot\left(\Theta_{\beta,0}^{\mathrm{hol}}-
\Theta_{\beta,0}\right)  &= 
\mathcal{O}\left(|x+iy-\beta|^{-1/2}|y||x-\beta|\right) \\
&=\mathcal{O}\left(|y||x+iy-\beta|^{1/2}\right),
\end{split}
\end{equation}
where we have used $|x-\beta|\le |x+iy-\beta|$ in the last step.  Furthermore, 
\begin{equation}
\dbar\Theta_{\beta,0}(x,y)=(\beta-(x+iy))^{1/2}\dbar\left[h_\beta(x)+
i(h_\beta(x+y)-h_\beta(x))\right] = 
\frac{1}{2}(i-1)(\beta-(x+iy))^{1/2}\left[h_\beta'(x+y)-h_\beta'(x)\right],
\end{equation}
so since $h_\beta'(x)$ is uniformly Lipschitz near $\beta$ and
$1-B_{\mathrm{ang},\beta}$ is bounded, we also have
\begin{equation}
\left[1-B_{\mathrm{ang},\beta}\right]\dbar\Theta_{\beta,0} = 
\mathcal{O}\left(|y||x+iy-\beta|^{1/2}\right)\,.
\end{equation}
Therefore, for $x+iy$ bounded away from $\alpha$ we have
\begin{equation}
\left|\dbar\Theta_\beta(x,y)\right|\le K|y||x+iy-\alpha|^{1/2}|x+iy-\beta|^{1/2}
\end{equation}
for some constant $K>0$.  In a completely analogous fashion we see that
for $x+iy$ bounded away from $\beta$ we have
\begin{equation}
\left|\dbar\Theta_\alpha(x,y)\right|\le K|y||x+iy-\alpha|^{1/2}|x+iy-\beta|^{1/2}.
\end{equation}
Combining these results with \eqref{eq:alphabetaextenddiff} we complete the
proof that
\begin{equation}
\left|\dbar\Theta(x,y)\right|\le K|y||x+iy-\alpha|^{1/2}|x+iy-\beta|^{1/2}.
\end{equation}

Now consider $\mathrm{Im}(\Theta(x,y))$.  Since all of the bump
functions $B_{[a,b]}$, $B_{\mathrm{ang},\alpha}$, and
$B_{\mathrm{ang},\beta}$ are real-valued, it will suffice to analyze
$\mathrm{Im}(\Theta_{\alpha,0}(x,y))$ and
$\mathrm{Im}(\Theta_{\alpha,0}^{\mathrm{hol}}(x,y))$ for $x+iy\in R$
bounded away from $\beta$ and to analyze
$\mathrm{Im}(\Theta_{\beta,0}(x,y))$ and
$\mathrm{Im}(\Theta_{\beta,0}^{\mathrm{hol}}(x,y))$ for $x+iy\in R$
bounded away from $\alpha$.  Writing $\beta-(x+iy)=|\beta-(x+iy)|e^{i\phi}$
with $|\phi|<\pi/2$, we have the exact formulae
\begin{multline}
\mathrm{Im}(\Theta_{\beta,0}(x,y))=|\beta-(x+iy)|^{1/2}y\\
{}\cdot\left[
\frac{h_\beta(x+y)-h_\beta(x)}{y}\cos\left(\frac{\phi}{2}\right) -
\frac{1}{2}\cdot\frac{h_\beta(x)}{\beta-x}\left(1-\tan^2\left(\frac{\phi}{2}
\right)\right)\cos\left(\frac{\phi}{2}\right)\right],
\label{eq:imthetabeta0exact}
\end{multline}
and
\begin{equation}
\mathrm{Im}(\Theta_{\beta,0}^{\mathrm{hol}}(x,y))=|\beta-(x+iy)|^{1/2}y
\left[\frac{1}{2}h_\beta'(\beta)
\left(4\cos\left(\frac{\phi}{2}\right)-\sec\left(\frac{\phi}{2}\right)\right)
\right].
\end{equation}
Since Condition~\ref{cond:strict} requires that $h_\beta'(\beta)<0$,
the condition that $|\phi|<\pi/2$ immediately implies that
\begin{equation}
\frac{1}{2}h_\beta'(\beta)\left(4\cos\left(\frac{\phi}{2}\right)-
\sec\left(\frac{\phi}{2}\right)\right)<\frac{\sqrt{2}}{2}h_\beta'(\beta)<0.
\label{eq:imthetabeta0holbracketsR}
\end{equation}
To analyze the terms in the square brackets in
\eqref{eq:imthetabeta0exact} requires a little more work.  Suppose
first that $|\beta-(x+iy)|<\epsilon_1$.  Then, as $\epsilon_1\to 0$, we have
both
\begin{equation}
\frac{h_\beta(x+y)-h_\beta(x)}{y}\to h_\beta'(\beta)<0\quad\text{and}\quad
-\frac{1}{2}\cdot\frac{h_\beta(x)}{\beta-x}\to\frac{1}{2}h_\beta'(\beta)<0,
\end{equation}
where the inequalities follow from Condition~\ref{cond:strict}.  So $\epsilon_1$
may be taken to be small enough that $|\beta-(x+iy)|<\epsilon_1$ implies both
\begin{equation}
\frac{h_\beta(x+y)-h_\beta(x)}{y}<\frac{1}{2}h_\beta'(\beta)<0
\quad\text{and}\quad
-\frac{1}{2}\cdot\frac{h_\beta(x)}{\beta-x}<\frac{1}{4}h_\beta'(\beta)<0.
\end{equation}
Therefore, since $|\phi|<\pi/2$ implies both
\begin{equation}
\frac{\sqrt{2}}{2}<\cos\left(\frac{\phi}{2}\right)\quad\text{and}\quad
0<1-\tan^2\left(\frac{\phi}{2}\right),
\end{equation}
we see that $|\beta-(x+iy)|<\epsilon_1$ implies
\begin{equation}
\frac{h_\beta(x+y)-h_\beta(x)}{y}\cos\left(\frac{\phi}{2}\right) -
\frac{1}{2}\cdot\frac{h_\beta(x)}{\beta-x}\left(1-\tan^2\left(\frac{\phi}{2}
\right)\right)\cos\left(\frac{\phi}{2}\right)
<\frac{\sqrt{2}}{4}h_\beta'(\beta)<0.
\label{eq:imthetabeta0bracketsdisk}
\end{equation}
On the other hand, if we suppose that $|\beta-(x+iy)|>\epsilon_1/2$,
but that $x>a$ and $|y|<\epsilon_2$ for some $\epsilon_2>0$, then as
$\epsilon_2\to 0$ we have both 
\begin{equation}
\phi\to 0 \quad\text{and}\quad \frac{h_\beta(x+y)-h_\beta(x)}{y}\to h_\beta'(x).
\end{equation}
Now, from \eqref{eq:hbetajdef}, for $\alpha<a<x<\beta$ we have
\begin{equation}
  h_\beta'(x)-\frac{1}{2}\frac{h_\beta(x)}{\beta-x} = 
\frac{\theta'(x)}{\sqrt{\beta-x}}=-\frac{2\pi\psi(x)}{\sqrt{\beta-x}}<-k_1
\end{equation}
with some constant $k_1>0$ as a consequence of
Condition~\ref{cond:strict} and the square-root vanishing of $\psi(x)$
as $x\uparrow\beta$.  Therefore, by choosing $\epsilon_2$ sufficiently
small we will have
\begin{equation}
\frac{h_\beta(x+y)-h_\beta(x)}{y}\cos\left(\frac{\phi}{2}\right)-
\frac{1}{2}\cdot\frac{h_\beta(x)}{\beta-x}\left(1-\tan^2\left(\frac{\phi}{2}
\right)\right)\cos\left(\frac{\phi}{2}\right)<-\frac{1}{2}k_1.
\label{eq:imthetabeta0bracketsstrip}
\end{equation}
as long as $|\beta-(x+iy)|>\epsilon_1/2$, $x>a$, and $|y|<\epsilon_2$.
To combine these estimates, note that if $\delta>0$ is sufficiently
small, the part of the rectangle $R$ given by the inequalities
$\alpha<a<x<\beta$ and $|y|<\delta$ consists of points $(x,y)$ for
which either $|\beta-(x+iy)|<\epsilon_1$ or $|y|<\epsilon_2$, so 
\eqref{eq:imthetabeta0bracketsdisk} and 
\eqref{eq:imthetabeta0bracketsstrip} may be combined to give
\begin{equation}
\frac{h_\beta(x+y)-h_\beta(x)}{y}\cos\left(\frac{\phi}{2}\right)-
\frac{1}{2}\cdot\frac{h_\beta(x)}{\beta-x}\left(1-\tan^2\left(\frac{\phi}{2}
\right)\right)\cos\left(\frac{\phi}{2}\right)
<-k_2<0
\label{eq:imthetabeta0bracketsR}
\end{equation}
for $(x,y)\in R$, where
\begin{equation}
k_2:=\min\left\{\frac{\sqrt{2}}{4}|h_\beta'(\beta)|,\frac{1}{2}k_1\right\}.
\end{equation}
Finally, we may combine \eqref{eq:imthetabeta0holbracketsR} with
\eqref{eq:imthetabeta0bracketsR} to find that 
for $(x,y)\in R$ with $x>a$,
\begin{equation}
\begin{aligned}
\mathrm{Im}(\Theta_{\beta,0}(x,y))&\le 
-k_3y |x+iy-\beta|^{1/2}|x+iy-\alpha|^{1/2}\\
\mathrm{Im}(\Theta_{\beta,0}^{\mathrm{hol}}(x,y))&\le 
-k_3y |x+iy-\beta|^{1/2}|x+iy-\alpha|^{1/2}
\end{aligned}\quad\quad\text{if $y\ge 0$}
\end{equation}
and
\begin{equation}
\begin{aligned}
\mathrm{Im}(\Theta_{\beta,0}(x,y))&\ge 
-k_3y |x+iy-\beta|^{1/2}|x+iy-\alpha|^{1/2}\\
\mathrm{Im}(\Theta_{\beta,0}^{\mathrm{hol}}(x,y))&\ge 
-k_3y |x+iy-\beta|^{1/2}|x+iy-\alpha|^{1/2}
\end{aligned}\quad\quad\text{if $y\le 0$},
\end{equation}
where
\begin{equation}
k_3:=\frac{\displaystyle\min\left\{k_2,\frac{\sqrt{2}}{2}|h_\beta'(\beta)|\right\}}
{\sqrt{(\beta-\alpha)^2+\delta^2}}.
\end{equation}
Completely analogous arguments show that for $(x,y)\in R$ with $x<b$,
\begin{equation}
\begin{aligned}
\mathrm{Im}(\Theta_{\alpha,0}(x,y))&\le 
-k_4y |x+iy-\beta|^{1/2}|x+iy-\alpha|^{1/2}\\
\mathrm{Im}(\Theta_{\alpha,0}^{\mathrm{hol}}(x,y))&\le 
-k_4y |x+iy-\beta|^{1/2}|x+iy-\alpha|^{1/2}
\end{aligned}\quad\quad\text{if $y\ge 0$}
\end{equation}
and
\begin{equation}
\begin{aligned}
\mathrm{Im}(\Theta_{\alpha,0}(x,y))&\ge 
-k_4y |x+iy-\beta|^{1/2}|x+iy-\alpha|^{1/2}\\
\mathrm{Im}(\Theta_{\alpha,0}^{\mathrm{hol}}(x,y))&\ge 
-k_4y |x+iy-\beta|^{1/2}|x+iy-\alpha|^{1/2}
\end{aligned}\quad\quad\text{if $y\le 0$},
\end{equation}
for some constant $k_4>0$.  Letting $k=\min\{k_3,k_4\}>0$ and using
the fact that $\Theta(x,y)$ is a convex combination of
$\mathrm{Im}(\Theta_{\alpha,0}(x,y))$,
$\mathrm{Im}(\Theta_{\alpha,0}^{\mathrm{hol}}(x,y))$,
$\mathrm{Im}(\Theta_{\beta,0}(x,y))$, and
$\mathrm{Im}(\Theta_{\beta,0}^{\mathrm{hol}}(x,y))$ through the
various bump functions involved in the definition, it follows that
\begin{equation}
\begin{split}
\mathrm{Im}(\Theta(x,y))&\le -ky|x+iy-\beta|^{1/2}|x+iy-\alpha|^{1/2},\quad
\quad y\ge 0\\
\mathrm{Im}(\Theta(x,y))&\ge -ky|x+iy-\beta|^{1/2}|x+iy-\alpha|^{1/2},\quad
\quad y\le 0
\end{split}
\label{eq:Imthetapenultimate}
\end{equation}
holds for $(x,y)\in R$ if the thickness parameter $\delta$ of the
rectangle $R$ is sufficiently small.  Now, if $x\le (\alpha+\beta)/2$
then $|x+iy-\beta|$ is bounded away from zero while
$|x+iy-\alpha|\ge |y|$, and if $x\ge (\alpha+\beta)/2$ then
$|x+iy-\alpha|$ is bounded away from zero while $|x+iy-\beta|\ge
|y|$.  Combining these observations with \eqref{eq:Imthetapenultimate}
yields \eqref{eq:thetaPbeh} and \eqref{eq:thetaMbeh}.

\paragraph{\underline{Confirmation of Property 3}}
To confirm that $\Theta(x,y)$ satisfies Property 3, note that 
\begin{equation}
G_{\beta,0}(x,y):=\frac{\Theta_{\beta,0}(x,y)}{(\beta-(x+iy))^{3/2}} = 
\frac{h_\beta(x)+i(h_\beta(x+y)-h_\beta(x))}{\beta-(x+iy)}
\end{equation}
is continuous near $(x,y)=(\beta,0)$ and satisfies
\begin{equation}
G_{\beta,0}(x,y)=-h_\beta'(\beta)+\mathcal{O}(|\beta-(x+iy)|)
\end{equation}
because $h_\beta'(x)$ is Lipschitz continuous.  Similarly,
\begin{equation}
  G_{\beta,0}^{\mathrm{hol}}(x,y):=\frac{\Theta_{\beta,0}^{\mathrm{hol}}(x,y)}{(\beta-(x+iy))^{3/2}}\equiv
  -h_\beta'(\beta)>0
\end{equation}
so since near $x+iy=\beta$ (that is, for $x>b$) $G_\beta(x,y)$ is a
convex combination of $G_{\beta,0}(x,y)$ and
$G_{\beta,0}^{\mathrm{hol}}(x,y)$, the requirement \eqref{eq:thetabC1}
on $G_\beta(x,y)$ given in Property 3 is met.  And since for $|y|\ge
\beta-x$ we have $G_\beta(x,y)=G_{\beta,0}^{\mathrm{hol}}(x,y)$, we
also confirm the requirement \eqref{eq:thetabC2}.  Similar
calculations show that $G_\alpha(x,y)$ satisfies the requirements
\eqref{eq:thetaaD1} and \eqref{eq:thetaaD2}.

\subsubsection{Proof of Lemma \ref{lem:phiextendgeneral}: extension of
  $\phi(x)$} The construction of a suitable extension of $\phi(x)$
follows the same general procedure as the construction above of
$\Theta(x,y)$.  We give all details of the construction, after which
it is straightforward to follow the reasoning given in the proof of
Lemma~\ref{lem:thetaextendgeneral} to establish that Properties 1, 2, 
and 3 are satisfied.  

We first define
\begin{equation}
\Phi_{\alpha,0}(x,y):=-(\alpha-(x+iy))^{1/2}\left[h_\alpha(x)+i
\left(h_\alpha(x+y)-h_\alpha(x)\right)\right], \quad
\text{$x<\beta$ and $x+y<\beta$},
\end{equation}
and
\begin{equation}
\Phi_{\beta,0}(x,y):=-((x+iy)-\beta)^{1/2}\left[h_\beta(x)+i\left(h_\beta(x+y)-
h_\beta(x)\right)\right],\quad \text{$x>\alpha$ and $x+y>\alpha$}.
\end{equation}
The analytic approximations of these functions valid for $x+iy\approx\alpha$
and $x+iy\approx\beta$ respectively are
\begin{equation}
\Phi_{\alpha,0}^{\mathrm{hol}}(x,y):=h_\alpha'(\alpha)(\alpha-(x+iy))^{3/2},\quad
\text{$x<\alpha$ or $y\neq 0$},
\end{equation}
and
\begin{equation}
\Phi_{\beta,0}^{\mathrm{hol}}(x,y):= -h_\beta'(\beta)((x+iy)-\beta)^{3/2},\quad
\text{$x>\beta$ or $y\neq 0$}.
\end{equation}
Since the rectangles $R_\alpha$ and $R_\beta$ are disjoint,
there is no need to merge functions defined near $x+iy=\alpha$ with functions
defined near $x+iy=\beta$, so we may simply define
\begin{equation}
\Phi(x,y):=\begin{cases}
\displaystyle
B\left(\left|\frac{y}{\alpha-x}\right|\right)\Phi_{\alpha,0}^{\mathrm{hol}}(x,y)+
\left[1-B\left(\left|\frac{y}{\alpha-x}\right|\right)\right]\Phi_{\alpha,0}(x,y),
&\quad (x,y)\in R_\alpha\\\\
\displaystyle
B\left(\left|\frac{y}{x-\beta}\right|\right)\Phi_{\beta,0}^{\mathrm{hol}}(x,y)+
\left[1-B\left(\left|\frac{y}{x-\beta}\right|\right)\right]\Phi_{\beta,0}(x,y),
&\quad (x,y)\in R_\beta.
\end{cases}
\end{equation}

\section{An equivalent Riemann-Hilbert-$\dbar$ problem.}
The jump condition satisfied by the boundary values taken by
$\mathbf{B}(z)$ on $(\alpha,\beta)$ can be written in the factored
form:
\begin{equation}
\mathbf{B}_+(x)=\mathbf{B}_-(x)\begin{pmatrix}
1 & 0 \\ e^{in\theta(x)} & 1\end{pmatrix}
\begin{pmatrix} 0 & 1\\-1 & 0\end{pmatrix}
\begin{pmatrix} 1 & 0 \\e^{-in\theta(x)} & 1\end{pmatrix}.
\end{equation}
Consider the contour $\Sigma$ illustrated in Figure~\ref{fig:OpenLenses}.
\begin{figure}[h]
\begin{center}
\includegraphics{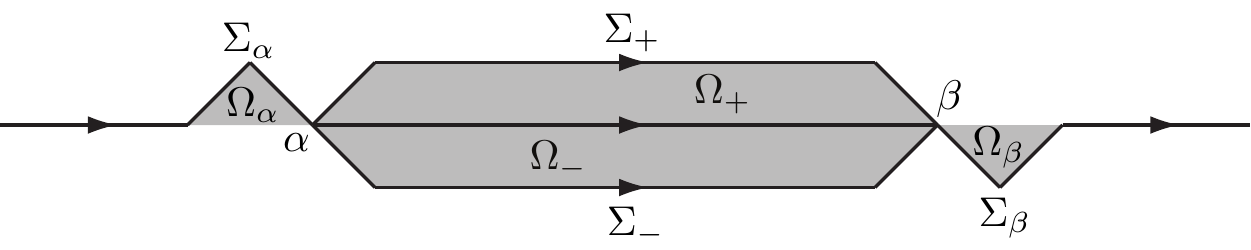}
\end{center}
\caption{The oriented contour $\Sigma$ consists of the real intervals
  $(-\infty,\alpha-2\delta)$, $(\alpha,\beta)$, and
  $(\beta+2\delta,+\infty)$, along with the indicated contour segments
  $\Sigma_\alpha$ connecting $\alpha-2\delta$ to $\alpha$, $\Sigma+$
  and $\Sigma_-$ connecting $\alpha$ to $\beta$, and $\Sigma_\beta$
  connecting $\beta$ to $\beta+2\delta$.  All contour segments are
  oriented left-to-right, and all nonhorizontal segments have
  slopes $\pm 1$.}
\label{fig:OpenLenses}
\end{figure}

Let $\Theta(x,y)$ be any extension of $\theta(x)$ having all three
properties described in Lemma~\ref{lem:thetaextendgeneral}, and let
$\Phi(x,y)$ be any extension of $\phi(x)$ having all three properties
described in Lemma~\ref{lem:phiextendgeneral}.  We define a matrix
$\mathbf{D}(x,y)$ for $z=x+iy\in \mathbb{C}\setminus \Sigma$ relative
to the domains $\Omega_+$, $\Omega_-$, $\Omega_\alpha$, and
$\Omega_\beta$ shown as shaded regions in Figure~\ref{fig:OpenLenses}
as follows.  Set
\begin{equation}
\mathbf{D}(x,y):=\mathbf{B}(x+iy)\begin{pmatrix}
1 & 0 \\-e^{-in\Theta(x,y)} & 1\end{pmatrix},\quad
x+iy\in \Omega_+,
\label{eq:DfromBOmegaplus}
\end{equation}
\begin{equation}
\mathbf{D}(x,y):=\mathbf{B}(x+iy)\begin{pmatrix}
1 & 0 \\
e^{in\Theta(x,y)} & 1\end{pmatrix},\quad
x+iy\in\Omega_-,
\end{equation}
\begin{equation}
\mathbf{D}(x,y):=\mathbf{B}(x+iy)\begin{pmatrix}
1 &  - e^{-n \Phi(x,y)}\\
0 & 1\end{pmatrix},\quad
x+iy\in\Omega_{\alpha},
\end{equation}
\begin{equation}
\mathbf{D}(x,y):=\mathbf{B}(x+iy)\begin{pmatrix}
1 &   e^{-n \Phi(x,y)}\\
0 & 1\end{pmatrix},\quad
x+iy\in\Omega_{\beta},
\end{equation}
and for all remaining $z\in\mathbb{C}\setminus \Sigma$, we set 
$\mathbf{D}(x,y):=\mathbf{B}(x+iy)$.

Because it is explicitly related to $\mathbf{B}(x+iy)$ and hence to
$\mathbf{A}(x+iy)$, the matrix $\mathbf{D}(x,y)$ will solve
Riemann-Hilbert-$\dbar$ problem \ref{rhdbp:D} to be defined below.
Define the jump matrix $\mathbf{V}_{\mathbf{D}}(z)$ for $z\in\Sigma$
as follows:
\begin{equation}
\mathbf{V}_{\mathbf{D}}(z) := \begin{pmatrix}
1 & e^{-n\phi(x)}\\ 0 & 1\end{pmatrix},\quad
\text{$z=x<\alpha-2\delta$ and $z=x>\beta+2\delta$,}
\end{equation}
\begin{equation}
\mathbf{V}_{\mathbf{D}}(z) := \begin{pmatrix}
1 & e^{-n\Phi(x,y)}\\ 0 & 1\end{pmatrix},\quad
z=x+iy\in \Sigma_{\alpha} \cup \Sigma_{\beta},
\end{equation}
\begin{equation}
\mathbf{V}_{\mathbf{D}}(z) := \begin{pmatrix}
0 & 1 \\ -1 & 0\end{pmatrix},\quad z=x\in (\alpha,\beta),
\label{eq:Djumpalphabeta}
\end{equation}
\begin{equation}
\mathbf{V}_{\mathbf{D}}(z) := \begin{pmatrix}
1 & 0 \\ e^{-in\Theta(x,y)} & 1\end{pmatrix},\quad
z=x+iy\in\Sigma_+,
\end{equation}
\begin{equation}
\mathbf{V}_{\mathbf{D}}(z) := \begin{pmatrix}
1 & 0 \\ e^{in\Theta(x,y)} & 1\end{pmatrix},\quad
z=x+iy\in\Sigma_-.
\end{equation}
Also, define the auxiliary matrix $\mathbf{W}_{0}(x,y)$ as  follows.  
\begin{equation}
\mathbf{W}_{0}(x,y):=\begin{pmatrix} 
0 & 0
\\
ine^{-in\Theta(x,y)}\dbar\Theta(x,y) & 0
\end{pmatrix},\quad x+iy\in\Omega_+,
\label{eq:W0defplus}
\end{equation}
\begin{equation}
\mathbf{W}_{0}(x,y):=\begin{pmatrix} 
0 & 0
\\
ine^{in\Theta(x,y)}\dbar\Theta(x,y) & 0
\end{pmatrix},\quad x+iy\in\Omega_-,
\label{eq:W0defminus}
\end{equation}
\begin{equation}
\mathbf{W}_{0}(x,y):=\begin{pmatrix} 
0 & n e^{-n \Phi(x,y)} \dbar \Phi(x,y)
\\
0 & 0
\end{pmatrix},\quad x+iy\in\Omega_{\alpha},
\label{eq:W0defalpha}
\end{equation}
\begin{equation}
\mathbf{W}_{0}(x,y):=\begin{pmatrix} 
0 & -n e^{-n \Phi(x,y)} \dbar \Phi(x,y)
\\
0 & 0
\end{pmatrix},\quad x+iy\in\Omega_{\beta}.
\label{eq:W0defbeta}
\end{equation}
For all remaining $(x,y)\in\mathbb{R}^2$, we set
$\mathbf{W}_0(x,y):=\mathbf{0}$.  Note that $\mathbf{W}_0(x,y)$
so-defined is compactly supported.  From the properties of the matrix
$\mathbf{B}(z)$ inherited via the substitution \eqref{eq:BfromA} from
properties of the matrix $\mathbf{A}(z)$ contained in the statement of
Riemann-Hilbert Problem~\ref{rhp:A}, it follows that $\mathbf{D}(x,y)$
solves the following hybrid Riemann-Hilbert-$\dbar$ problem:
\begin{rhdbp}
Find a $2\times 2$ matrix $\mathbf{D}(x,y)$ with the properties:
\begin{itemize}
\item[]\textbf{Continuity.}  $\mathbf{D}(x,y)$ is a continuous
  function of $x$ and $y$ for $x+iy\in\mathbb{C}\setminus\Sigma$ 
  taking continuous boundary values $\mathbf{D}_+(x,y)$ (respectively
  $\mathbf{D}_-(x,y)$) on $\Sigma$ from the left (respectively right).
\item[]\textbf{Jump Conditions.} The boundary values are connected by the
relation
\begin{equation}
\mathbf{D}_+(x,y)=\mathbf{D}_-(x,y)
\mathbf{V}_{\mathbf{D}}(x+iy),\quad
x+iy\in \Sigma.
\end{equation}
\item[]\textbf{Deviation From Analyticity.}  For $x+iy\in\mathbb{C}$,
\begin{equation}
\dbar\mathbf{D}(x,y) = \mathbf{D}(x,y)\mathbf{W}_{0}(x,y). 
\end{equation}
(Note that in particular $\dbar\mathbf{D}(x,y)=\mathbf{0}$ for
$x+iy\not\in\Omega_+\cup\Omega_-\cup\Omega_\alpha\cup\Omega_\beta$.)
\item[]\textbf{Normalization.}  The matrix $\mathbf{D}(x,y)$ is
  normalized as follows:
\begin{equation}
\lim_{x,y\rightarrow\infty}\mathbf{D}(x,y)=\mathbb{I}.
\end{equation}
\end{itemize}
\label{rhdbp:D}
\end{rhdbp}

\section{Construction of a Global Approximation to $\mathbf{D}(x,y)$}
In this section we will build a global approximation to $\mathbf{D}(x,y)$
by considering a Riemann-Hilbert problem obtained from
Riemann-Hilbert-$\dbar$ problem~\ref{rhdbp:D} by ignoring the
``$\dbar$ component'' of the problem:
\begin{rhp}
Find a $2\times 2$ matrix $\dot{\mathbf{D}}(z)$ with the properties:
\begin{itemize}
\item[]\textbf{Analyticity.}  $\dot{\mathbf{D}}(z)$ is an analytic function
  for $z \in \mathbb{C}\setminus\Sigma$ taking
  continuous boundary values $\dot{\mathbf{D}}_+(z)$ (respectively 
$\dot{\mathbf{D}}_-(z)$) on $\Sigma$ from the left (respectively right).
\item[]\textbf{Jump Conditions.} The boundary values are connected by the
relation
\begin{equation}
\dot{\mathbf{D}}_+(z)=\dot{\mathbf{D}}_-(z)
\mathbf{V}_{\mathbf{D}}(x,y),\quad
z=x+iy\in \Sigma.
\end{equation}
\item[]\textbf{Normalization.}  The matrix $\dot{\mathbf{D}}(z)$ is
  normalized as follows:
\begin{equation}
\lim_{z\rightarrow\infty}\dot{\mathbf{D}}(z)=\mathbb{I}.
\end{equation}
\end{itemize}
\label{rhp:Ddot}
\end{rhp}

Riemann-Hilbert Problem~\ref{rhp:Ddot} has been obtained from
Riemann-Hilbert-$\dbar$ Problem~\ref{rhdbp:D} in an ad-hoc fashion,
and even though $\mathbf{D}(x,y)$ clearly exists, it is not
immediately clear that a solution $\dot{\mathbf{D}}(z)$ to
Riemann-Hilbert Problem~\ref{rhp:Ddot} exists.
Theorem~\ref{thm:RHPExist} below asserts that a unique solution
exists, and describes important asymptotic properties of
$\dot{\mathbf{D}}(z)$.  The Theorem describes the asymptotic behavior
of the solution in three different regions of the complex plane: two
square domains $S_\alpha$ and $S_\beta$ of
side-length $2 \delta$, with $S_\alpha$ centered at $\alpha$
and $S_\beta$ centered at $\beta$, and one exterior domain,
$\mathbb{C} \setminus \left( S_{\alpha} \cup
  S_{\beta} \right)$.  We further subdivide each square into
four regions according to the contour $\Sigma$ as indicated in
Figure~\ref{fig:Squares}.
\begin{figure}[h]
\begin{center}
\includegraphics{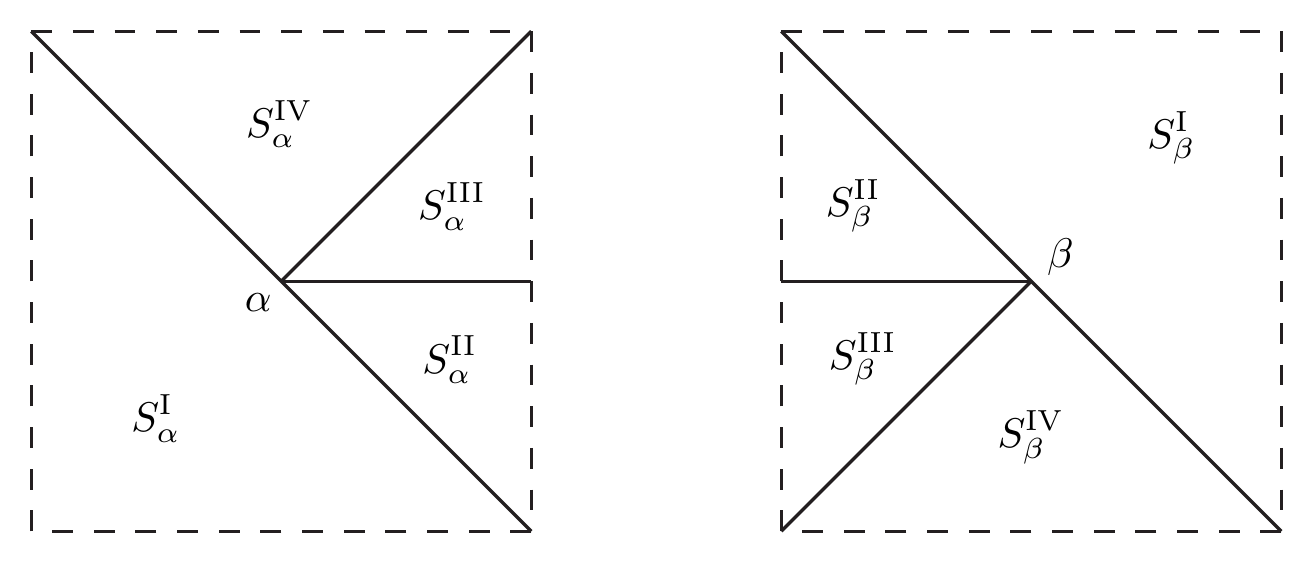}
\end{center}
\caption{The squares $S_\alpha$ (left) and
  $S_\beta$ (right), each subdivided into four regions as
  indicated.}
\label{fig:Squares}
\end{figure}

Note that according to Property 3 in
Lemma~\ref{lem:thetaextendgeneral} and
Lemma~\ref{lem:phiextendgeneral}, the restriction of the jump matrix
$\mathbf{V}_{\mathbf{D}}(z)$ to $\Sigma\cap S_\alpha$ and
$\Sigma\cap S_\beta$ is piecewise analytic.  Indeed, if we define
\begin{equation}
u_\alpha=u_\alpha(z):=
\left[h_\alpha'(\alpha)\right]^{2/3}(\alpha-z),\quad\text{and}\quad
u_\beta=u_\beta(z):=\left[-h_\beta'(\beta)\right]^{2/3}(z-\beta),
\end{equation}
with the positive two-thirds power meant in each case, then we have
\begin{equation}
\mathbf{V}_{\mathbf{D}}(z)=\begin{pmatrix} 1 & e^{-nu_\alpha^{3/2}}\\
0 & 1\end{pmatrix},\quad z\in\Sigma_\alpha\cap S_\alpha,
\end{equation}
\begin{equation}
\mathbf{V}_{\mathbf{D}}(z)=\begin{pmatrix}0 & 1\\-1 & 0\end{pmatrix},
\quad z\in(\alpha,\beta)\cap S_\alpha,
\end{equation}
\begin{equation}
\mathbf{V}_{\mathbf{D}}(z)=\begin{pmatrix}1 & 0 \\e^{nu_\alpha^{3/2}} & 1
\end{pmatrix},\quad z\in \Sigma_\pm\cap S_\alpha,
\end{equation}
and
\begin{equation}
\mathbf{V}_{\mathbf{D}}(z)=\begin{pmatrix}1 & e^{-nu_\beta^{3/2}}\\
0 & 1\end{pmatrix},\quad z\in\Sigma_\beta\cap S_\beta,
\end{equation}
\begin{equation}
\mathbf{V}_{\mathbf{D}}(z)=\begin{pmatrix}0 & 1\\-1 & 0\end{pmatrix},
\quad z\in(\alpha,\beta)\cap S_\beta,
\end{equation}
\begin{equation}
\mathbf{V}_{\mathbf{D}}(z)=\begin{pmatrix}1 & 0 \\e^{nu_\beta^{3/2}} & 1
\end{pmatrix},\quad z\in \Sigma_\pm\cap S_\beta.
\end{equation}
For all $z\in\Sigma$ outside of the squares $S_\alpha$ and
$S_\beta$, with the notable exception of the interval
$(\alpha+\delta,\beta-\delta)$ where $\mathbf{V}_{\mathbf{D}}(z)$ is a
constant matrix, the jump matrix $\mathbf{V}_{\mathbf{D}}(z)$ decays
exponentially to the identity matrix as $n\to\infty$, as a consequence
of both the variational inequality $\phi(x)>0$ for $x<\alpha-2\delta$
and $x>\beta+2\delta$ and also the inequalities on
$\mathrm{Im}(\Theta(x,y))$ in Property 2 of
Lemma~\ref{lem:thetaextendgeneral} and the inequality on
$\mathrm{Re}(\Phi(x,y))$ in Property 2 of
Lemma~\ref{lem:phiextendgeneral}.

In a way that is by now quite standard (see \cite{op1,op2}), these
facts suggest an explicit model for $\dot{\mathbf{D}}(z)$ that we will
call $\hat{\mathbf{D}}(z)$ and that we will now define.  Let
$\gamma(z)$ be the function analytic in $\mathbb{C}\setminus
[\alpha,\beta]$ determined by the conditions
\begin{equation}
\gamma(z)^4 = \frac{z-\beta}{z-\alpha}\,,\quad\text{and}
\quad \lim_{z\to\infty}\gamma(z)=1,
\label{eq:gamma}
\end{equation}
and let $\mathbf{U}$ denote the unitary eigenvector matrix for
$\mathbf{V}_{\mathbf{D}}(z)$ on $(\alpha,\beta)$:
\begin{equation}
\mathbf{U}:=\frac{1}{\sqrt{2}}\begin{pmatrix}e^{-i\pi/4} & e^{i\pi/4}\\
e^{i\pi/4} & e^{-i\pi/4}\end{pmatrix}.
\end{equation}
Then, we set
\begin{equation}
\hat{\mathbf{D}}(z):=\mathbf{U}\gamma(z)^{\sigma_3}\mathbf{U}^\dagger=
\begin{pmatrix}\tfrac{1}{2}\left(\gamma(z)+\gamma(z)^{-1}\right)&
\tfrac{1}{2i}\left(\gamma(z)-\gamma(z)^{-1}\right)\\\\
-\tfrac{1}{2i}\left(\gamma(z)-\gamma(z)^{-1}\right)
&\tfrac{1}{2}\left(\gamma(z)+\gamma(z)^{-1}\right)
\end{pmatrix},\quad
z\in\mathbb{C}\setminus(S_\alpha\cup S_\beta),
\label{eq:Dhatout}
\end{equation}
\begin{equation}
\hat{\mathbf{D}}(z):=-\sqrt{2\pi}\mathbf{U}\left(\frac{3n}{4}\right)^{\sigma_3/6}
\gamma(z)^{\sigma_3}u_\alpha(z)^{\sigma_3/4}\sigma_2\mathbf{M}(u_\alpha(z))
\sigma_3e^{nu_\alpha(z)^{3/2}\sigma_3/2},\quad z\in S_\alpha,
\label{eq:Airybeta}
\end{equation}
and
\begin{equation}
\hat{\mathbf{D}}(z):=\sqrt{2\pi}\mathbf{U}\left(\frac{4}{3n}\right)^{\sigma_3/6}
\gamma(z)^{\sigma_3}u_\beta(z)^{-\sigma_3/4}\mathbf{M}(u_\beta(z))
e^{nu_\beta(z)^{3/2}\sigma_3/2},\quad z\in S_\beta,
\label{eq:Airyalpha}
\end{equation}
where $\mathbf{M}(u)$ is defined as follows with
$\xi:=\left(\tfrac{3n}{4} \right)^{2/3}u$.
\begin{equation}
\mathbf{M}(u):=\begin{pmatrix}
e^{-3\pi i/4}\mathrm{Ai}'\left(\xi\right) &
e^{11\pi i/12}\mathrm{Ai}'\left(\xi e^{-2\pi i/3}
\right)\\
e^{-i\pi/4}\mathrm{Ai}\left(\xi\right) &
e^{i\pi/12}\mathrm{Ai}\left(\xi e^{-2\pi i/3}
\right)\end{pmatrix},\quad -\frac{\pi}{4}<\arg(u)<\frac{3\pi}{4},
\label{eq:MI}
\end{equation}
\begin{equation}
\mathbf{M}(u):=\begin{pmatrix}
e^{-5\pi i/12}\mathrm{Ai}'\left(\xi e^{2\pi i/3}\right) &
e^{11\pi i/12}\mathrm{Ai}'\left(\xi e^{-2\pi i/3}
\right)\\
e^{-7\pi i/12}\mathrm{Ai}\left(\xi e^{2\pi i/3}\right) &
e^{i\pi/12}\mathrm{Ai}\left(\xi e^{-2\pi i/3}
\right)\end{pmatrix},\quad \frac{3\pi}{4}<\arg(u)<\pi,
\label{eq:MII}
\end{equation}
\begin{equation}
\mathbf{M}(u):=\begin{pmatrix}
e^{11\pi i/12}\mathrm{Ai}'\left(\xi e^{-2\pi i/3}\right) &
e^{7\pi i/12}\mathrm{Ai}'\left(\xi e^{2\pi i/3}
\right)\\
e^{i\pi /12}\mathrm{Ai}\left(\xi e^{-2\pi i/3}\right) &
e^{5\pi i/12}\mathrm{Ai}\left(\xi e^{2\pi i/3}
\right)\end{pmatrix},\quad -\pi<\arg(u)<-\frac{3\pi}{4},
\label{eq:MIII}
\end{equation}
\begin{equation}
\mathbf{M}(u):=\begin{pmatrix}
e^{-3\pi i/4}\mathrm{Ai}'\left(\xi\right) &
e^{7\pi i/12}\mathrm{Ai}'\left(\xi e^{2\pi i/3}
\right)\\
e^{-i\pi /4}\mathrm{Ai}\left(\xi\right) &
e^{5\pi i/12}\mathrm{Ai}\left(\xi e^{2\pi i/3}
\right)\end{pmatrix},\quad -\frac{3\pi}{4}<\arg(u)<-\frac{\pi}{4}.
\label{eq:MIV}
\end{equation}
Here $\mathrm{Ai}(\xi)$ denotes the Airy function, the unique solution
of $y''-\xi y=0$ with the asymptotic behavior
\begin{equation}
\mathrm{Ai}(\xi)=
\frac{e^{-2\xi^{3/2}/3}}{2\xi^{1/4}\sqrt{\pi}}(1+\mathcal{O}(|\xi|^{-3/2}))\quad
\text{and}\quad
\mathrm{Ai}'(\xi)=-\frac{\xi^{1/4}e^{-2\xi^{3/2}/3}}{2\sqrt{\pi}}
(1+\mathcal{O}(|\xi|^{-3/2}))
\label{eq:AiryAsymp}
\end{equation}
as $\xi\to\infty$ with $\arg(\xi)\in (-\pi,\pi)$.

\begin{remark}
  For those readers familiar with the notation of the paper
  \cite{op2} we make the following clarification.  
  The matrix $\mathbf{M}(u)$ defined here may be expressed in the form
\begin{equation}
\mathbf{M}(u)=
\frac{1}{\sqrt{2\pi}}\begin{pmatrix}
0 & -i\\1 & 0
\end{pmatrix}
\mathbf{P}\left(\left(\tfrac{3}{4}\right)^{2/3}u\right)e^{-nu^{3/2}\sigma_3/2}.
\end{equation}
where the local parametrix $\mathbf{P}(\zeta)$ is as defined in
\cite{op2}, equations (1.36)--(1.40).
\end{remark}

The point of introducing the matrix $\hat{\mathbf{D}}(z)$ is the
following.
\begin{theorem}
\label{thm:RHPExist}
Assume the conditions on the external field $V$ stated in the
Introduction.  Let $n, N \to \infty$ so that $N/n \to c$ with $0 < c <
\infty$.  Then for $n$ sufficiently large, there is a unique solution
$\dot{\mathbf{D}}(z)$ to Riemann-Hilbert Problem~\ref{rhp:Ddot}, which
possesses the following global asymptotic description:  
\begin{equation}
\dot{\mathbf{D}}(z) = \left(\mathbb{I}+\mathcal{O}\left(
\frac{1}{n\sqrt{1+|z|^2}}\right)\right)
\hat{\mathbf{D}}(z),
\label{eq:DdotDhat}
\end{equation}
uniformly with respect to $z\in\mathbb{C}$ as $n\to\infty$.  
\end{theorem}
\begin{proof}
  Let $\mathbf{F}(z):=\dot{\mathbf{D}}(z)\hat{\mathbf{D}}(z)^{-1}$.
  It is easy to see from the properties of $\dot{\mathbf{D}}(z)$
  required by the conditions of Riemann-Hilbert Problem~\ref{rhp:Ddot}
  and the explicit formulae given for the matrix $\hat{\mathbf{D}}(z)$
  in various parts of the complex plane that $\mathbf{F}(z)$ is a
  matrix that is required to have the following properties.  Firstly,
  $\mathbf{F}(z)$ is analytic at least for
  $z\in\mathbb{C}\setminus\Sigma^{\mathbf{F}}$, where
  $\Sigma^{\mathbf{F}}$ is the union of $\Sigma$ and the boundaries of
  the square regions $S_\alpha$ and $S_\beta$, and
  $\mathbf{F}(z)$ takes continuous boundary values on
  $\Sigma^{\mathbf{F}}$.  Secondly, the boundary values satisfy
  $\mathbf{F}_+(z)=\mathbf{F}_-(z)\mathbf{V}_{\mathbf{F}}(z)$ for some
  jump matrix function $\mathbf{V}_{\mathbf{F}}(z)$ defined on
  $\Sigma^{\mathbf{F}}$ that is explicitly calculable in terms of
  $\mathbf{V}_{\mathbf{D}}(z)$ and the boundary values taken on
  $\Sigma^{\mathbf{F}}$ by $\hat{\mathbf{D}}(z)$.  Thirdly,
  $\mathbf{F}(z)$ must tend to the identity matrix as $z\to\infty$.
  In other words, these three facts show that $\mathbf{F}(z)$
  satisfies its own Riemann-Hilbert problem.

  The Riemann-Hilbert problem satisfied by $\mathbf{F}(z)$ is of a
  particularly convenient type: it is a ``small-norm'' problem in the
  sense that the jump matrix $\mathbf{V}_{\mathbf{F}}(z)$ is a small
  perturbation of the identity matrix in a suitable space of
  matrix-valued functions on the contour $\Sigma^{\mathbf{F}}$.  In
  fact, it is easy to check by direct calculation that
  $\mathbf{V}_{\mathbf{F}}(z)\equiv\mathbb{I}$ for
  $\Sigma^{\mathbf{F}}\cap (S_\alpha\cup S_\beta)$.
  This is a direct consequence of Property 3 in
  Lemma~\ref{lem:thetaextendgeneral} and
  Lemma~\ref{lem:phiextendgeneral} characterizing respectively the
  extensions $\Theta(x,y)$ and $\Phi(x,y)$ on these portions of the
  contour $\Sigma$, and of the identity 
\begin{equation}
\mathrm{Ai}(\xi) + e^{-2\pi i/3}\mathrm{Ai}(\xi e^{-2\pi i/3}) +
e^{2\pi i/3}\mathrm{Ai}(\xi e^{2\pi i/3})\equiv 0.
\label{eq:AiryIdentity}
\end{equation}
An even easier calculation shows that
$\mathbf{V}_{\mathbf{F}}(z)\equiv\mathbb{I}$ for
$\alpha+\delta<z<\beta-\delta$.  With the use of the asymptotic
formulae \eqref{eq:AiryAsymp}, one sees that on the boundaries of the
two squares $S_\alpha$ and $S_\beta$,
$\mathbf{V}_{\mathbf{F}}(z)-\mathbb{I}$ is uniformly
$\mathcal{O}(n^{-1})$, and on all remaining parts of
$\Sigma^{\mathbf{F}}$ one finds (in part by the estimates
\eqref{eq:thetaPbeh}--\eqref{eq:thetaMbeh} on
$\mathrm{Im}(\Theta(x,y))$ in Property 2 of
Lemma~\ref{lem:thetaextendgeneral} and the estimate
\eqref{eq:RePhiraw} on $\mathrm{Re}(\Phi(x,y))$ in Property 2 of
Lemma~\ref{lem:phiextendgeneral}) that
$\mathbf{V}_{\mathbf{F}}(z)-\mathbb{I}$ is uniformly exponentially
small as $n\to\infty$ and also decays rapidly as $z\to\infty$.

Since for our purposes we need to control the size of
$\mathbf{F}(z)-\mathbb{I}$ right up to the contour
$\Sigma^{\mathbf{F}}$, we need to formulate the Riemann-Hilbert
problem for $\mathbf{F}(z)$ in an appropriate space in which the
boundary values of $\mathbf{F}(z)$ are H\"older continuous with some
exponent $\alpha\in (0,1]$.  To do this we need to observe that as a
consequence of our assumptions on the external field $V$ and the
corresponding smoothness of $\Theta(x,y)$ and $\Phi(x,y)$ described in
Property 1 of Lemma~\ref{lem:thetaextendgeneral} and
Lemma~\ref{lem:phiextendgeneral}, and also as a consequence of the
piecewise analyticity of the comparison matrix $\hat{\mathbf{D}}(z)$,
the jump matrix $\mathbf{V}_{\mathbf{F}}(z)$ is sufficiently smooth on
a sufficiently (piecewise) smooth contour that the H\"older version of
the small-norm theory applies.  The result is that as $n\to\infty$,
$\mathbf{F}(z)$ exists uniquely in the space of matrices with
H\"older-continuous boundary values, and also
$\mathbf{F}(z)-\mathbb{I}=\mathcal{O}(n^{-1})$ holds uniformly
throughout the complex $z$-plane.

  A detailed account of the existence theory for small-norm
  Riemann-Hilbert problems is discussed, for example, in \cite{op1}.
  Specific information relevant to the application of small-norm
  theory in H\"older spaces can be found in Appendix A of \cite{SSE}.
\end{proof}

An important property of $\hat{\mathbf{D}}(z)$ is that for all $z$
where it is defined, $\det(\hat{\mathbf{D}}(z))\equiv 1$.  Therefore,
$\hat{\mathbf{D}}(z)$ and its inverse $\hat{\mathbf{D}}(z)^{-1}$ have
comparable bounds in any matrix norm.  From \eqref{eq:Dhatout} one may
then see that $\hat{\mathbf{D}}(z)$ and its inverse are bounded as
$n\to\infty$ uniformly for $z\in\mathbb{C}\setminus(S_\alpha\cup
S_\beta)$.  On the other hand, from \eqref{eq:Airyalpha} and
\eqref{eq:Airybeta} together with the definition
\eqref{eq:MI}--\eqref{eq:MIV} of $\mathbf{M}(u)$ one sees that 
$\hat{\mathbf{D}}(z)=\mathcal{O}(n^{1/6})$ and $\hat{\mathbf{D}}(z)^{-1}=
\mathcal{O}(n^{1/6})$ hold for $z\in S_\alpha\cup S_\beta$, although
for our purposes a more useful estimate coming from the same formulae
is that
\begin{equation}
\hat{\mathbf{D}}(z)=\mathcal{O}(|z-\alpha|^{-1/4}|z-\beta|^{-1/4})\quad
\text{and}\quad
\hat{\mathbf{D}}(z)^{-1}=\mathcal{O}(|z-\alpha|^{-1/4}|z-\beta|^{-1/4})
\label{eq:Dhatgrow}
\end{equation}
holds uniformly for $z$ in bounded sets (the constants implicit in the order
relations are independent of both $n$ and $z$).

The construction of $\hat{\mathbf{D}}(z)$ is one part of the argument
where the details are somewhat different for $G>0$ than for $G=0$.  To
handle the case with more than one interval of support one must
replace the definition of $\hat{\mathbf{D}}(z)$ for
$z\in\mathbb{C}\setminus(S_\alpha\cup S_\beta)$
with a matrix constructed from Riemann theta functions for
hyperelliptic curves of nonzero genus modeled by two copies of the
complex $z$-plane identified along cuts made on the real axis in the
support intervals of the equilibrium measure $\mu_*$.  Full details
may be found, for example, in \cite{op2}.  The key property of uniform
boundedness of $\hat{\mathbf{D}}(z)$ away from the endpoints of the
support intervals remains valid in this more general case.

\section{A $\dbar$ Problem and Existence Theorem}
\label{sec:DBARExist}
Having constructed $\dot{\mathbf{D}}(x+iy)$, we now define
$\mathbf{E}(x,y)$ via
\begin{equation}
\mathbf{E}(x,y)  := \mathbf{D}(x,y) \dot{\mathbf{D}}(x+iy)^{-1}.
\label{eq:EfromD}
\end{equation}
It is immediately clear that $\mathbf{E}(x,y)$ is continuous in
$\mathbb{C}$.  It is straightforward to compute the $\dbar$ derivative
of $\mathbf{E}(x,y)$, and we learn that $\mathbf{E}(x,y)$ solves
$\dbar$ Problem~\ref{dbp:01} below.  We define a ``dressed'' version
of the matrix $\mathbf{W}_0(x,y)$ as follows:
\begin{equation}
\label{eq:calWdef1}
\mathbf{W}(x,y) :=  \dot{\mathbf{D}}(x+iy)
\mathbf{W}_{0}(x,y)\dot{\mathbf{D}}(x+iy)^{-1}, \quad (x,y) \in \mathbb{R}^2.
\end{equation}

\begin{dbp}
\label{dbp:01}
Find a $2\times 2$ matrix $\mathbf{E}(x,y)$ with the properties:
\begin{itemize}
\item[]\textbf{Continuity.}  $\mathbf{E}(x,y)$ is a continuous function
  of $x$ and $y$ for $(x,y)\in\mathbb{R}^2$.
\item[]\textbf{Deviation From Analyticity.}  For $(x,y)\in\mathbb{R}^2$,
\begin{equation}
\dbar\mathbf{E}(x,y) = \mathbf{E}(x,y)\mathbf{W}(x,y). 
\label{eq:dbarE}
\end{equation}
(Note that in particular $\dbar\mathbf{E}(x,y)=\mathbf{0}$ for
$(x,y)\not\in\Omega_+\cup\Omega_-\cup\Omega_\alpha\cup\Omega_\beta$.)
\item[]\textbf{Normalization.}  The matrix $\mathbf{E}(x,y)$ is
  normalized as follows:
\begin{equation}
\lim_{x,y\rightarrow\infty}\mathbf{E}(x,y)=\mathbb{I}.
\label{eq:ENorm}
\end{equation}
\end{itemize}
\end{dbp}

In view of the normalization condition \eqref{eq:ENorm} and the fact
that $\mathbf{W}(x,y)\equiv \mathbf{0}$ outside some compact set, we may
invert the $\dbar$ operator in \eqref{eq:dbarE} with the help of the
Cauchy kernel:
\begin{equation}
\mathbf{E}(x,y)=\mathbb{I}+\mathcal{K}\mathbf{E}(x,y),
\label{eq:Eintegralequation}
\end{equation}
where
\begin{equation}
\mathcal{K}\mathbf{E}(x,y):=-\frac{1}{\pi}\iint_{\mathbb{R}^2}
((u+iv)-(x+iy))^{-1}\mathbf{E}(u,v)\mathbf{W}(u,v)\,du\,dv.
\end{equation}  
It is a basic fact that if $\mathbf{E}(x,y)$ satisfies the integral
equation \eqref{eq:Eintegralequation} then $\mathbf{E}(x,y)$ also
solves $\dbar$ Problem~\ref{dbp:01}; in fact the integral equation
\eqref{eq:Eintegralequation} is equivalent to $\dbar$
Problem~\ref{dbp:01}.

In this section we will show that the integral operator $\mathcal{K}$,
when considered in the space $L^{\infty}(\mathbb{R}^{2})$, has norm
bounded by $C n^{-1/3}\log(n)$ for some $C>0$.  This implies that the
integral equation \eqref{eq:Eintegralequation} may be solved by
Neumann series.

The strategy to prove this is quite straightforward: because the
singularity of the Cauchy kernel is integrable in $\mathbb{R}^{2}$,
the basic estimate is:
\begin{equation}
  \| \mathcal{K} \mathbf{H} \|_{L^{\infty}(\mathbb{R}^{2})} \le \|\mathbf{H}
  \|_{L^{\infty}(\mathbb{R}^{2})} \sup_{(x,y)\in\mathbb{R}^2}\left[
\frac{1}{\pi}\iint_{\mathbb{R}^2}
  \frac{\|\mathbf{W}(u,v)\|}{\left| ((u+iv)-(x+iy))\right|}\,du\,dv\right],
\label{eq:KOpNorm}
\end{equation}
where $\|\mathbf{W}(u,v)\|$ is a pointwise matrix norm, i.e. a norm of
the matrix $\mathbf{W}$ evaluated at $(u,v)$.  Since $\mathbf{W}$
is uniformly bounded and has compact support, this immediately implies
that $\mathcal{K}$ is a bounded operator on
$L^{\infty}(\mathbb{R}^{2})$.  The goal is then to prove that the
integral appearing on the right hand side of \eqref{eq:KOpNorm} is
small, and for this the $(u,v)$ dependence of $\|\mathbf{W}(u,v) \|$
will be essential.

\begin{theorem}
  There is a unique solution $\mathbf{E}$ to $\dbar$
  Problem~\ref{dbp:01}, which possesses the following uniform
  asymptotic description, valid for all $(x,y) \in \mathbb{R}^2$:
\begin{equation}
\mathbf{E}(x,y) =  \mathbb{I} 
+\mathcal{O}\left(\frac{\log(n)}{n^{1/3}\sqrt{1+x^2+y^2}}\right),
\label{eq:Eestimate}
\end{equation}
where the constant implicit in the order notation is independent of
$n$ and $z$ and depends only on the external field $V$ and the
constant $c$.
\end{theorem}
\begin{proof}
  We begin by describing the asymptotic behavior of $\mathbf{W}(x,y)$.
  According to \eqref{eq:calWdef1}, $\mathbf{W}(x,y)$ is obtained from
  $\mathbf{W}_0(x,y)$ by conjugation, so we start by making the
  following two observations about $\mathbf{W}_0(x,y)$. Firstly, 
\begin{equation}
\begin{split}
  \mathrm{supp}(\mathbf{W}_0)&=
  \overline{\Omega_+\cup\Omega_-\cup\Omega_\alpha\cup\Omega_\beta}\\
  &\subset B:=\{(x,y)\in\mathbb{R}^2:\quad \alpha-2\delta\le x\le \beta+2\delta,\,\,\, |y|\le\delta\}.
\end{split}
\end{equation}
Secondly, as a consequence of the
  definition of $\mathbf{W}_0(x,y)$ given in
  \eqref{eq:W0defplus}--\eqref{eq:W0defbeta} and Property 2 in
  Lemma~\ref{lem:thetaextendgeneral} and
  Lemma~\ref{lem:phiextendgeneral} describing $\dbar\Theta(x,y)$,
  $\mathrm{Im}(\Theta(x,y))$, $\dbar\Phi(x,y)$, and
  $\mathrm{Re}(\Phi(x,y))$ for $(x,y)$ in the support of
  $\mathbf{W}_0$, we may assume that
\begin{equation}
\left\|\mathbf{W}_0(u,v)\right\|\le Kne^{-kn|v|^{3/2}}
|v||w-\alpha|^{1/2}|w-\beta|^{1/2},\quad w=u+iv\in B,
\end{equation}
for some constants $K>0$ and $k>0$.  
Using the definition \eqref{eq:calWdef1} of
$\mathbf{W}(x,y)$ in terms of $\mathbf{W}_0(x,y)$,
Theorem~\ref{thm:RHPExist} together with \eqref{eq:Dhatgrow} implies
that 
\begin{equation}
  \left\|\mathbf{W}(u,v)\right\|\le Kn|v|e^{-kn|v|^{3/2}},\quad w=u+iv\in B.
\label{eq:WEstimate}
\end{equation}
Of course $\|\mathbf{W}(u,v)\|\equiv 0$ for $u+iv\not\in B$.
Therefore, we have
\begin{equation}
\begin{split}
\iint_{\mathbb{R}^2}\frac{\|\mathbf{W}(u,v)\|}{|(u+iv)-(x+iy)|}\,du\,dv
&\le Kn\iint_B\frac{|v|e^{-kn|v|^{3/2}}}{|(u+iv)-(x+iy)|}\,du\,dv \\ &=
Kn\int_{-\delta}^{\delta}|v|e^{-kn|v|^{3/2}}\int_{\alpha-2\delta}^{\beta+2\delta}
\frac{du}{\sqrt{(u-x)^2+(v-y)^2}}\,dv.
\end{split}
\label{eq:keyestimate1}
\end{equation}

Now we will show that there is a constant $C>0$ such that 
\begin{equation}
I(x,y,v):=\int_{\alpha-2\delta}^{\beta+2\delta}\frac{du}{\sqrt{(u-x)^2+(v-y)^2}}\le
C\log\left(1+\frac{1}{|v-y|}\right),\quad (x,y)\in\mathbb{R}^2,\quad
v\in\mathbb{R}.
\label{eq:innerintegralestimate}
\end{equation}
Indeed, since $|(u+iv)-(x+iy)|\ge |v-y|$, on the one hand we have
\begin{equation} 
I(x,y,v)\le\int_{\alpha-2\delta}^{\beta+2\delta}\frac{du}{|v-y|}=\frac{\beta-\alpha+4\delta}{|v-y|}=\frac{C_1}{|v-y|}.
\label{eq:vmybig}
\end{equation}
We will use this estimate when $|v-y|$ is large.  On the other hand,
for $|v-y|$ small we have the following.  Firstly, since
$|(u+iv)-(x+iy)|\ge |u-x|$,
\begin{equation}
\text{$x\le\alpha-2\delta-1$ or $x\ge \beta+2\delta+1$}
\quad\implies\quad 
I(x,y,v)\le\int_{\alpha-2\delta}^{\beta+2\delta}\frac{du}{|u-x|}\le
\int_{\alpha-2\delta}^{\beta+2\delta}du = \beta-\alpha+4\delta.
\label{eq:otherhandfirst}
\end{equation}
Secondly, we have may evaluate $I(x,y,v)$ explicitly as
\begin{equation}
I(x,y,v)=\mathrm{arcsinh}\left(\frac{\beta+2\delta-x}{|v-y|}\right)-
\mathrm{arcsinh}\left(\frac{\alpha-2\delta-x}{|v-y|}\right),
\end{equation}
from which it follows that 
\begin{equation}
\alpha-2\delta-1\le x\le\alpha-2\delta<\beta+2\delta\quad\implies\quad
I(x,y,v)\le\mathrm{arcsinh}\left(\frac{\beta+2\delta-x}{|v-y|}\right)
\le\mathrm{arcsinh}\left(\frac{\beta-\alpha+4\delta+1}{|v-y|}\right),
\label{eq:otherhandsecond}
\end{equation}
\begin{equation}
\begin{split}
\alpha-2\delta\le x\le \beta+2\delta\quad\implies\quad
I(x,y,v)&=\mathrm{arcsinh}\left(\frac{\beta+2\delta-x}{|v-y|}\right)+
\mathrm{arcsinh}\left(\frac{x-\alpha+2\delta}{|v-y|}\right)\\ &\le
2\,\mathrm{arcsinh}\left(\frac{\beta-\alpha+4\delta}{|v-y|}\right),
\end{split}
\end{equation}
and
\begin{equation}
\alpha-2\delta<\beta+2\delta\le x\le\beta+2\delta+1\quad\implies\quad
I(x,y,v)\le\mathrm{arcsinh}\left(\frac{x-\alpha+2\delta}{|v-y|}\right)\le
\mathrm{arcsinh}\left(\frac{\beta-\alpha+2\delta+1}{|v-y|}\right).
\label{eq:otherhandlast}
\end{equation}
The estimates \eqref{eq:otherhandfirst} and
\eqref{eq:otherhandsecond}--\eqref{eq:otherhandlast} may be combined
for $|v-y|$ sufficiently small to give
\begin{equation}
I(x,y,v)\le C_2\log\left(\frac{1}{|v-y|}\right).
\label{eq:vmysmall}
\end{equation}
This estimate is useful for $|v-y|$ small.  Taking \eqref{eq:vmysmall} with
\eqref{eq:vmybig} yields \eqref{eq:innerintegralestimate}.

Using \eqref{eq:innerintegralestimate} in \eqref{eq:keyestimate1} and
extending the integration from $|v|\le\delta$ to $v\in\mathbb{R}$
yields
\begin{equation}
\iint_{\mathbb{R}^2}\frac{\|\mathbf{W}(u,v)\|}{|(u+iv)-(x+iy)|}\,du\,dv
\le Kn\int_{-\infty}^{+\infty}|v|e^{-kn|v|^{3/2}}\log\left(1+\frac{1}{|v-y|}
\right)\,dv
\end{equation}
for some modified positive constant $K>0$ independent of
$(x,y)\in\mathbb{R}^2$ and $n$.  Now rescaling the integration
variable by $v=n^{-2/3}s$ gives
\begin{equation}
\begin{split}
n\int_{-\infty}^{+\infty}|v|e^{-kn|v|^{3/2}}\log\left(1+\frac{1}{|v-y|}\right)\,dv &
=
n^{-1/3}\int_{-\infty}^{+\infty}|s|e^{-k|s|^{3/2}}\log\left(1+\frac{n^{2/3}}{|s-n^{2/3}y|}\right)\,ds\\
&\le n^{-1/3}\int_{-\infty}^{+\infty}|s|e^{-k|s|^{3/2}}\log\left(n^{2/3}+\frac{n^{2/3}}{|s-n^{2/3}y|}\right)\,ds\\
&= \frac{2}{3}n^{-1/3}\log(n)\int_{-\infty}^{+\infty}|s|e^{-k|s|^{3/2}}\,ds \\
&\quad\quad\quad\quad{}+
n^{-1/3}\int_{-\infty}^{+\infty}|s|e^{-k|s|^{3/2}}\log\left(1+\frac{1}{|s-n^{2/3}y|}
\right)\,ds.
\end{split}
\end{equation}
The first integral is clearly finite and independent of $n$, and since by
Cauchy-Schwarz,
\begin{equation}
\int_{-\infty}^{+\infty}|s|e^{-k|s|^{3/2}}\log\left(1+\frac{1}{|s-n^{2/3}y|}
\right)\,ds\le\left[\int_{-\infty}^{+\infty}
s^2e^{-2k|s|^{3/2}}\,ds\right]^{1/2}
\left[\int_{-\infty}^{+\infty}\log\left(1+\frac{1}{|s|}\right)^2\,ds\right]^{1/2}
\end{equation}
the second integral is bounded by a finite quantity independent of $n$.
This proves that
\begin{equation}
\sup_{(x,y)\in\mathbb{R}^2}\iint_{\mathbb{R}^2}\frac{\|\mathbf{W}(u,v)\|}
{|(u+iv)-(x+iy)|}\,du\,dv\le Cn^{-1/3}\log(n)
\end{equation}
holds for some constant $C>0$ independent of $n$, for $n$ sufficiently large.

The Neumann series for the integral equation
\eqref{eq:Eintegralequation} corresponding to $\dbar$
Problem~\ref{dbp:01} therefore converges in $L^\infty(\mathbb{R}^2)$
for sufficiently large $n$, and the estimate \eqref{eq:Eestimate}
follows immediately.
\end{proof}

\section{Large-$n$ Asymptotics for $\mathbf{A}(z)$ and the Orthogonal
  Polynomials}
We will restrict our attention to $\Omega_{+}$ in ``the bulk'' (i.e.\@
away from the endpoints $\alpha$ and $\beta$) and also to the upper
half-plane in the vicinity of the endpoint $\beta$.  Considerations
for $z$ near $\alpha$ are nearly identical, and we will omit them for
the sake of brevity.  Since the orthogonal polynomials being
considered have real coefficients, the asymptotic behavior for $z$ in
the lower half-plane, $\mathbb{C}_{-}$, may always be obtained by
complex conjugation.  In this section and the next section we will use
the notation
\begin{equation}
\Delta_n:=n^{-1/3}\log(n).
\label{eq:deltadef}
\end{equation}

\subsection{Asymptotics of the leading coefficients}
The leading coefficients $\kappa_{n-1,n-1}>0$ and $\kappa_{n,n}>0$ are
obtained from $\mathbf{A}(z)$ using \eqref{eq:kappas}.  In a neighborhood
of $z=\infty$, the matrix $\mathbf{A}(z)$ is related to $\mathbf{E}(x,y)$
via
\begin{equation}
\mathbf{A}(z)=e^{n\ell\sigma_3/2}\mathbf{E}(x,y)\dot{\mathbf{D}}(z)e^{ng(z)\sigma_3}
e^{-n\ell\sigma_3/2}.
\end{equation}
Using \eqref{eq:Dhatout}, \eqref{eq:DdotDhat}, and \eqref{eq:Eestimate}, we
get
\begin{equation}
\begin{split}
A_{12}(z)&=\frac{1}{2i}e^{n(\ell-g(z))}\left(\gamma(z)-\gamma(z)^{-1}+\mathcal{O}\left(\frac{\Delta_n}{z}\right)\right)\\
A_{21}(z)&=-\frac{1}{2i}e^{n(g(z)-\ell)}\left(\gamma(z)-\gamma(z)^{-1}+
\mathcal{O}\left(\frac{\Delta_n}{z}\right)\right),
\end{split}
\end{equation}
where the error terms are valid for $z$ near $z=\infty$.  Now, from 
\eqref{eq:gdef} we have
\begin{equation}
g(z)=\log(z)+\mathcal{O}(z^{-1}),\quad z\to\infty,
\end{equation}
and from \eqref{eq:gamma} and the condition that $\gamma(z)\to 1$ as
$z\to\infty$ we easily obtain
\begin{equation}
\gamma(z)=1-\frac{1}{4z}(\beta-\alpha)+\mathcal{O}(z^{-2}),\quad z\to\infty.
\end{equation}
Therefore, using \eqref{eq:kappas}, we obtain
\begin{equation}
\kappa_{n-1,n-1}^2 = \frac{\beta-\alpha}{8\pi}e^{-n\ell}\left(1+\mathcal{O}(\Delta_n)\right)\quad\text{and}\quad
\kappa_{n,n}^2 = \frac{2}{(\beta-\alpha)\pi}e^{-n\ell}\left(1+\mathcal{O}(\Delta_n)\right)
\label{eq:kappaasymp}
\end{equation}
as $n\to\infty$.  These formulae may be combined with the asymptotic
formulae for $A_{11}(z)$ and $A_{22}(z)$ to be given below to obtain
asymptotic formulae for the orthonormal polynomials $p_{n-1}(z)$ and 
$p_n(z)$.

\subsection{Asymptotics of the orthogonal polynomials in the bulk}
Within the set $\Omega_{+}$, the solution $\mathbf{A}(z)$ to
Riemann-Hilbert Problem~\ref{rhp:A} is related to $\mathbf{E}(x,y)$ via
\begin{equation}
\mathbf{A}(z) = e^{n\ell\sigma_{3}/2} \mathbf{E}(x,y)\dot{\mathbf{D}}(z) 
\begin{pmatrix}
1 & 0 \\e^{-in\Theta(x,y)} & 1\end{pmatrix}e^{ng(z)\sigma_3}e^{-n\ell\sigma_3/2},
\quad z=x+iy\in\Omega_+,
\label{eq:AEDdot}
\end{equation}
as follows from \eqref{eq:BfromA}, \eqref{eq:DfromBOmegaplus}, and
\eqref{eq:EfromD}.
According to \eqref{eq:Acontents}, the first column of $\mathbf{A}(z)$
contains the orthogonal polynomials:
\begin{equation}
A_{11}(z)=\frac{1}{\kappa_{n,n}}p_n(z),\quad\text{and}\quad
A_{21}(z)=-\frac{2\pi i}{\kappa_{n-1,n-1}}p_{n-1}(z).
\label{eq:Afirstcolumn}
\end{equation}
Via \eqref{eq:AEDdot} these may be expressed as follows:
\begin{align}
A_{11}(z) &=  
\left[E_{11}(x,y)\left(\dot{D}_{11}(z) + e^{-in\Theta(x,y)}
\dot{D}_{12}(z) \right)\right.\\
\nonumber &\quad\quad\quad\left.{}+ E_{12}(x,y) \left( 
\dot{D}_{21}(z) + e^{-in\Theta(x,y)} \dot{D}_{22}(z) \right)
\right] e^{ng(z)},\\
A_{21}(z) &=  
\left[E_{21}(x,y)\left(\dot{D}_{11}(z) + e^{-in\Theta(x,y)} 
\dot{D}_{12}(z) \right)\right.\\ 
\nonumber &\quad\quad\quad\left.{}+ E_{22}(x,y) \left( 
\dot{D}_{21}(z) + e^{-in\Theta(x,y)}\dot{D}_{22}(z) \right)
\right] e^{n(g(z)  - \ell )}.
\end{align}

Using the asymptotic estimate \eqref{eq:Eestimate} of
$\mathbf{E}(x,y)-\mathbb{I}$, the relation \eqref{eq:DdotDhat} between
$\dot{\mathbf{D}}(z)$ and $\hat{\mathbf{D}}(z)$, and the
explicit formula \eqref{eq:Dhatout} for $\hat{\mathbf{D}}(z)$ valid for
$z\in\Omega_+$, straightforward manipulations yield
\begin{align}
A_{11}(z)&=e^{n(g(z)-i\Theta(x,y)/2)}a(z)
\left[\cos\left(\frac{1}{2}\left(n\Theta(x,y)-\varphi(z)\right)\right)
\left(1+\mathcal{O}\left(\Delta_n\right)\right)\right.
\label{eq:polyas1}\\
\nonumber &\quad\quad\quad\left.{}+
\sin\left(\frac{1}{2}\left(n\Theta(x,y)+\varphi(z)\right)\right)\mathcal{O}\left(\Delta_n\right)\right]\\
A_{21}(z)&=
-ie^{n(g(z)-\ell-i\Theta(x,y)/2)}a(z)
\left[\sin\left(\frac{1}{2}\left(n\Theta(x,y)+\varphi(z)\right)\right)
\left(1+\mathcal{O}\left(\Delta_n\right)\right)\right.
\label{eq:polyas2}\\
\nonumber &\quad\quad\quad\left.{}+
\cos\left(\frac{1}{2}\left(n\Theta(x,y)-\varphi(z)\right)\right)
\mathcal{O}\left(\Delta_n\right)\right],
\end{align}
where
\begin{equation}
a(z):=\frac{\sqrt{\beta-\alpha}}{(z-\alpha)^{1/4}(\beta-z)^{1/4}}\quad
\text{and}\quad
\varphi(z):=\arcsin\left(\frac{2z-(\alpha+\beta)}{\beta-\alpha}\right)
\end{equation}
are both functions analytic for
$z\in\mathbb{C}\setminus(\mathbb{R}\setminus[\alpha,\beta])$.  One
apparent difficulty with these asymptotic formulae is that they
involve an extension $\Theta(x,y)$ of $\theta(x)$ that is completely
arbitrary except that it must satisfy Properties 1--3 of
Lemma~\ref{lem:thetaextendgeneral}.  (While our proof of
Lemma~\ref{lem:thetaextendgeneral} was by construction, there was no
assertion of uniqueness, and indeed there are many extensions
$\Theta(x,y)$ having the required properties.)  On the other hand, it
is also easy to see that the differences between various extensions
$\Theta(x,y)$ may be absorbed into the error terms.  For example, if
we fix $x\in (\alpha,\beta)$ and fix $y>0$ sufficiently small so as to
be in the region $\Omega_+$, then by Property~2 of
Lemma~\ref{lem:thetaextendgeneral} we have $\mathrm{Im}(\Theta(x,y))\le -k y$
for some $k>0$ (here we are using the fact that $x$ is bounded away
from $\alpha$ and $\beta$), so \eqref{eq:polyas1} and \eqref{eq:polyas2}
become
\begin{align}
A_{11}(z)= \frac{1}{2}e^{ng(z)}a(z)e^{-i\varphi(z)/2}\left(1+\mathcal{O}\left(
\Delta_n\right)\right),\quad y\gg n^{-1},\label{eq:polyas1a}\\
\intertext{and}
A_{21}(z)=-\frac{1}{2}e^{n(g(z)-\ell)}a(z)e^{i\varphi(z)/2}\left(1+\mathcal{O}
\left(\Delta_n\right)\right),\quad y\gg n^{-1}.\label{eq:polyas2a}
\end{align}
On the other hand, if we suppose that $x\in(\alpha,\beta)$ is fixed and
$y=\mathcal{O}(n^{-1})$, then using Property 1 of
Lemma~\ref{lem:thetaextendgeneral} and the Mean Value Theorem we have
\begin{equation}
\Theta(x,y)=\theta(x)+\Theta_y(x,\xi_1)y,\quad 0\le \xi_1\le y
\end{equation}
and using \eqref{eq:CR} to eliminate $\Theta_y$ in terms of $\Theta_x$
and $\dbar\Theta$ this becomes
\begin{equation}
\Theta(x,y)=\theta(x)+i\Theta_x(x,\xi_1)y
-2i\dbar\Theta(x,\xi_1)y,\quad 0\le\xi_1\le y.
\end{equation}
Since $x$ is bounded away from $\alpha$ and $\beta$, Property 1 of
Lemma~\ref{lem:thetaextendgeneral} guarantees further that $\Theta_{xy}$ is
continuous, so another application of the Mean Value Theorem gives
\begin{equation}
\begin{split}
\Theta_x(x,\xi)y &= \Theta_x(x,0)y +\Theta_{xy}(x,\xi_2)y\xi_1\\
 &=\theta'(x)y + \Theta_{xy}(x,\xi_2)y\xi_1.
\end{split}
\end{equation}
Finally, using Property 2 of Lemma~\ref{lem:thetaextendgeneral} to control
$\dbar\Theta$ and using the assumption that $y=\mathcal{O}(n^{-1})$ we find
\begin{equation}
\Theta(x,y)=\theta(x)+i\theta'(x)y + \mathcal{O}(n^{-2}).
\end{equation}
It follows that \eqref{eq:polyas1} and \eqref{eq:polyas2} become
\begin{align}
A_{11}(z)&=e^{n(g(z)-i\theta(x)/2+\theta'(x)y/2)}a(x)\left[
\cos\left(\frac{1}{2}(n\theta(x)+in\theta'(x)y-\varphi(x))\right)
+\mathcal{O}\left(\Delta_n\right)\right]\label{eq:polyas1b}\\
\intertext{and}
A_{21}(z)&=-ie^{n(g(z)-\ell-i\theta(x)/2+\theta'(x)y/2)}a(x)
\left[\sin\left(\frac{1}{2}(n\theta(x)+in\theta'(x)y+\varphi(x))\right)
+\mathcal{O}\left(\Delta_n\right)\right].\label{eq:polyas2b}
\end{align}
In particular, for $y=0+$ we can write these in the form
\begin{align}
\label{eq:A11axis}
A_{11}(x)&=e^{n(cV(x)+\ell-\phi(x))/2}a(x)\cos\left(\frac{1}{2}
\left(n\theta(x)-\varphi(x)\right)\right) + 
\mathcal{O}\left(\Delta_ne^{n(cV(x)+\ell-\phi(x))/2}\right)\\
\intertext{and}
\label{eq:A21axis}
A_{21}(x)&=-ie^{n(cV(x)-\ell-\phi(x))/2}a(x)\sin\left(\frac{1}{2}
\left(n\theta(x)-\varphi(x)\right)\right) + 
\mathcal{O}\left(\Delta_ne^{n(cV(x)-\ell-\phi(x))/2}\right).
\end{align}

\subsection{Asymptotics of the orthogonal polynomials at the edge}
Next, suppose that $z\in S_\beta^\mathrm{II}$ (see
Figure~\ref{fig:Squares}).  Then similar calculations in which the
formulae \eqref{eq:Airybeta} and \eqref{eq:MII} are used to find
$\hat{\mathbf{D}}(z)$ yield
\begin{align}
\label{eq:AiryAsTri1}
A_{11}(z)&=
\sqrt{\pi}e^{ng(z)}\left[F_n^1(z)
\left(1+\mathcal{O}\left(\Delta_n\right)\right)
+F_n^2(z)\mathcal{O}\left(\Delta_n\right)\right],\\
\label{eq:AiryAsTri2}
A_{21}(z)&=
\sqrt{\pi}e^{n(g(z)-\ell)}\left[F_n^2(z)
\left(1+\mathcal{O}\left(\Delta_n\right)\right)+
F_n^1(z)\mathcal{O}\left(\Delta_n\right)\right],
\end{align}
where
\begin{align}
\nonumber F_n^1(z)&:=n^{1/6}w(z)\left[e^{-i\pi/3}e^{nu_\beta(z)^{3/2}/2}\mathrm{Ai}
\left(\left(\tfrac{3n}{4}\right)^{2/3}u_\beta(z)e^{2\pi i/3}\right)\right.\\
&\quad\quad\quad\quad\quad\quad\left.{}+
e^{i\pi/3}e^{-nu_\beta(z)^{3/2}/2}e^{-in\Theta(x,y)}
\mathrm{Ai}\left(\left(\tfrac{3n}{4}\right)^{2/3}u_\beta(z)e^{-2\pi i/3}\right)
\right]\\
\nonumber &\quad\quad{}+
n^{-1/6}w(z)^{-1}\left[e^{-2\pi i/3}e^{nu_\beta(z)^{3/2}/2}\mathrm{Ai}'
\left(\left(\tfrac{3n}{4}\right)^{2/3}u_\beta(z)e^{2\pi i/3}\right)\right.\\
\nonumber &\quad\quad\quad\quad\quad\quad\left.{}+
e^{2\pi i/3}e^{-nu_\beta(z)^{3/2}/2}e^{-in\Theta(x,y)}\mathrm{Ai}'
\left(\left(\tfrac{3n}{4}\right)^{2/3}u_\beta(z)e^{-2\pi i/3}\right)\right],\\
\nonumber\\
\nonumber F_n^2(z)&:=n^{1/6}w(z)\left[e^{-5\pi i/6}e^{nu_\beta(z)^{3/2}/2}
\mathrm{Ai}\left(\left(\tfrac{3n}{4}\right)^{2/3}u_\beta(z)e^{2\pi i/3}\right)
\right.\\
&\quad\quad\quad\quad\quad\quad\left.{}+
e^{-i\pi/6}e^{-nu_\beta(z)^{3/2}/2}e^{-in\Theta(x,y)}
\mathrm{Ai}\left(\left(\tfrac{3n}{4}\right)^{2/3}u_\beta(z)e^{-2\pi i/3}\right)
\right]\\
\nonumber&\quad\quad{}+
n^{-1/6}w(z)^{-1}\left[e^{-i\pi/6}e^{nu_\beta(z)^{3/2}/2}
\mathrm{Ai}'\left(\left(\tfrac{3n}{4}\right)^{2/3}u_\beta(z)e^{2\pi i/3}\right)
\right.\\
\nonumber &\quad\quad\quad\quad\quad\quad\left.{}+
e^{-5\pi i/6}e^{-nu_\beta(z)^{3/2}/2}e^{-in\Theta(x,y)}
\mathrm{Ai}'\left(\left(\tfrac{3n}{4}\right)^{2/3}u_\beta(z)e^{-2\pi i/3}\right)
\right],
\end{align}
and
\begin{equation}
w(z):=\left(\frac{3}{4}\right)^{1/6}[-h_\beta'(\beta)]^{1/6}(z-\alpha)^{1/4}.
\end{equation}

Now it will also be useful to have the asymptotic behavior of
$A_{11}(z)$ and $A_{21}(z)$ for $z\in
S_\beta^\mathrm{I}\cap\mathbb{C}_+$ (see Figure~\ref{fig:Squares}),
and for this purpose we note that for such $z$ we have
$\mathbf{B}(z)\equiv \mathbf{D}(z)$, so in place of \eqref{eq:AEDdot}
we have instead
\begin{equation}
\mathbf{A}(z)=e^{n\ell\sigma_3/2}\mathbf{E}(x,y)\dot{\mathbf{D}}(z)
e^{ng(z)\sigma_3}e^{-n\ell\sigma_3/2},\quad z=x+iy\in\mathbb{C}\setminus(\Omega_+\cup\Omega_-\cup\Omega_\alpha\cup\Omega_\beta).
\label{eq:AEDdotbeta}
\end{equation}
Therefore, for such $z$:
\begin{align}
A_{11}(z)&=\left[E_{11}(x,y)\dot{D}_{11}(z)+E_{12}(x,y)\dot{D}_{21}(z)\right]e^{ng(z)},\\
A_{21}(z)&=\left[E_{21}(x,y)\dot{D}_{11}(z)+E_{22}(x,y)\dot{D}_{21}(z)\right]
e^{n(g(z)-\ell)}.
\end{align}
Supposing that $z\in S_\beta^\mathrm{I}\cap \mathbb{C}_+$, we may now
proceed by using \eqref{eq:Airybeta} and \eqref{eq:MI} to find
$\hat{\mathbf{D}}(z)$, with the result that
\begin{align}
\label{eq:A11asymp}
A_{11}(z)&=\sqrt{\pi}e^{ng(z)}e^{nu_\beta(z)^{3/2}/2}\left[G_n^1(z)
\left(1+\mathcal{O}\left(\Delta_n\right)\right)+G_n^2(z)
\mathcal{O}\left(\Delta_n\right)\right],\\
 \label{eq:A11asymp2}
A_{21}(z)&=\sqrt{\pi}e^{n(g(z)-\ell)}e^{nu_\beta(z)^{3/2}/2}\left[G_n^2(z)
\left(1+\mathcal{O}\left(\Delta_n\right)\right)+G_n^1(z)
\mathcal{O}\left(\Delta_n\right)\right],
\end{align}
where
\begin{align}
G_n^1(z)&:=n^{1/6}w(z)\mathrm{Ai}
\left(\left(\tfrac{3n}{4}\right)^{2/3}u_\beta(z)\right)-n^{-1/6}w(z)^{-1}
\mathrm{Ai}'\left(\left(\tfrac{3n}{4}\right)^{2/3}u_\beta(z)\right),\\
G_n^2(z)&:=-in^{1/6}w(z)\mathrm{Ai}
\left(\left(\tfrac{3n}{4}\right)^{2/3}u_\beta(z)\right)-in^{-1/6}w(z)^{-1}
\mathrm{Ai}'\left(\left(\tfrac{3n}{4}\right)^{2/3}u_\beta(z)\right).
\end{align}

If we assume that $\zeta:=(3n/4)^{2/3}u_\beta(z)$ is bounded, then Property
3 of Lemma~\ref{lem:thetaextendgeneral} guarantees that
$e^{-in\Theta(x,y)}= e^{2\zeta^{3/2}/3}(1+\mathcal{O}(n^{-2/3}))$, and
so with the use of \eqref{eq:AiryIdentity} we see that
\eqref{eq:AiryAsTri1} and \eqref{eq:AiryAsTri2} agree, respectively,
with \eqref{eq:A11asymp} and \eqref{eq:A11asymp2} up to error terms;
we therefore have
\begin{align}
\label{eq:A11betaclose}
A_{11}(z)&=e^{ng(z)}\left[\sqrt{\pi}w(\beta)n^{1/6}e^{2\zeta^{3/2}/3}\mathrm{Ai}
(\zeta) + 
\mathcal{O}\left(n^{1/6}\Delta_n\right)\right]\\
\intertext{and}
\label{eq:A21betaclose}
A_{21}(z)&=e^{n(g(z)-\ell)}\left[-i\sqrt{\pi}w(\beta)n^{1/6}e^{2\zeta^{3/2}/3}
\mathrm{Ai}(\zeta) + 
\mathcal{O}\left(n^{1/6}\Delta_n\right)\right]
\end{align}
for $\zeta$ bounded with $0\le\arg(\zeta)\le\pi$, where
\begin{equation}
z=\beta+(\lambda n)^{-2/3}\zeta,\quad\quad
\lambda:=\frac{3}{4}[-h_\beta'(\beta)]^{-1}.
\end{equation}
Moreover, we may observe that for $z$ near $\beta$ , and $z \in
\mathbb{C}_{+}$, the following local expansion holds true:
\begin{equation}
g(z)+\frac{2}{3n}\zeta^{3/2}=g(z)+\frac{1}{2}u_\beta(z)^{3/2}=
\frac{cV(\beta)+\ell}{2}+
\frac{cV'(\beta)}{2}(z-\beta) +\mathcal{O}((z-\beta)^2).
\label{eq:gthetalocexp}
\end{equation}
(This is proved in the Appendix under the conditions on the external field
$V$ in force in this paper.)  Therefore, \eqref{eq:A11betaclose}
and \eqref{eq:A21betaclose} may also be written as 
\begin{align}
\label{eq:A11betacloseAGAIN}
A_{11}\left(\beta+(\lambda n)^{-2/3}\zeta\right)&=
n^{1/6}e^{n(cV(\beta)+\ell)/2}e^{n^{1/3}cV'(\beta)\lambda^{-2/3}\zeta/2}
\sqrt{\pi}w(\beta)\mathrm{Ai}(\zeta) \\
\nonumber &\quad\quad{}+
\mathcal{O}\left(n^{1/6}\Delta_ne^{n(cV(\beta)+\ell)/2}e^{n^{1/3}cV'(\beta)\lambda^{-2/3}\zeta/2}\right)\\
\intertext{and}
\label{eq:A21betacloseAGAIN}
A_{21}\left(\beta+(\lambda n)^{-2/3}\zeta\right)&=
-in^{1/6}e^{n(cV(\beta)-\ell)/2}e^{n^{1/3}cV'(\beta)\lambda^{-2/3}\zeta/2}
\sqrt{\pi}w(\beta)\mathrm{Ai}(\zeta) \\
\nonumber &\quad\quad{}+ 
\mathcal{O}\left(n^{1/6}\Delta_ne^{n(cV(\beta)-\ell)/2}
e^{n^{1/3}cV'(\beta)\lambda^{-2/3}\zeta/2}\right).
\end{align}
Here we used the fact that the error term in \eqref{eq:gthetalocexp}
is $\mathcal{O}(n^{-4/3})$ for $|\zeta|$ bounded, so the dominant
terms in the errors come from \eqref{eq:A11betaclose} and
\eqref{eq:A21betaclose}.  

\section{Asymptotics for the Derivative and Applications to Random
  Matrix Theory}
\label{sec:derivs}
We have established large-$n$ asymptotics uniform with respect to $z$
for the matrix $\mathbf{A}(z)$, for all $z \in \mathbb{C}$.  In
particular, $A_{11}(z)$ and $A_{21}(z)$ possess asymptotic
descriptions in a neighborhood of the interval $[\alpha, \beta]$.  In
this section we derive asymptotic descriptions for the derivatives
$A_{11}'(z)$ and $A_{21}'(z)$, both in ``the bulk'', i.e.\@ for
$x \in (\alpha, \beta)$ as well as near the endpoints $\alpha$ and
$\beta$.  (In fact we will only consider the endpoint $\beta$.)

Our aim is to obtain derivative asymptotics in order to establish bulk
and edge universality for unitarily invariant matrix models with
external fields that possess only two Lipschitz continuous
derivatives, as described in the Introduction.  It is by now
well-known (see, for example, \cite{Mehta}) that if one obtains an
asymptotic description of the orthogonal polynomials and their
derivatives of the form which we obtain here, then the corresponding
asymptotic formulae for the reproducing kernels follows and exhibits
universality (independence of details of the external field $V$), and
so we will omit these details.

\subsection{Analysis of derivatives in the bulk.}
Let $x\in(\alpha,\beta)$ be bounded away from the endpoints as $n\to\infty$.
Since $A_{11}(z)$ and $A_{21}(z)$ are polynomials and hence entire functions,
we may express their derivatives at $x$ by Cauchy's integral formula:
\begin{equation}
A_{j1}'(x)=\frac{1}{2\pi i}\oint\frac{A_{j1}(s)\,ds}{(s-x)^2},\quad
j=1,2.
\end{equation}
Moreover, since (see \eqref{eq:Afirstcolumn}) $A_{11}(z)$ and
$iA_{21}(z)$ have real coefficients, we may use complex-conjugation symmetry
and the reality of $x$ to write
\begin{equation}
A_{11}'(x)=\frac{1}{\pi}\mathrm{Im}\left(\int_\Gamma
\frac{A_{11}(s)\,ds}{(s-x)^2}\right)\quad\text{and}\quad
A_{21}'(x)=\frac{1}{\pi i}\mathrm{Re}\left(\int_\Gamma
\frac{A_{21}(s)\,ds}{(s-x)^2}\right),
\end{equation}
where $\Gamma$ is any path of integration that begins on the real axis
to the right of $x$ and terminates on the real axis to the left of
$x$, and that avoids the singularity at $x$ by passing through the
upper half-plane.  For our calculations, we will take the path
$\Gamma$ to be a semicircle of radius $n^{-1}$ centered at $s=x$:
$s=x+n^{-1}e^{i\omega}$ for $0<\omega<\pi$.  Thus we have
\begin{align}
\label{eq:A11primeraw}
A_{11}'(x)&=\frac{n}{\pi}\mathrm{Im}\left(\int_0^\pi A_{11}(x+n^{-1}e^{i\omega})
ie^{-i\omega}\,d\omega\right),\\
\label{eq:A21primeraw}
A_{21}'(x)&=\frac{n}{\pi i}\mathrm{Re}\left(\int_0^\pi A_{21}(x+n^{-1}e^{i\omega})
ie^{-i\omega}\,d\omega\right).
\end{align}

We will now substitute from \eqref{eq:polyas1} and \eqref{eq:polyas2},
but first we write them in a more suitable form.  Since $a(z)$ is bounded
for $z\in\Gamma$, and since $g(x+iy)$,
$\Theta(x,y)$, and $\varphi(x+iy)$ are all differentiable, Taylor expansion
about $x=0$ shows that for $z\in\Gamma$ we have both
\begin{equation}
A_{11}(z)=M_1(x,y)+\mathcal{O}\left(\Delta_ne^{ng_+(x)}\right)\quad
\text{and}\quad
e^{n\ell}A_{21}(z)=M_2(x,y)+\mathcal{O}\left(\Delta_ne^{ng_+(x)}\right),
\label{eq:A11A21onGamma}
\end{equation}
where $g_+(x)$ is the boundary value taken
by $g(z)$ as $z\to x$ from $\mathbb{C}_+$, and where ($z=x+iy$)
\begin{align}
M_1(x,y)&:=e^{n(g(z)-i\Theta(x,y)/2)}a(z)\cos\left(\frac{1}{2}(n\Theta(x,y)-\varphi(z))\right)\\
\intertext{and}
M_2(x,y)&:=-ie^{n(g(z)-i\Theta(x,y)/2)}a(z)\sin\left(\frac{1}{2}(n\Theta(x,y)+
\varphi(z))\right).
\end{align}
One important observation is that $M_1(x,y)$ and $iM_2(x,y)$ have real
boundary values taken on the real axis from the upper half-plane.
Indeed, the analytic functions $a(z)$ and $\varphi(z)$ are real for
real $z$, Property 1 of Lemma~\ref{lem:thetaextendgeneral} implies
$\Theta(x,0)=\theta(x)\in\mathbb{R}$, and furthermore by
\eqref{eq:thetadefine},
\begin{equation}
2g_+(x)-i\Theta(x,0)  =2g_+(x)-i\theta(x)=g_+(x)+g_-(x)
\in\mathbb{R},\quad x\in (\alpha,\beta).
\end{equation}
Using \eqref{eq:A11A21onGamma} in \eqref{eq:A11primeraw} and
\eqref{eq:A21primeraw} gives
\begin{align}
\label{eq:A11primenext}
A_{11}'(x)&=\frac{n}{\pi}\mathrm{Im}\left(\int_0^\pi
M_1(x+n^{-1}\cos(\omega),n^{-1}\sin(\omega))ie^{-i\omega}\,d\omega
\right)+\mathcal{O}\left(n\Delta_ne^{ng_+(x)}\right),\\
\label{eq:A21primenext}
e^{n\ell}A_{21}'(x)&=\frac{n}{\pi i}\mathrm{Re}\left(\int_0^\pi
M_2(x+n^{-1}\cos(\omega),n^{-1}\sin(\omega))ie^{-i\omega}\,d\omega
\right)+\mathcal{O}\left(n\Delta_ne^{ng_+(x)}\right).
\end{align}

We begin our analysis by integrating by parts:  since for $M=M_1$ or $M=M_2$,
\begin{equation}
\begin{split}
\int_0^\pi M(x+n^{-1}\cos(\omega),n^{-1}\sin(\omega))ie^{-i\omega}\,d\omega
&=
-\int_0^\pi M(x+n^{-1}\cos(\omega),n^{-1}\sin(\omega))\frac{d}{d\omega}(e^{-i\omega})\,d\omega\\
&= M(x-n^{-1},0)+M(x+n^{-1},0)\\
&\quad{}+\frac{1}{n}\int_0^\pi
M_y(x+n^{-1}\cos(\omega),n^{-1}\sin(\omega))\cos(\omega)e^{-i\omega}\,d\omega\\
&\quad{}-
\frac{1}{n}\int_0^\pi
M_x(x+n^{-1}\cos(\omega),n^{-1}\sin(\omega))\sin(\omega)e^{-i\omega}\,d\omega,
\end{split}
\end{equation}
we see that upon taking the imaginary part (for $A_{11}'(x)$) or real
part (for $A_{21}'(x)$), the boundary terms vanish:
\begin{multline}
n\mathrm{Im}\left(\int_0^\pi M_1(x+n^{-1}\cos(\omega),n^{-1}\sin(\omega))
ie^{-i\omega}\,d\omega\right)\\
\begin{aligned}
&=
\mathrm{Im}\left(\int_0^\pi M_{1y}(x+n^{-1}\cos(\omega),n^{-1}\sin(\omega))
\cos(\omega)e^{-i\omega}\,d\omega\right)\\
&\quad\quad{}-  
\mathrm{Im}\left(\int_0^\pi M_{1x}(x+n^{-1}\cos(\omega),n^{-1}\sin(\omega))
\sin(\omega)e^{-i\omega}\,d\omega\right),
\end{aligned}
\end{multline}
and
\begin{multline}
n\mathrm{Re}\left(\int_0^\pi M_2(x+n^{-1}\cos(\omega),n^{-1}\sin(\omega))
ie^{-i\omega}\,d\omega\right)\\
\begin{aligned}
&=
\mathrm{Re}\left(\int_0^\pi M_{2y}(x+n^{-1}\cos(\omega),n^{-1}\sin(\omega))
\cos(\omega)e^{-i\omega}\,d\omega\right)\\
&\quad\quad{}-  
\mathrm{Re}\left(\int_0^\pi M_{2x}(x+n^{-1}\cos(\omega),n^{-1}\sin(\omega))
\sin(\omega)e^{-i\omega}\,d\omega\right).
\end{aligned}
\end{multline}
Moreover, using \eqref{eq:CR} to eliminate $M_{jy}$ in favor of
$M_{jx}$ and $\dbar M_j$ for $j=1,2$, these become
\begin{multline}
n\mathrm{Im}\left(\int_0^\pi M_1(x+n^{-1}\cos(\omega),n^{-1}\sin(\omega))
ie^{-i\omega}\,d\omega\right)\\
\begin{aligned}
&=\mathrm{Re}\left(\int_0^\pi M_{1x}(x+n^{-1}\cos(\omega),n^{-1}\sin(\omega))
\,d\omega\right)\\
&\quad\quad{}-  
2\mathrm{Re}\left(\int_0^\pi \dbar M_1(x+n^{-1}\cos(\omega),n^{-1}\sin(\omega))
\cos(\omega)e^{-i\omega}\,d\omega\right),
\end{aligned}
\end{multline}
and
\begin{multline}
n\mathrm{Re}\left(\int_0^\pi M_2(x+n^{-1}\cos(\omega),n^{-1}\sin(\omega))
ie^{-i\omega}\,d\omega\right)\\
\begin{aligned}
&=
-\mathrm{Im}\left(\int_0^\pi M_{2x}(x+n^{-1}\cos(\omega),n^{-1}\sin(\omega))
\,d\omega\right)\\
&\quad\quad{}+  
2\mathrm{Im}\left(\int_0^\pi \dbar M_2(x+n^{-1}\cos(\omega),n^{-1}\sin(\omega))
\cos(\omega)e^{-i\omega}\,d\omega\right).
\end{aligned}
\end{multline}

By splitting cosines and sines into exponentials, we may
write $M_1(x,y)$ and $M_2(x,y)$ in the form
\begin{align}
M_1(x,y)&:=u^-(z)e^{ng(z)}+u^+(z)e^{n(g(z)-i\Theta(x,y))}\\
\intertext{and}
M_2(x,y)&:=u^-(z)e^{n(g(z)-i\Theta(x,y))} -u^+(z)e^{ng(z)},
\end{align}
where
\begin{equation}
u^\pm(z):=\frac{1}{2}a(z)e^{\pm i\varphi(z)/2}.
\end{equation}
One reason for writing $M_1(x,y)$ and $M_2(x,y)$ in this way is to
explicitly display their dependence on $n$; indeed $u^\pm(z)$, $g(z)$,
and $\Theta(x,y)$ are independent of $n$.  Using the fact that
$u^\pm(z)$ and $g(z)$ are all analytic functions of $z=x+iy$ on the
contour $\Gamma$, we have
\begin{equation}
\dbar M_1(x,y)=-inu^+(z)e^{n(g(z)-i\Theta(x,y))}\dbar\Theta(x,y)\quad\text{and}
\quad
\dbar M_2(x,y)=-inu^-(z)e^{n(g(z)-i\Theta(x,y))}\dbar\Theta(x,y).
\end{equation}
Now, $u^\pm(z)$ are bounded functions, and according to Property 2 of
Lemma~\ref{lem:thetaextendgeneral}, $\dbar\Theta(x,y)=O(n^{-1})$ for
$x+iy\in\Gamma$.  Since also $g(z)$ and $\Theta(x,y)$ are
differentiable and (by Property 1 of
Lemma~\ref{lem:thetaextendgeneral})
$\Theta(x,0)=\theta(x)\in\mathbb{R}$, we finally learn that
\begin{equation}
\dbar M_1(x,y)=\mathcal{O}\left(e^{ng_+(x)}\right)\quad\text{and}\quad
\dbar M_2(x,y)=\mathcal{O}\left(e^{ng_+(x)}\right),\quad x+iy\in\Gamma.
\end{equation}
Therefore, \eqref{eq:A11primenext} and \eqref{eq:A21primenext} may 
be written in the form
\begin{align}
\label{eq:A11primefurther}
A_{11}'(x)&=\frac{1}{\pi}\mathrm{Re}\left(\int_0^\pi M_{1x}(x+n^{-1}\cos(\omega),
n^{-1}\sin(\omega))\,d\omega\right)+\mathcal{O}\left(e^{ng_+(x)}\right)\\
\intertext{and}
\label{eq:A21primefurther}
e^{n\ell}A_{21}'(x)&=\frac{i}{\pi}\mathrm{Im}\left(\int_0^\pi
M_{2x}(x+n^{-1}\cos(\omega),n^{-1}\sin(\omega))\,d\omega\right)
+\mathcal{O}\left(e^{ng_+(x)}\right).
\end{align}
Now, each term in $M_1(x,y)$ and $M_2(x,y)$ is of the form
$u^\pm(z)e^{n(g(z)-i\sigma\Theta(x,y))}$ where either $\sigma=0$ or
$\sigma=1$, and
\begin{equation}
\frac{\partial}{\partial x}\left[u^\pm(z)e^{n(g(z)-i\sigma\Theta(x,y))}\right]=
\left[u^{\pm\prime}(z)+nu^\pm(z)\left(g'(z)-i\sigma\Theta_x(x,y)
\right)\right]e^{n(g(z)-i\sigma\Theta(x,y))},
\end{equation}
so evaluating for $z=x+n^{-1}e^{i\omega}$ we may expand the result for 
large $n$.  Using Property 1 of Lemma~\ref{lem:thetaextendgeneral}
to assert second-order differentiability of $\Theta(x,y)$ in the bulk
and Property 2 of the same Lemma to guarantee that $\dbar\Theta(x,0)=0$,
we finally arrive at
\begin{equation}
\left.
\frac{\partial}{\partial x}\left[u^\pm(z)e^{n(g(z)-i\sigma\Theta(x,y))}\right]\right|_{z=x+n^{-1}e^{i\omega}}=nu^\pm(x)\left(g_+'(x)-i\sigma\theta'(x)\right)
e^{n(g_+(x)-i\sigma\theta(x))}e^{qe^{i\omega}} + 
\mathcal{O}\left(e^{ng_+(x)}\right),
\end{equation}
where $q:=g_+'(x)-i\sigma\theta'(x)$ is independent of $\omega$.
Since for any complex number $q$
\begin{equation}
\int_0^\pi e^{qe^{i\omega}}\,d\omega = \pi,
\end{equation}
we obtain from \eqref{eq:A11primefurther} and \eqref{eq:A21primefurther}
that 
\begin{align}
\nonumber
A_{11}'(x)&=\mathrm{Re}\left(nu^-(x)g_+'(x)e^{ng_+(x)}+nu^+(x)\left(g_+'(x)-i\theta'(x)\right)e^{n(g_+(x)-i\theta(x))}\right)+\mathcal{O}\left(e^{ng_+(x)}\right)\\
&=\mathrm{Re}\left(\frac{d}{dx}M_1(x,0+)\right) +
\mathcal{O}\left(e^{ng_+(x)}\right)\\
\nonumber
&=\frac{d}{dx}\left[e^{n(cV(x)+\ell-\phi(x))/2}a(x)\cos\left(\frac{1}{2}(n\theta(x)-\varphi(x))\right)\right]+\mathcal{O}\left(e^{n(cV(x)+\ell-\phi(x))/2}\right)\\
\intertext{and}
\nonumber
A_{21}'(x)&=i\mathrm{Im}\left(nu^-(x)\left(g_+'(x)-i\theta'(x)\right)
e^{n(g_+(x)-i\theta(x)-\ell)}-nu^+(x)g_+'(x)e^{n(g_+(x)-\ell)}\right) + 
\mathcal{O}\left(e^{n(g_+(x)-\ell)}\right)\\
&=i\mathrm{Im}\left(e^{-n\ell}\frac{d}{dx}M_2(x,0+)\right) +
\mathcal{O}\left(e^{n(g_+(x)-\ell)}\right)\\
\nonumber 
&= \frac{d}{dx}\left[-ie^{n(cV(x)-\ell-\phi(x))/2}a(x)\sin\left(\frac{1}{2}(n\theta(x)+\varphi(x))\right)\right] + \mathcal{O}\left(e^{n(cV(x)-\ell-\phi(x))/2}\right).
\end{align}
Comparing these results with \eqref{eq:A11axis} and \eqref{eq:A21axis}
shows that the asymptotic formulae for the derivatives of the orthogonal 
polynomials on the real axis in the bulk may be obtained from the corresponding
asymptotic formulae for the polynomials themselves by differentiating the
leading terms.

\subsection{Analysis of derivatives at the edge}
We will now apply similar considerations to the asymptotic formulae
\eqref{eq:A11betaclose}--\eqref{eq:A21betaclose}, to obtain
asymptotics for derivatives of the orthogonal polynomials that are
valid for $z$ in a vicinity of the endpoints $\beta$ and $\alpha$.  We
will present the details for the endpoint $z = \beta$, as the argument
for the behavior near $z = \alpha$ is entirely similar.  More
precisely, our aim is to establish asymptotic formulae for the quantities
\begin{equation}
  \frac{d}{d \zeta} A_{11}\left(\beta + (\lambda n)^{-2/3}\zeta 
  \right) \quad\text{and}\quad
  \frac{d}{d \zeta} A_{21}\left(\beta + (\lambda n)^{-2/3}\zeta
 \right),\quad\quad\lambda:=\frac{3}{4}[-h_\beta'(\beta)]^{-1},
\end{equation}
which are what one needs to establish universality of the distribution of
the largest eigenvalue in Hermitian random matrix theory.  

Let
\begin{align}
\tau_1(\zeta)&:=e^{-n(cV(\beta)+\ell)/2}e^{-n^{1/3}cV'(\beta)\lambda^{-2/3}\zeta/2}
A_{11}\left(\beta+(\lambda n)^{-2/3}\zeta\right)\\
\tau_2(\zeta)&:=e^{-n(cV(\beta)-\ell)/2}e^{-n^{1/3}cV'(\beta)\lambda^{-2/3}
\zeta/2}
A_{21}\left(\beta+(\lambda n)^{-2/3}\zeta\right).
\end{align}
According to \eqref{eq:A11betacloseAGAIN} and
\eqref{eq:A21betacloseAGAIN}, these may be expressed as
\begin{align}
\label{eq:f1expand}
\tau_1(\zeta)&=
n^{1/6}\sqrt{\pi}w(\beta)\mathrm{Ai}(\zeta) 
+
\mathcal{O}\left(n^{1/6}\Delta_n\right)\\
\intertext{and}
\label{eq:f2expand}
\tau_2(\zeta)&=
-in^{1/6}\sqrt{\pi}w(\beta)
\mathrm{Ai}(\zeta) 
+\mathcal{O}\left(n^{1/6}\Delta_n\right)
\end{align}
for $|\zeta|$ bounded.  

Since $\tau_1(\zeta)$ and $\tau_2(\zeta)$ are entire functions that are real
for real $\zeta$, just as in the analysis in the bulk we may express
the derivatives $\tau_1'(\zeta)$ and $\tau_2'(\zeta)$ in terms of Cauchy's
formula as
\begin{align}
\tau_1'(\zeta)&=\frac{1}{\pi}\mathrm{Im}\left(\int_\Gamma
\frac{\tau_1(\xi)\,d\xi}{(\xi-\zeta)^2}\right)
\\
\intertext{and}
\tau_2'(\zeta)&=\frac{1}{\pi i}\mathrm{Re}\left(\int_\Gamma
\frac{\tau_2(\xi)\,d\xi}{(\xi-\zeta)^2}\right)
\end{align}
as long as $\zeta\in\mathbb{R}$, where $\Gamma$ is any path in the
upper half-plane from the real axis to the right of $\zeta$ to another
point on the real axis to the left of $\zeta$.
But, since the dominant terms in \eqref{eq:f1expand} and \eqref{eq:f2expand}
are entire functions of $\zeta$, and since $|\xi-\zeta|$ is bounded away
from zero on the contour $\Gamma$ of finite length, it follows from
a residue calculation that
\begin{align}
\tau_1'(\zeta)&=\frac{d}{d\zeta}\left(n^{1/6}\sqrt{\pi}w(\beta)\mathrm{Ai}
(\zeta) 
 \right)+\mathcal{O}\left(n^{1/6}\Delta_n\right)\\
\intertext{and}
\tau_2'(\zeta)&=\frac{d}{d\zeta}\left(-in^{1/6}\sqrt{\pi}w(\beta)
\mathrm{Ai}(\zeta)\right)
+\mathcal{O}\left(n^{1/6}\Delta_n\right).
\end{align}
Therefore,
\begin{align}
\frac{d}{d\zeta}A_{11}\left(\beta+(\lambda n)^{-2/3}\zeta\right)&=
\frac{1}{2}n^{1/2}cV'(\beta)\lambda^{-2/3}
e^{n(cV(\beta)+\ell)/2}e^{n^{1/3}cV'(\beta)\lambda^{-2/3}\zeta/2}
\sqrt{\pi}w(\beta)\mathrm{Ai}(\zeta)\\
\nonumber&\quad\quad{}
+\mathcal{O}\left(n^{1/2}\Delta_ne^{n(cV(\beta)+\ell)/2}e^{n^{1/3}cV'(\beta)\lambda^{-2/3}\zeta/2}\right)\\
\frac{d}{d\zeta}A_{21}\left(\beta+(\lambda n)^{-2/3}\zeta\right)&=
-i\frac{1}{2}n^{1/2}cV'(\beta)\lambda^{-2/3}
e^{n(cV(\beta)-\ell)/2}e^{n^{1/3}cV'(\beta)\lambda^{-2/3}\zeta/2}
\sqrt{\pi}w(\beta)\mathrm{Ai}(\zeta)\\
\nonumber&\quad\quad{}
+\mathcal{O}\left(n^{1/2}\Delta_ne^{n(cV(\beta)-\ell)/2}
e^{n^{1/3}cV'(\beta)\lambda^{-2/3}\zeta/2}\right).
\end{align}
These may also be written as
\begin{align}
\frac{d}{d\zeta}A_{11}\left(\beta+(\lambda n)^{-2/3}\zeta\right)&=
\frac{d}{d\zeta}\left[n^{1/6}e^{n(cV(\beta)+\ell)/2}e^{n^{1/3}cV'(\beta)\lambda^{-2/3}\zeta/2}
\sqrt{\pi}w(\beta)\mathrm{Ai}(\zeta)\right]\\
\nonumber&\quad\quad{}
+\mathcal{O}\left(n^{1/2}\Delta_ne^{n(cV(\beta)+\ell)/2}e^{n^{1/3}cV'(\beta)\lambda^{-2/3}\zeta/2}\right)\\
\frac{d}{d\zeta}A_{21}\left(\beta+(\lambda n)^{-2/3}\zeta\right)&=
\frac{d}{d\zeta}\left[-in^{1/6}e^{n(cV(\beta)-\ell)/2}
e^{n^{1/3}cV'(\beta)\lambda^{-2/3}\zeta/2}
\sqrt{\pi}w(\beta)\mathrm{Ai}(\zeta)\right]\\
\nonumber&\quad\quad{}
+\mathcal{O}\left(n^{1/2}\Delta_ne^{n(cV(\beta)-\ell)/2}
e^{n^{1/3}cV'(\beta)\lambda^{-2/3}\zeta/2}\right),
\end{align}
which, upon comparing with \eqref{eq:A11betacloseAGAIN} and
\eqref{eq:A21betacloseAGAIN}, show that asymptotic formulae for
derivatives valid at the edge may be obtained by differentiating the
leading terms of the corresponding formulae for the polynomials
themselves.

\appendix
\section*{Appendix:  Convex External Fields}
\label{Sec:Convex}
Everywhere in this Appendix we shall assume that (i) the external
field grows sufficiently rapidly as $|x| \to \infty$, and (ii) the
external field is strictly convex, and possesses $d$ continuous
derivatives, with $d \ge 2$.  In a separate discussion below, we will
consider the specific situation that $V$ possesses just two Lipschitz
continuous derivatives.  The main results are summarized in
Lemma~\ref{lem:Lem1} at the end of this Appendix.

The assumed growth and strict convexity of $V(x)$ and the positivity
of $c$ implies that $\mu_*$ is compactly supported and absolutely
continuous with respect to Lebesgue measure, with support consisting
of a single interval $[\alpha,\beta]$ for some real $\alpha<\beta$.
To obtain a formula for $\mu_*$ in this case, we consider the
auxiliary function $g(z)$ defined in \eqref{eq:gdef}, analytic for
$z\in\mathbb{C}\setminus (-\infty,\beta]$ (in the present case the
integral is taken over the single interval $[\alpha, \beta]$).  In
terms of $g(z)$ the variational condition \eqref{eq:ELinside} becomes
\eqref{eq:ELinsideSpecial} which we rewrite here:
\begin{equation}
cV(x)-\left(g_+(x)+g_-(x)\right)=- \ell,\quad\quad
\alpha<x<\beta,
\end{equation}
where $g_+(x)$ and $g_-(x)$ denote the boundary values taken by $g(z)$
as $z\rightarrow x$ with $z\in \mathbb{C}_+$ and $z\in\mathbb{C}_-$
respectively.

Assuming that differentiation commutes with taking boundary values
\eqref{eq:ELinsideSpecial} and \eqref{eq:DifferenceOutside} imply that
\begin{equation}
\begin{array}{rcll}
g'_+(x)+g'_-(x)&=&cV'(x),&\quad\quad x\in (\alpha,\beta)\\\\
g'_+(x)-g'_-(x)&=&0, &\quad\quad x \in \mathbb{R} \setminus (\alpha,\beta).
\end{array}
\label{eq:gpeqK}
\end{equation}
In particular, $g'(z)$ is an analytic function for
$z\in\mathbb{C}\setminus[\alpha,\beta]$.  To find $g'(z)$ from these
conditions, we introduce the function $R(z)$ satisfying $R(z)^2 =
(z-\alpha)(z-\beta)$ such that $R(z)$ is analytic for
$z\in\mathbb{C}\setminus[\alpha,\beta]$ and $R(z)=z+\mathcal{O}(1)$ as
$z\to\infty$.  Setting $g'(z)=f(z)R(z)$ for some new unknown
function $f(z)$, we find that like $g'(z)$ and $R(z)$, $f(z)$ is an
analytic function of $z$ for $z\in\mathbb{C}\setminus[\alpha,\beta]$,
and that its boundary values taken on $(\alpha,\beta)$ from the upper
and lower half-planes satisfy the relation
\begin{equation}
f_+(x)-f_-(x) = \frac{cV'(x)}{R_+(x)},\quad\quad\alpha<x<\beta\,.
\end{equation}
Since $g'(z)=1/z + \mathcal{O}(1/z^2)$ as $z\to\infty$, it follows that
$f(z)=1/z^2 + \mathcal{O}(1/z^3)$ as $z\to\infty$, and hence
\begin{equation}
f(z)=\frac{c}{2\pi i}\int_\alpha^\beta \frac{V'(s)\,ds}{(s-z)R_+(s)}.
\label{eq:fformula}
\end{equation}
Considering \eqref{eq:fformula} for large $z$,
we see that 
\begin{equation}
c\int_\alpha^\beta \frac{V'(s)\,ds}{R_+(s)} =0,\quad\quad
c\int_\alpha^\beta \frac{sV'(s)\,ds}{R_+(s)} = -2\pi i.
\label{eq:endpoints}
\end{equation}
These two equations determine the endpoints $\alpha$ and $\beta$.  
With $\alpha$ and $\beta$ chosen so that the equations
\eqref{eq:endpoints} hold, we may obtain a formula, in terms of a
Cauchy principal value integral, for the density
$\psi(x)$ of the equilibrium measure $\mu_*$ valid in the support interval
$\alpha\le x\le \beta$:
\begin{equation}
\psi(x) =  \frac{cR_+(x)}{2\pi^2}\dashint_\alpha^\beta
\frac{V'(s)\,ds}{(s-x)R_+(s)}.
\label{eq:psiformula}
\end{equation}

Defining a real-valued function $h(x)$ for $x\in\mathbb{R}$ by the
formula
\begin{equation}
h(x):=\frac{i}{\pi}
\int_\alpha^\beta\frac{V'(s)-V'(x)}{s-x}\frac{ds}{R_+(s)},
\label{eq:hdef}
\end{equation}
it is straightforward to verify that
\begin{equation}
\label{eq:psihform}
\psi(x)=\frac{c}{2\pi i}R_+(x)h(x),\quad\quad\alpha<x<\beta,
\end{equation}
and that
\begin{equation}
\label{eq:gpmpmfull}
\phi(x):=cV(x) + \ell-g_+(x)-g_-(x)=
\begin{cases}
\displaystyle -c\int_x^\alpha R(s)h(s)\,ds,&\quad x<\alpha
\\\\
\displaystyle c\int_\beta^xR(s)h(s)\,ds,&\quad x>\beta.
\end{cases}
\end{equation}

From the assumption that $V(x)$ is $d$ times continuously
differentiable, we see that $h(x)$ is $d-2$ times continuously
differentiable.  Also, since $i/R_+(x)$ is positive for
$\alpha<x<\beta$ and
\begin{equation}
\frac{V'(s)-V'(z)}{s-z}=\int_0^1V''(ts+(1-t)z)\,dt,
\end{equation}
the assumption of convexity of $V(x)$ implies that $h(x)$ is strictly
positive for all $x\in\mathbb{R}$.  
Also, since $g_{+}(\beta)
- g_{-}(\beta)=0$ and since for $\alpha<x<\beta$ we have
$g_{+}'(x)-g_{-}'(x)=-2\pi i\psi(x)$, we conclude that
\begin{equation}
\theta(x):=-i (g_{+}(x) - g_{-}(x))=-ic\int_x^\beta R_+(s)h(s)\,ds,
\quad\quad\alpha<x<\beta.
\label{eq:thetaexpression}
\end{equation}

We may now use \eqref{eq:gpmpmfull} to verify (from the positivity of
$h(x)$ and the facts that $R(x)>0$ for $x>\beta$ while $R(x)<0$ for
$x<\alpha$) the strict inequality $\phi(x)>0$ for $x<\alpha$ and
$x>\beta$, as required by Condition~\ref{cond:strict}.  Also, from
\eqref{eq:thetaexpression} we see that $0<\theta(x)<2\pi$ and
$\psi(x)>0$ (i.e.\@ $\theta'(x)<0$) both hold strictly for
$\alpha<x<\beta$ as required by Condition~\ref{cond:strict}.
Moreover, since $h(s)$ is $d-2$ times continuously differentiable,
$\theta(x)$ is $d-1$ times continuously differentiable for
$\alpha<x<\beta$.

One may prove \eqref{eq:gthetalocexp} as follows.  From the identity
$g'(z) = R(z)f(z)$, we have that for $z$ near $\beta$,
\begin{multline}
f(z) = \frac{ c V'(\beta)}{2 R(z)} + 
\frac{c}{2 \pi i} \int_{\alpha}^{\beta} 
\frac{ V'(s) - V'(\beta) }{(s - \beta) R_{+}(s)}\, ds 
+ \frac{c V''(\beta)}{2 R(z)}(z-\beta) \\
+ 
\frac{ c }{2 \pi i}( z - \beta)  \int_{\alpha}^{\beta} \left( 
\frac{V'(s) - V'(\beta)}{s - \beta}  - V''(\beta) \right) \frac{ds }{ ( s - z) R_{+}(s) }.
\end{multline}
Recalling that $2g(\beta) - c V(\beta) - \ell = 0$, we learn that 
\begin{multline}
\label{eq:geqlocbet}
g(z) = \frac{cV(\beta)}{2}  + \frac{\ell}{2}+ 
\frac{cV'(\beta)}{2}(z - \beta) + \left( 
 \frac{c}{2 \pi i} \int_{\alpha}^{\beta} \frac{ V'(s) - V'(\beta) }{(s - \beta) R_{+}(s)}\, ds
\right) \int_{\beta}^{z} R(s)\, ds +  \frac{cV''(\beta) }{4} 
( z - \beta)^{2}  \\
+ 
\int_{\beta}^{z} \left[ \frac{ c ( z' - \beta) R(z') }{2 \pi i} \int_{\alpha}^{\beta} \left( 
\frac{V'(s) - V'(\beta)}{s - \beta}  - V''(\beta) \right) \frac{ds }{ ( s - z') R_{+}(s) } \right]\, dz'.
\end{multline}
Now, recalling that 
\begin{equation}
\frac{1}{2}u_\beta(z)^{3/2} = -\frac{h_\beta'(\beta)}{2}(z-\beta)^{3/2},
\label{eq:otherterm}
\end{equation}
and
\begin{equation}
h_\beta(x)=\frac{\theta(x)}{\sqrt{\beta-x}}=-i\frac{g_+(x)-g_-(x)}{\sqrt{\beta-x}},\quad \alpha<x<\beta,
\end{equation}
we may use \eqref{eq:geqlocbet} to obtain the value of $h_\beta'(\beta)$:
\begin{equation}
h_\beta'(\beta)=\frac{2c \sqrt{ \beta - \alpha} }{3 \pi i} \int_{\alpha}^{\beta} \frac{ V'(s) - V'(\beta) }{(s - \beta) R_{+}(s)}\, ds<0,
\end{equation}
and then the expansion \eqref{eq:gthetalocexp} follows by adding
\eqref{eq:geqlocbet} and \eqref{eq:otherterm}.

\subsection*{Regularity for $V''$ Lipschitz} \label{sec:app:reg} It is
also straightforward to derive a formula for $g''(z)$.  This is a
useful exercise if one assumes only that $V''$ is Lipschitz
continuous, which we do throughout this subsection.

By further differentiation, \eqref{eq:ELinsideSpecial} and
\eqref{eq:DifferenceOutside} imply that
\begin{equation}
\begin{array}{rcll}
g''_{+}(x)+g''_{-}(x)&=&cV''(x),&\quad\quad x\in (\alpha,\beta)\\\\
g''_{+}(x)-g''_{-}(x)&=&0, &\quad\quad x \in \mathbb{R} \setminus 
(\alpha,\beta).
\end{array}
\end{equation}
In particular, $g''(z)$ is an analytic function for
$z\in\mathbb{C}\setminus[\alpha,\beta]$.  To find $g''(z)$ from these
conditions, we set $g''(z)=F(z)/R(z)$ for some new unknown function
$F(z)$, we find that like $g''(z)$ and $R(z)$, $F(z)$ is an analytic
function of $z$ for $z\in\mathbb{C}\setminus[\alpha,\beta]$, and that
its boundary values taken on $(\alpha,\beta)$ from the upper and lower
half-planes satisfy the relation
\begin{equation}
F_+(x)-F_-(x) = cV''(x)R_+(x),\quad\alpha<x<\beta\,.
\end{equation}
It follows that $F(z)$ is a function of the form
\begin{equation}
F(z)=\frac{c}{2\pi i}\int_\alpha^\beta \frac{V''(s)R_{+}(s) \,ds}{(s-z)}.
\label{eq:FFormula}
\end{equation}

Since $\psi'(x) = -2 \pi i (g''_{+}(x) - g''_{-}(x))$ for
$\alpha<x<\beta$, we obtain
\begin{equation}
\begin{split}
\psi'(x)&= -\frac{ 2 \pi i}{R_{+}(x)} \left(  F_{+}(x) + F_{-}(x) \right) \\
&=-\frac{2c}{R_{+}(x)} 
 \int_\alpha^\beta \frac{V''(s) - V''(x)}{s-x}R_{+}(s) \,ds +
 \frac{ 2 \pi i c V''(x)}{R_{+}(x)} \left(  x - \frac{ \alpha + \beta}{2} \right),
\end{split}
\end{equation}
from which it follows that $\sqrt{ (x - \alpha) ( \beta - x)} \psi'(x)$ 
is bounded \emph{uniformly} for $\alpha<x<\beta$.  In other words,
\begin{equation}
\label{eq:psiprimebd}
\left| \psi'(x) \right| \le \frac{ C}{ \sqrt{ ( x - \alpha)(\beta-x)}},
\quad\quad \alpha<x<\beta.
\end{equation}

From the above considerations we have the following formula, which is
valid for $x< \alpha$ and also for $x > \beta$:
\begin{equation}
\phi''(x) = \frac{ic}{\pi R(x)} 
\int_{\alpha}^{\beta} \frac{ V''(s) - V''(x)}{s - x} R_{+}(s) ds + 
\frac{c V''(x)}{R(x)} \left( x - \frac{ \alpha + \beta}{2} \right),
\end{equation}
which in turn implies that $|x - \alpha|^{1/2}|x - \beta|^{1/2}
\phi''(x)$ is bounded on any compact subset of $\mathbb{R}$.  (Of
course, this quantity diverges as $|x| \to \infty$, with $x \in
\mathbb{R}$.)

On the other hand, one may also consider the quantity $\hat{F}(z) :=
g(z) / R(z)$, for which the following identity can be shown to hold
true:
\begin{equation}
\hat{F}_{\pm}(z) =\left(  \int_{-\infty}^{\alpha} \frac{ds}{R(s) ( s -z )} \right)_{\pm}+ \frac{c}{2 \pi i} \int_{\alpha}^{\beta} \frac{V(s) - V(z)}{s-z} \frac{ds}{R_{+}(s)} + \frac{c V(z) + \ell}{2R_{\pm}(z)}, \quad z \in \mathbb{R}.
\end{equation}
One direct consequence of this last identity is that
\begin{equation}
\label{eq:gpmgm}
g_{+}(z) -g_{-}(z) = R_{+}(z) \left[  2
\int_{-\infty}^{\alpha} \frac{ds}{R(s) ( s -z )} + \frac{c}{ \pi i} \int_{\alpha}^{\beta} \frac{V(s) - V(z)}{s-z} \frac{ds}{R_{+}(s)}
\right],\quad z \in (\alpha, \beta],
\end{equation}
with both quantities appearing within the square brackets on the right
hand side of \eqref{eq:gpmgm} possessing at least one Lipschitz
continuous derivative for all $z \in [\alpha+\epsilon,\beta]$ for any
$\epsilon > 0$.  Similarly, we have
\begin{equation}
\label{eq:gpPgm}
g_{+}(z) +g_{-}(z) -c V(z) - \ell= R(z) \left[  2
\int_{-\infty}^{\alpha} \frac{ds}{R(s) ( s -z )} + \frac{c }{ \pi i} \int_{\alpha}^{\beta} \frac{V(s) - V(z)}{s-z} \frac{ds}{R_{+}(s)}
\right], \quad z \in [\beta, \infty),
\end{equation}
and again the quantities appearing within the square brackets on the
right hand side of \eqref{eq:gpPgm} possess at least one Lipschitz
continuous derivative for all $z \in [\beta, \infty)$.

The behavior near $z = \alpha$ is slightly more subtle, but using the identity
\begin{equation}
\label{eq:usefulintegral}
\int_{-\infty}^{\alpha} \frac{ds}{R(s)(s-z)} = \frac{\pi i}{\hat{R}(z)} + \int_{\infty}^{\beta} \frac{ds}{R(s)(s-z)},
\end{equation}
where $\hat{R}(z) = \mathrm{sgn}\left(\mathrm{Im}(z)\right) R(z)$ is the
function which coincides with $R(z)$ in $\mathbb{C}_{+}$, and is
analytic in $\mathbb{C} \setminus ( (-\infty, \alpha] \cup
  [\beta,\infty) )$.  Indeed, using \eqref{eq:usefulintegral},
the identity \eqref{eq:gpmgm} becomes
\begin{equation}
\label{eq:gpmgmalp}
g_{+}(z) -g_{-}(z) = 2 \pi i + R_{+}(z) \left[  2
\int_{\infty}^{\beta} \frac{ds}{R(s) ( s -z )} + \frac{c}{ \pi i} \int_{\alpha}^{\beta} \frac{V(s) - V(z)}{s-z} \frac{ds}{R_{+}(s)}
\right], \quad z \in [\alpha, \beta),
\end{equation}
and once again the quantity within the square brackets on the
right-hand side of \ref{eq:gpmgmalp} possesses at least one Lipschitz
continuous derivative.  Similarly, the identity \eqref{eq:gpPgm} can be
rewritten, in light of \ref{eq:usefulintegral}, as follows:
\begin{equation}
\label{eq:gpPgmalp}
g_{+}(z) +g_{-}(z) -c V(z) - \ell= R(z) \left[  2
\int_{\infty}^{\beta} \frac{ds}{R(s) ( s -z )} + \frac{c }{ \pi i} \int_{\alpha}^{\beta} \frac{V(s) - V(z)}{s-z} \frac{ds}{R_{+}(s)}
\right], \quad z \in (-\infty,\alpha],
\end{equation}
the quantity within the square brackets again possessing one Lipschitz
continuous derivative.

We summarize the results of this Appendix with the following Lemma.
\begin{lemma}
\label{lem:Lem1}
Suppose that the external field $V$ possesses two Lipschitz continuous
derivatives, is strictly convex, and grows faster than $[\log(1
    + x^{2})]^{1 + \epsilon}$ for some $\epsilon>0$.  Then the
density $\psi(x)$ of the equilibrium measure $\mu_*$ is supported on a
single interval, $[\alpha, \beta]$.  On this interval, the function
$\psi$ has the following properties.
\begin{itemize}
\item The function $\psi$ may be expressed via \eqref{eq:psihform}
  with $h$, defined in \eqref{eq:hdef}, being Lipschitz continuous on
  $[\alpha, \beta]$.
\item The function $\psi$ has one derivative, which satisfies the bound 
\eqref{eq:psiprimebd}.
\item In vicinities of the endpoints $\beta$ and $\alpha$, the related
  function $\theta$ (recall $\theta'(x)=-2\pi\psi(x)$) satisfies
  (cf.\@ \eqref{eq:gpmgm} and \eqref{eq:gpmgmalp})
  \begin{equation}
\theta(x)= 2 \pi \int_{x}^{\beta} \psi(s)\, ds = 
\begin{cases}
  -iR_{+}(x) \hat{h}_{\beta}(x), & \quad x \in (\alpha+\epsilon, \beta)\\
  2 \pi  +i R_{+}(x) \hat{h}_{\alpha}(x), &\quad x \in (\alpha, \beta-\epsilon),
\end{cases} 
\end{equation}
for some small $\epsilon>0$, with $\hat{h}_{\beta}$ and
$\hat{h}_{\alpha}$ being positive functions on $(\alpha, \beta)$.  In
addition, $\hat{h}_{\beta}$ possesses one Lipschitz continuous
derivative on $(\alpha, \beta]$, and $\hat{h}_{\alpha}$ possesses one
Lipschitz continuous derivative on $[\alpha, \beta)$.
\end{itemize}
On the complementary set $(-\infty, \alpha) \cup (\beta, \infty)$, the
following properties hold true:
\begin{itemize}
\item The quantity $\phi(x):=cV(x)+\ell-g_+(x)-g_-(x)$ possess two
  derivatives.  The first derivative $\phi'(x)$ may be obtained from
  \eqref{eq:gpmpmfull}, with $h(x)$ being Lipschitz continuous on $(-\infty,
  \alpha] \cup [\beta, \infty)$.
\item The second derivative $\phi''(x)$ satisfies the inequality
\begin{equation}
|x-\alpha|^{1/2}|x-\beta|^{1/2}|\phi''(x)|\le C
\end{equation}
on any compact subset of $\mathbb{R}$.  
\item In vicinities of the endpoints $\alpha$ and $\beta$, the
  function $\phi(x)$ satisfies
  (cf.\@ \eqref{eq:gpPgm} and \eqref{eq:gpPgmalp})
\begin{equation}
\phi(x) = \begin{cases}
-R(x) \hat{h}_{\beta}(x), &\quad x \in [\beta, \infty) \\
R(x) \hat{h}_{\alpha}(x), & \quad x \in (- \infty,\alpha],
\end{cases}
\end{equation}
with $\hat{h}_{\beta}$ and $\hat{h}_{\alpha}$ being extensions, to
$[\beta,\infty)$ and $(-\infty,\alpha]$ respectively, of the
functions of the same name, formerly defined on $(\alpha, \beta]$ and
$[\alpha, \beta)$, respectively.  These extensions possess one
Lipschitz continuous derivative as well.  The function
$\hat{h}_{\beta}$ is strictly negative on all of $(\beta, \infty)$,
and the function $\hat{h}_{\alpha}$ is strictly positive on all of
$(-\infty, \alpha)$.
\end{itemize}
\end{lemma}
Note that the functions $h_\alpha(x)$ and $h_\beta(x)$ used in the main
text are simply related to $\hat{h}_\alpha(x)$ and $\hat{h}_\beta(x)$ as
follows:
\begin{equation}
h_\alpha(x)=\sqrt{\beta-x}\,\hat{h}_\alpha(x)\quad\text{and}\quad
h_\beta(x)=\sqrt{x-\alpha}\,\hat{h}_\beta(x).
\end{equation}


\begin{thebibliography}{99}
\bibitem{DOP} J. Baik, T. Kriecherbauer, K. T.-R. McLaughlin, and
P. D. Miller, \textit{Discrete Orthogonal Polynomials.  Asymptotics
and Applications}, Volume 164, Annals of Math. Studies, Princeton
University Press, Princeton, 2007.
\bibitem{BI}
P. Bleher and A. Its, ``Semiclassical asymptotics of orthogonal
polynomials, Riemann-Hilbert problem, and the universality in the
matrix model,''  \textit{Ann. Math.}, \textbf{50} 185--266, 1999.
\bibitem{Deift}
P. Deift,  ``Integrable operators,'' in \textit{Differential
Operators and Spectral Theory: M. Sh. Birman's 70th Anniversary
Collection} (V. Buslaev, M. Solomyak, D. Yafaev, eds.), \textit{Amer.
Math. Soc. Transl.}, ser. 2, \textbf{159}, American Mathematical Society,
Providence,  1999.
\bibitem{DeiftOpen}
P. Deift, ``Some open problems in random matrix theory and the theory of integrable systems,'' \url{arXiv:0712.0849}, 2007.
\bibitem{DG06}P. Deift and D. Gioev, ``Universality at the edge of the
  spectrum for unitary, orthogonal and symplectic ensembles of random
  matrices,'' \url{math-ph/0507023}, 2006.
\bibitem{op1} P. Deift, T. Kriecherbauer, K. T.-R. McLaughlin,
S. Venakides, and X. Zhou, ``Uniform asymptotics for polynomials
orthogonal with respect to varying exponential weights and
applications to universality questions in random matrix theory,''
\textit{Comm. Pure Appl. Math.}, \textbf{52}, 1335--1425, 1999.
\bibitem{op2} P. Deift, T. Kriecherbauer, K. T.-R. McLaughlin,
S. Venakides, and X. Zhou, ``Strong asymptotics of orthogonal
polynomials with respect to exponential weights,'' \textit{Comm. Pure
Appl. Math.}, \textbf{52}, 1491--1552, 1999.
\bibitem{steepestKdV} P. Deift, S. Venakides, and X. Zhou, ``New
results in small dispersion KdV by an extension of the steepest
descent method for Riemann-Hilbert problems'', \textit{Internat.
Math. Res. Notices}, No. 6, 285--299, 1997.
\bibitem{steepestintro} P. Deift and X. Zhou, ``A steepest descent
method for oscillatory Riemann-Hilbert problems: asymptotics for
the mKdV equation'', \textit{Ann. of Math.}, \textbf{137}, 295--368,
1993.
\bibitem{FIK2}
A. Fokas, A. Its, and A. V. Kitaev, ``Discrete
  Painlev\'e equations and their appearance in quantum gravity,''
\textit{Commun. Math. Phys.}, \textbf{142}, 313--344, 1991.
\bibitem{SSE} S. Kamvissis, K. T.-R. McLaughlin, and
P. D. Miller, \textit{Semiclassical Soliton Ensembles for the
Focusing Nonlinear Schr\"odinger Equation}, Volume 154, Annals of Math.
Studies,  Princeton University Press, Princeton, 2003.
\bibitem{KrMcL} T. Kriecherbauer and K. T.-R. McLaughlin, ``Strong
  asymptotics of polynomials orthogonal with respect to Freud
  weights,'' \textit{Internat. Math. Res. Notices}, \textbf{1999},
  299--324, 1999.
\bibitem{KMVV} A. B. J. Kuijlaars, K. T.-R. McLaughlin, W. Van Assche,
and M. Vanlessen, ``The Riemann-Hilbert approach to strong asymptotics
for orthogonal polynomials on $[-1,1]$,'' \textit{Adv. Math.},
\textbf{188}, 337--398, 2004.
\bibitem{LLBulk} E. Levin and D. Lubinsky, ``Universality limits in
the bulk for varying measures,'' to appear in \textit{Adv. Math.}
\bibitem{DbarOP1} K. T.-R. McLaughlin and P. D. Miller, ``The
  $\overline{\partial}$ steepest descent method and the asymptotic
  behavior of polynomials orthogonal on the unit circle with fixed and
  exponentially varying nonanalytic weights.'',
  \textit{Internat. Math. Res. Papers}, \textbf{2006}, Art. ID 48673,
  1--77, 2006.
 \bibitem{Mehta}
M. L. Mehta,  \textit{Random Matrices}, 2nd Edition, Academic Press, San
Diego, CA, 1991.
\bibitem{GaudinMehta}
M. L. Mehta and M. Gaudin, ``On the density of eigenvalues of a
random matrix,'' \textit{Nuclear Phys.}, \textbf{18}, 420--427, 1960.
\bibitem{AAA} P. D. Miller, \textit{Applied Asymptotic Analysis},
  Volume 75, Graduate Studies in Mathematics, American Mathematical
  Society, Providence, RI, 2006.
\bibitem{PS97} L. Pastur and M. Shcherbina, ``Universality of the local
eigenvalue statistics for a class of unitary invariant random matrix
ensembles,'' \textit{J. Statist. Phys.}, \textbf{86}, 109--147, 1997.
\bibitem{PS03} L. Pastur and M. Shcherbina, ``On the edge universality
  of the local eigenvalue statistics of matrix models,''
  \textit{Mat. Fiz.  Anal. Geom.}, \textbf{10}, 335--365, 2003.
\bibitem{PlaRot} M. Plancherel and W. Rotach.  ``Sur les valeurs
  asymptotiques des polynomes d'Hermite $H_{n}(x) = (-1)^{n}
  e^{x^{2}/2} d^{n}(e^{-x^{2}/2})/dx^{n}$,''
  \textit{Comment. Math. Helv.}, \textbf{1}, 227--254, 1929.
\bibitem{TracyWidom} C. A. Tracy and H. Widom, 
``Level Spacing Distributions and the Airy Kernel,'' 
\textit{Commun. Math. Phys.}, \textbf{159}, 151--174, 1994.
\end{thebibliography}
\end{document}